\documentclass[12pt]{amsart}
\usepackage{amssymb} 
\usepackage[T1]{fontenc}
\DeclareMathAlphabet{\mathantt}{OT1}{antt}{li}{it}
\DeclareMathAlphabet{\mathpzc}{OT1}{pzc}{m}{it}
\usepackage[normalem]{ulem} 
\usepackage{tikz-cd}
\usepackage{tikz}
\usepackage[colorlinks=true,linkcolor=blue,urlcolor=blue,citecolor=blue, pagebackref]{hyperref}
\usepackage{soul} 
\usepackage[margin=1.0in]{geometry}
\usepackage{verbatim, amscd,color} 
\usepackage{mathrsfs}
\usepackage{tikz-cd}

\usepackage{enumitem}

\usepackage{eucal}

\usepackage{courier}

\usepackage{url} %

\numberwithin{equation}{section}

\newcommand{\vv}{\vec}

\newcommand{\mss}{\mathscr}

\newcommand\starr{\bullet}

\newcommand{\rnc}{\renewcommand}

\newcommand\op{\operatorname}

\rnc{\c}{\mathcal}
\newcommand{\g}{\mathfrak g}

\newcommand\mfrfm{\mathfrak{R}^{\operatorname{FM}}}

\newcommand{\ord}{\operatorname{ord}}

\newcommand{\nc}{\newcommand}
\nc{\Z}{\mathbb{Z}}
\nc{\Q}{\mathbb{Q}}
\nc{\C}{\mathbb{C}}
\nc{\R}{\mathbb{R}}
\nc{\A}{\mathbb{A}}
\nc{\G}{\mathbb{G}}
\nc{\N}{\mathbb{N}}
\nc{\V}{\mathbb{V}}
\nc{\dc}{\Delta}

\newcommand{\ch}{\chi}

\nc{\ms}{\mathscr}
\nc{\mr}{\mathrm}
\nc{\mc}{\mathcal}
\nc{\mt}{\mathtt}
\nc{\msf}{\mathsf}
\nc{\mf}{\mathfrak}
\nc{\mb}{\mathbb}

\nc{\wtid}{\widetilde{\mathcal{D}}}
\nc{\wtic}{\widetilde{\mathcal{C}}}
\nc{\wh}{\widehat}
\nc{\ti}{\widetilde}

\nc{\rat}{\mt{rat}}

\nc{\la}{\lambda}

\nc{\msp}{\ms{P}}

\nc{\ve}{\varepsilon}
\nc{\vpi}{\varpi}

\nc{\sm}{\setminus}

\nc{\dbc}{\mr{D}^b_c}
\nc{\dpl}{\mr{D}^+}
\nc{\mrd}{\mr{D}}
\nc{\db}{\mr{D}^b}

\nc{\msm}{\ms{M}}

\nc{\msk}{\msf{k}}

\nc{\h}{\mbox{-}}

\nc{\bs}{\backslash}
\nc{\ov}{\overline}

\nc{\rank}{\mr{rank}}

\nc{\fil}{\mr{Fil}}

\nc{\Hom}{\operatorname{Hom}}

\nc{\w}{\widetilde}

\nc{\fm}{\mf{FM}}

\nc{\dq}{\Delta_q}
\nc{\nq}{\nabla_q}
\nc{\tq}{T_q}
\rnc{\lq}{L_q}

\newcommand{\FM}{{FM}}

\newcommand\DD{\delta}

\newcommand\Ext{\operatorname{Ext}}

\newcommand\tensor{\otimes}
\newcommand\ml{\mathcal{L}}

\newcommand{\Qk}{Q_\kappa}

\newcommand{\leto}[1]{\stackrel{#1}{\to}}

\rnc{\c}{\mathcal}

\nc{\tw}{\widetilde}

\nc{\uq}{\mathbf{U}_q}
\nc{\uqg}{\mathbf{U}_q(\mathfrak{g})}
\nc{\ua}{\mathbf{U}_A}
\nc{\Mod}{\mathbf{Mod}}
\rnc{\mf}{\mathfrak}
\nc{\mcl}{\mathcal}

\newcommand{\ExpM}{\operatorname{\mr{D}}}

\newtheorem{theorem}[equation]{Theorem}

\newtheorem{corollary}[equation]{Corollary}
\newtheorem{question}[equation]{Question}
\newtheorem{proposition}[equation]{Proposition}

\newtheorem{lemma}[equation]{Lemma}

\newtheorem{conj}[equation]{Conjecture}

\newcommand{\slg}{\mathfrak{sl}}

\newcommand{\Inv}{\operatorname{Inv}}

\newcommand{\Jop}{\mathcal J_\kappa}
\newcommand{\Sop}{\mathcal S}
\newcommand{\Fop}{\mathcal F_\kappa}
\newcommand{\Aop}{\mathcal A_\kappa}

\theoremstyle{remark}

\newtheorem{remark}[equation]{Remark}
\newtheorem{example}[equation]{\bf Example}

\theoremstyle{definition}

\newtheorem{defi}[equation]{Definition}
\newtheorem{ass}[equation]{Assumption}
\newtheorem{definition}[equation]{Definition}

\let\svthefootnote\thefootnote
\newcommand\blankfootnote[1]{%
  \let\thefootnote\relax\footnotetext{#1}%
  \let\thefootnote\svthefootnote%
}

\begin{document} 
\title{Weights of KZ motives and quantum groups}
 \author{P. Belkale and N. Fakhruddin } 
\maketitle

\blankfootnote{
P.~B.: Department of Mathematics, University of North Carolina, Chapel
Hill, NC 27599, USA.\\ 
 Email belkale@email.unc.edu
 
  N.~F.: Email naf@math.tifr.res.in
  \vspace{0.08 cm} }

\begin{abstract}
The KZ connection is a connection on the constant bundle associated to a set of $n$ finite dimensional irreducible representations of a simple Lie algebra $\mf{g}$ and a nonzero complex number $\kappa$, over the configuration space of $n$-distinct points on the affine line. Via the work of Schechtman–Varchenko and Looijenga, when $\kappa$ is rational, the associated local systems can be seen to be realisations of naturally defined motivic local systems which are called KZ motives.  

The Kazhdan--Lusztig equivalence gives an isomorphism of the local systems underlying these KZ motivic variations with certain pure braid group representations via the action in the category of the modules for the corresponding  Lusztig quantum group at roots of unity.  We formulate, and provide evidence for, conjectures relating the weight filtrations of KZ motives with the theory of quantum groups, especially properties of tilting modules and filtrations of Jantzen and Andersen.

Our methods include the use of a new deformation of the representation
ring which keeps track of the first moments of the weight filtration
as well as Schmid’s theorem on limiting mixed Hodge structures, and results about the
Hodge filtration obtained in our earlier joint work with
S.~Mukhopadhyay.

\end{abstract}

\maketitle

\tableofcontents

\section{Introduction}

Let $\mf{g}$ be a finite-dimensional simple Lie algebra over $\C$ and
for a dominant weight $\lambda \in P^+$ let $V_{\lambda}$ be the
corresponding irreducible representation.

Let $\vec{\lambda} = (\lambda_1,\dots, \lambda_n) \in (P^+)^n$ and let
$\kappa \in \C^{\times}$. The Knizhnik--Zamolodchikov (KZ) connection,
introduced in \cite{KZ}, is defined on the constant vector bundle on
$\mathcal{C}_n$, the configuration space $\mc{C}_n$ of $n$-distinct
points on the affine line, with fibre
$V(\vv{\lambda})=V_{\lambda_1}\tensor V_{\lambda_2}\tensor\dots\tensor
V_{\lambda_n}$. Using variables $z_1,z_2,\dots,z_n$ on $\mc{C}_n$, the
connection equations for flat solutions are \footnote{Here we note
  that Equation (1.5) in \cite{BFM}, and the first equation in
  \cite[Section 4]{BF-survey} are off by a sign. See Remark
  \ref{signkappa}.}:
\begin{equation}\label{e:KZ}
    \kappa\frac{\partial}{\partial z_i}f =-\biggl(\sum_{j\neq
    i}\frac{\Omega_{ij}}{z_i-z_j}\biggr)f.  
\end{equation}
Here $f$ is any local section of
$V(\vv{\lambda})\otimes\mathcal{O}_{\mathcal{C}_n}$, $\Omega_{ij}$ is
the normalised Casimir element acting on the $i$ and $j$ tensor
factors, and $\kappa\in\C^{\times}$.

Since its introduction by physicists, the monodromy of the KZ
connection has been extensively studied by mathematicians, leading in
particular to the definition of quantum groups by Drinfeld and Kohno
who considered the case of generic $\kappa$. A few years later Kazhdan
and Lusztig \cite{KL} extended these results to all nonrational
$\kappa \in \C^{\times}$.

There is a canonical decomposition
\begin{equation*}\label{decompose}
  V(\vec{\lambda})
  =\bigoplus_{\nu\in
    P^+} \op{Hom}_{\mf{g}}(V_{\nu}, V(\vec{\lambda}))\tensor V_{\nu} \, .
\end{equation*}
The multiplicity space $\op{Hom}_{\mf{g}}(V_{\nu}, V(\vec{\lambda}))$
is isomorphic to the space of coinvariants
$\mb{A}(\vec{\lambda},\nu^*)$  defined by
\begin{equation}\label{coinvariants}
  \mb{A}(\vec{\lambda},\nu^*)= (V(\vec{\lambda})\tensor
  V^*_{\nu})_{\mf{g}}=\frac{V(\vec{\lambda})\tensor V^*_{\nu}}{\mf{g}
    (V(\vec{\lambda})\tensor V^*_{\nu})}.
\end{equation}
and the KZ connection \eqref{e:KZ} induces a flat connection on the trivial
bundle on $\mc{C}_n$ with fibre $\mb{A}(\vec{\lambda},\nu^*)$, which we
also call the KZ connection.  Note that for a finite dimensional
representation $V$ of $\mf{g}$, the invariants $V^{\mf{g}}$ map
isomorphically to the coinvariants $V_{\mf{g}}$.

\bigskip

If $\kappa \in \Q^{\times}$ the KZ local systems have a motivic
structure: Using results of Schechtman--Varchenko and Looijenga
\cite{SV,L2,BBM}, the authors, together with S.~Mukhopadhyay
\cite{BFM}, constructed certain local systems of Nori motives
$\mc{KZ}_{\kappa}(\vec{\lambda},\nu^*)$ (with coefficients in the
cyclotomic field $\Q(\mu_N)$ for a suitable integer $N$) which have
the property that the Betti realisations over $\C$ are naturally
isomorphic to the dual of the local system given by the solutions of
the KZ equations on $(V(\vec{\lambda})\tensor V_{\nu}^*)_{\mf{g}}$.
Roughly speaking, $\mc{KZ}_{\kappa}(\vec{\lambda}, \nu^*)$ is a family
of motives associated to the cohomology with supports of families of
finite covers of affine spaces ramified along a certain finite set of
hyperplanes; for the precise definition see \cite[Section 3.2]{BFM} or
Section \ref{DProof}. For our present purposes, it is enough to
consider the Hodge realisations $KZ_{\kappa}(\vec{\lambda}, \nu^*)$ of
these motivic local systems, which we will call KZ variations. These
are admissible variations of mixed Hodge structure (VMHS) on
$\mc{C}_n$.

A VMHS $(\mb{V}, \mr{W}_{\bullet},\mr{F}^{\bullet})$ on a complex manifold
$X$ with coefficients in a subfield $k \subset \C$ consists of a
(finite dimensional) $k$-local $\mb{V}$ on $X$ together with an
increasing exhaustive filtration by $k$-local subsystems
$\mr{W}_{\bullet}$ on $\mb{V}$ and a decreasing exhaustive filtration
$\mr{F}^{\bullet}$ consisting of locally free subsheaves on
$\mb{V} \otimes_k \mc{O}_{X}$. These filtrations are required to
satisfy several axioms which we do not recall here, but we mention
that the induced filtrations on each fibre induce a mixed Hodge
structure on it. A VMHS has two basic numerical invariants: the Hodge
polynomial
\[
  \mr{P}(\mb{V}) = \sum_p \rank(\mr{F}^p/ \mr{F}^{p+1} ) t^p \in
  \Z[t,t^{-1}] 
\]
and the weight polynomial
\[
  \mr{Wt(\mb{V})} = \sum_i \rank(\mr{W}_i/\mr{W}_{i-1}) w^i \in
  \Z[w,w^{-1}] .
\]
(In general these are Laurent polynomials, but in our setting they
will be ordinary polynomials.)

In \cite{BFM} we proved that that the structure of the Hodge
polynomials of all KZ variations is governed by a certain commutative
$\Z[t]$-algebra $\mf{R}_{\kappa}(\mf{g})$, defined using the Hodge
polynomials for KZ motives when $n=2$. This is a $\Z[t]$-algebra which
is free as a $\Z[t]$-module with basis
$\{[\lambda]\}_{\lambda \in P^+}$ which is a deformation of the usual
representation ring $\mf{R}(\mf{g})$ of $\mf{g}$ in the sense that
setting $t=1$ in $\mf{R}_{\kappa}(\mf{g})$ recovers
$\mf{R}(\mf{g})$. For $\mf{g} = \mf{sl}_n$ we also gave an algorithm
to compute all products in $\mf{R}_{\kappa}(\mf{g})$.

The goal of the present article is to investigate the weight
filtration of KZ local systems and in particular prove properties of
their weight polynomials. There is no ring analogous to
$\mf{R}_{\kappa}(\mf{g})$ which determines all weight polynomials,
instead we will formulate general conjectures which relate the weight
filtration to the structure of tensor products of modules over quantum
groups.  Although we are still far from proving the full conjectures,
we prove several implications and special cases and also provide some
numerical evidence in support of these conjectures.

Aside from the representation theory of quantum groups, one of our key
methods is the use of another deformation $\mfrfm_{\kappa}(\mf{g})$ of
the representation ring $\mf{R}(\mf{g})$, which we call the first
moment weight deformation. As a module it is free over the dual number
ring $\Z[\epsilon]/(\epsilon^2)$ with basis
$\{[\lambda]\}_{\lambda \in P^+}$, and its algebra structure recovers
$\mf{R}(\mf{g})$ by setting $\epsilon = 0$. The structure constants of
this ring can be defined using the derivative at $1$ of the weight
polynomials of KZ motives when $n=2$ or the derivative at $1$ of the
Hodge polynomials of KZ motives when $n=2$ (for $\kappa$ and its
``conjugate''). The dual nature of this ring provides a basic link
which allows us to transfer results about Hodge polynomials (derived
from refining the basic results in \cite{BFM}) to results about the
weight polynomials.

\smallskip

We now proceed to explain our conjectures and results in more detail.

\subsection{The KZ connection and quantum groups}

For the rest of the introduction, in order to avoid technicalities, we
assume that $\mf{g} = \mf{sl}_m$ for some $m \geq 2$. The results that
we prove in this article apply only to such $\mf{g}$, but our
conjectures apply to all simple $\mf{g}$ and we formulate them as such
in the main body of this article.

From now on we also assume that $\kappa \in \Q_{>0}$.  This does not
cause any loss of generality as far as the monodromy of the KZ
connection is concerned as there is a simple relation between the
cases of $\kappa$ and $-\kappa$: since the Casimir operator is
self-dual, it follows from the definition of the dual connection that
the monodromy of the KZ connection for
$\vec{\lambda} =(\lambda_1,\dots,\lambda_n)$ and $\kappa$ is dual to
that for $(\vec{\lambda})^* := (\lambda_1^*,\dots,\lambda_n^*)$ with
$\kappa$ replaced by $-\kappa$.  Thus, for $\kappa \in \Q_{> 0}$ the
monodromy of the KZ connection on the dual of
$(V(\vec{\lambda})\tensor V_{\nu}^*)_{\mf{g}}$---this is the setting
of \cite{BFM}---is equal to the monodromy of the KZ connection with
parameter $-\kappa$ on the space
$(V(\vec{\lambda})\tensor V_{\nu}^*)_{\mf{g}}$. We also note that the
variations $KZ_{\kappa}(\vec{\lambda}, \nu^*)$ and
$KZ_{-\kappa}(\vec{\lambda}, \nu^*)$ are related by Verdier duality
\ref{signkappa} so there is no loss of generality in considering only
$KZ_{\kappa}(\vec{\lambda}, \nu^*)$ for positive $\kappa$.

In their series of articles \cite{KL}, Kazhdan and Lusztig introduced
a category of modules $\mss{O}_{-\kappa}$ for the affine Lie algebra
$\hat{\mf{g}}$ associated to $\mf{g}$. They also defined a braided
monoidal structure on this category and related it to the monodromy of
the KZ connections. Furthermore, they introduced a braided monoidal
structure on the category $\uqg$-$\Mod$ of finite-dimensional
representations (of type $1$) of the Lusztig quantum group $\uqg$ for
$\mf{g}$ and the root of unity $q = \exp(-i \pi/\kappa)$ and proved
that these categories of modules for $\hat{\mf{g}}$ and $\uqg$ are
equivalent as braided monoidal categories. In particular, it follows
from their results that the monodromy of the KZ connections is
determined purely algebraically from the braiding on either of these
categories. Since the weight filtration on the KZ variations is a
filtration by local systems one might hope that this can also be
determined purely algebraically from these categories and this is the
basic idea that motivated our conjectures.

Given the equivalence of categories, our conjectures could in
principle be formulated in terms of either category, but since some
results that we need, primarily concerning the Andersen filtration
\cite{and-filt} for tilting modules, have only been formulated and
proved in the quantum group setting, we have chosen to work
exclusively with $\uqg$-$\Mod$. However, we expect that working with
the Kazhdan--Lusztig category $\mss{O}_{-\kappa}$ will be essential
for proving the conjectures in full generality.

\subsection{Conjectures}

\subsubsection{A conjecture for simple Weyl modules}
Assume that $q$ is a complex root of unity such that the order of
$q^2$ is $l >1$.  The category $\uq$-$\Mod$ of finite dimensional
modules (of type 1) over $\uq$ (we suppress the $\mf{g}$ from now on
since it is fixed) has properties analogous to the category of finite
dimensional representations of $SL_m$ over a field of characteristic
$p$. In particular, for each dominant weight $\lambda$ there is a Weyl
module $\dq (\lambda)$ whose dimension is equal to that of the simple
$\mf{g}$-module with highest weight $\lambda$. In general,
$\dq(\lambda)$ need not be simple, but we recall that it is simple for
certain $\lambda$: If $l > m$ then this is the case for all $\lambda $
of level $\ell = l -m$, i.e., for those $\lambda$ which can be written
as a sum of at most $\ell$ fundamental weights. This is also the case
for $\lambda$ of the form $ l\lambda ' + (l-1)\rho$ for an arbitrary
dominant weights $\lambda'$, and $\rho$, as usual, being the half-sum
of the positive roots. In general, whenever $\dq(\lambda)$ is simple,
it is also a tilting module.

Now suppose $\vec{\lambda} = (\lambda_1,\lambda_2,\dots,\lambda_n)$ is
such that each $\dq(\lambda_i)$ is simple and therefore also
tilting. A fundamental property of tilting modules is that a tensor
product of tilting modules is also tilting, so
$ \dq(\vec{\lambda}) := \otimes_i \dq(\lambda_i)$ is a tilting
module. For any dominant weight $\nu$, the braiding on the category
$\uq$-$\Mod$ makes the vector space
$\Hom_{\uq}(\dq(\vec{\lambda}), \nq (\nu))$ (where $\nq (\nu)$ denotes
the dual Weyl module with highest weight $\nu$) into a representation
of the pure braid group on $n$-strands, i.e., $\pi_1(\mc{C}_n)$, and
the results of Kazhdan--Lusztig \cite{KL} imply that this
representation is the same as the monodromy representation of the KZ
connection on the dual of
$(V(\vec{\lambda})\tensor V_{\nu}^*)_{\mf{g}}$ when
$q = \exp(-i\pi/\kappa)$.

For an arbitrary tilting $\uq$-module $Q$ and a dominant weight $\nu$,
Andersen has defined \cite{and-filt} a decreasing, separated and
functorial filtration $\Hom_{\uq}(\dq(\nu),Q)^{\bullet} $ on
$\Hom_{\uq}(\dq(\nu),Q)$ with $\Hom_{\uq}(\dq(\nu),Q)^0$ =
$\Hom_{\uq}(\dq(\nu),Q)$.  Taking $Q$ to be $D(\dq(\vec{\lambda}))$,
where $D$ denotes the duality functor in $\uq$-$\Mod$ (which preserves
tilting modules), by duality we get an increasing
filtration\footnote{We index the dual filtration so that the layers
  for the same index $i$ are dual. So if $E^{\bullet}$ is a finite
  dimensional vector space with a decreasing filtration then we define
  an increasing filtration $E^*_{\bullet}$ on $E^*$ by
  $E^*_{i} = \ker(E^* \to (E^{i+1})^*)$.}
$\Hom_{\uq}(\dq(\vec{\lambda}), \nq (\nu))^{\bullet}$ on
$\Hom_{\uq}(\dq(\vec{\lambda}), \nq (\nu))$ by
$\pi_1(\mc{C}_n)$-subrepresentations (since the Andersen filtration is
functorial) which we call the dual Andersen filtration.

\smallskip

If $(V(\vec{\lambda})\tensor V_{\nu}^*)_{\mf{g}}\neq 0$ then
$(\sum_i \lambda_i) - \nu$ is a sum of a simple roots. We let
$\ExpM(\vec{\lambda},\nu^*)$ denote the non-negative integer
$s((\sum_i \lambda_i) - \nu)$, where $s:P \to \Q$ is the linear
function defined by $s(\alpha_i) =1 $ for each simple root $\alpha_i$.
We will sometimes call $\ExpM(\vec{\lambda},\nu^*)$ the ``expected
weight'' of $KZ_{\kappa}(\vec{\lambda}, \nu^*)$ and abbreviate it to
$\ExpM$ to simplify notation when its arguments the function are clear
from the context.

We then have:
\begin{conj}\label{c:tilt}
  Let $q = \exp(-i\pi/\kappa)$ and assume that $q^2$ has order
  $l > 1$.  For $\vec{\lambda}$ such that all $\dq(\lambda_i)$ are
  simple and an arbitrary dominant weight $\nu$, the weight filtration
  of $KZ_{\kappa}(\vec{\lambda}, \nu)$, when transported to
  $\Hom_{\uq}(\dq(\vec{\lambda}), \nq (\nu))$ via the Kazhdan--Lusztig
  equivalence, is equal to the dual Andersen filtration on
  $\Hom_{\uq}(\dq(\vec{\lambda}), \nq (\nu))^{\bullet}$ shifted by
  $\ExpM(\vec{\lambda},\nu^*)$.
\end{conj}
In particular, the conjecture implies that all weights of
$KZ_{\kappa}(\vec{\lambda}, \nu^*)$ are greater than or equal to
$\ExpM(\vec{\lambda},\nu^*)$. It is expected that the ranks of the
layers in the Andersen filtration can be computed in terms of
parabolic Kazhdan--Lusztig polynomials, so the conjecture leads to a
precise numerical prediction for these ranks.  We will prove several
results which provide convincing evidence for this conjecture, but
before describing these results we discuss a more general conjecture
for arbitrary $\vec{\lambda}$.

\subsubsection{A conjecture for arbitrary Weyl modules}

It is possible to extend the Andersen filtration to arbitrary
$\uq$-modules $M$, see Definition \ref{d:filt}, but substituting this
filtration in Conjecture \ref{c:tilt} for arbitrary $\vec{\lambda}$
does not give a correct description of the weight filtration. In fact,
it may happen that $\dq(\vec{\lambda})$ is a tilting module even when
some $\dq(\lambda_i)$ is not simple and in such cases the weight
filtration need not satisfy the condition that all weights are at
least $\ExpM(\vec{\lambda},\nu^*)$. In order to formulate a more
reasonable conjecture one must also consider a filtration on all the
$\dq(\lambda_i)$ which are not simple. We formulate the conjecture
axiomatically by giving the properties we require of this filtration,
see Assumption \ref{a:filtweyl}. It is expected that the radical and
Jantzen filtrations both satisfy this assumption (and are in fact
equal).

\smallskip

The precise statement of our extension of Conjecture \ref{c:tilt} is
somewhat complicated so we refer the reader to Conjecture
\ref{c:simple} in the main text for this, but we describe here the
basic constructions and then give an informal statement.

The main ingredients are:
\begin{itemize}
\item For each $\lambda_i$ we consider a suitable, e.g., radical,
  filtration $\dq(\lambda_i)^{\bullet}$ of $\dq(\lambda_i)$ and the
  tensor product of these filtrations $\dq(\vec{\lambda})^{\bullet}$
  on $\dq(\vec{\lambda})$.
\item We extend the Andersen filtration for tilting modules to a
  filtration on $\Hom_{\uq}(\dq(\nu), M)$ for $M$ an arbitrary
  $\uq$-module.  We apply this to the duals of the layers of the
  filtration $\dq(\vec{\lambda})^{\bullet}$ which are sums of tensor
  products of simple modules.
\end{itemize}

We then consider the $E_1$-spectral sequence computing
$\Ext^*_{\uq}(\dq(\vec{\lambda}), \nq(\nu))$ corresponding to the
filtration $\dq(\vec{\lambda})^{\bullet}$. The filtrations above
induce filtrations on
$\Hom_{\uq}(\mr{Gr}^j(\dq(\vec{\lambda})), \nq(\nu))$ for each
$j$. This is a term on the $E_1$-page of the spectral sequence and the
relevant $E_{\infty}$-term is a subspace of
$\Hom_{\uq}(\mr{Gr}^j(\dq(\vec{\lambda})),\nq(\nu))$, so it acquires
an induced filtration.  Our conjecture then roughly claims that these
induced filtrations determine the weight filtration on
$KZ_{\kappa}(\vec{\lambda}, \nu^*)$.

In general, we do not know how to compute explicitly the
$E_{\infty}$-terms of the spectral sequence so the conjecture does not
immediately lead to explicit formulae for the ranks of the layers of
the weight filtration. However, there are some cases where such
formulae can be obtained.
\begin{itemize}
\item If some $ \lambda_i = l \lambda_i' + (l-1)\rho$ for a dominant
  $\lambda_i'$ then $\dq(\lambda_i)$ is projective. All layers of the
  filtration $\dq(\vec{\lambda})^{\bullet}$ are then also projective,
  and the $\Ext$-spectral sequence degenerates at $E_1$. Furthermore,
  each layer of $\dq(\vec{\lambda})^{\bullet}$ is also tilting, so the
  generalised Andersen filtration equals the original Andersen
  filtration. It follows that the conjecture leads to an explicit
  prediction in terms of Kazhdan--Lusztig polynomials for the ranks of
  the layers of the weight filtration.
\item If $\mf{g} = \mf{sl}_2$ then the $\Ext$-spectral sequence
  degenerates at $E_2$ and the $E_2 = E_{\infty}$ terms can be
  computed explicitly as filtered modules. It follows that in this
  case we get explicit predictions for the ranks of the layers of the
  weight filtration of $KZ_{\kappa}(\vec{\lambda},\nu^*)$ for arbitrary
  $\vec{\lambda}$ and $\nu$.
\end{itemize}
We provide strong evidence for this conjecture for $\mf{sl}_2$ for
arbitrary $\vec{\lambda}$ and $\nu$ and also evidence for all
$\mf{sl}_m$ when some $\dq(\lambda_i)$ or $\dq(\nu)$ is projective (so
the $\Ext$ spectral sequence degenerates).

\subsection{Results}

\subsubsection{The first moment weight deformation ring}

We now describe the main results we prove in this article in support
of our conjectures as well as a basic result on the structure of the
ring $\fm_{\kappa}$, beginning with the latter.

We first recall the definition of $\mf{R}_{\kappa}(\mf{g})$. It is
defined in \cite[Section 1]{BFM} as the free $\mathbb Z[t]$-module
with basis $[\lambda]$ indexed by dominant integral weights and the
product given by
\begin{equation}\label{e:BFM-product}
[\lambda]\star_\kappa[\mu]
=\sum_{\nu\in P^+}P_{\kappa}^{\lambda,\mu,\nu^*}(t)[\nu],
\end{equation}
where
\[
P_\kappa^{\lambda,\mu,\nu^*}(t)
=P\bigl(KZ_{\kappa}((\lambda,\mu),\nu^*)\bigr)
\]
is the Hodge polynomial of $KZ_{\kappa}((\lambda,\mu),\nu^*)$. By
\cite[Theorem 1.4]{BFM}, $\mf{R}_{\kappa}(\mf{g})$ is a unital
commutative associative $\mathbb Z[t]$-algebra such that
specialization at $t=1$ gives the ordinary representation ring:
\begin{equation}\label{e:specialize}
P_\kappa^{\lambda,\mu,\nu^*}(1)=c_{\lambda\mu}^{\nu},
\end{equation}
where $c_{\lambda\mu}^{\nu}$ is the usual tensor product multiplicity.
\begin{definition}\label{d:fmw}
  For a VMHS $(\mb{V},\mr{W}_{\bullet},\mr{F}^{\bullet})$ let
  \[
    FM^{\mr{W}}(\mb{V}) = \sum_a a\cdot\rank(\mr{W}_a/\mr{W}_{a-1}) = \mr{Wt}(\mb{V})'(1)
    \in \Z.
  \]
 
  We call this integer the \emph{first moment} of the weight
  filtration.
\end{definition}
We then define a new ring as follows:
\begin{definition}\label{d:fmring}
For any $\vec{\lambda}$ and $\nu$ let
\[
  FM^\mr{W}_{\kappa}(\vec{\lambda},\nu^*) =
  FM^{\mr{W}}(KZ_{\kappa}(\vec{\lambda},\nu^*))
\]
and define the \emph{first moment weight deformation ring}
$\mfrfm_{\kappa}(\mf{g})$ to be the free
$\Z[\epsilon]/(\epsilon^2)$-module with basis $[\lambda]$ indexed by
dominant integral weights and product given by
\begin{equation}\label{e:fmring}
  [\lambda] \diamond_{\kappa}[\mu] 
  =\sum_\nu\Bigl(c_{\lambda\mu}^{\nu}
+\epsilon FM^{\mr{W}}_\kappa((\lambda,\mu),\nu^*)\Bigr)[\nu].
\end{equation}
\end{definition}
The associativity of this product is nontrivial, but follows from the
results of \cite{BFM}; see Section \ref{s:assoc} for the proof. It is clear
from the definition that setting $\epsilon = 0$ in
$\mfrfm_{\kappa}(\mf{g})$ recovers $\mf{R}(\mf{g})$.

\subsubsection{The general first-moment formula}

We extend the definition of the Littlewood--Richardson coefficients
$c_{\lambda,\mu}^{\nu}$ to allow one of the weights to be
non-dominant. For an arbitrary integral weight $\xi$, say that $\xi$
is \emph{dot-singular} if $\xi+\rho$ lies on a finite Weyl wall, i.e.,
if $ \langle\xi+\rho,\alpha^\vee\rangle=0$ for some root $\alpha$.  We
say that $\xi$ is \emph{dot-regular} if it is not dot-singular.
\begin{itemize}
\item Suppose $\mu$ and $\nu$ are dominant and $\xi$ is a possibly
  non-dominant weight. Then let
  \begin{equation*}
    {c}_{\xi,\mu}^{\nu}={c}_{\mu,\xi}^{\nu}
=
\begin{cases}
0,&\xi+\rho\text{ is dot-singular},\\
(-1)^{\ell(w)}c_{\eta,\mu}^{\nu},
 &\xi+\rho\text{ is dot-regular and }\eta=w\cdot\xi\in P^+.
\end{cases}
\end{equation*}
\item Suppose $\lambda$ and $\mu$ are dominant and $\xi$ is a possibly
  non-dominant weight. Then let
\begin{equation*}%
 {c}_{\lambda,\mu}^{\xi}
=
\begin{cases}
0,&\xi+\rho\text{ is dot-singular},\\
(-1)^{\ell(w)}c_{\lambda\mu}^{\eta},
 &\xi+\rho\text{ is dot-regular and }\eta=w\cdot\xi\in P^+.
\end{cases}
\end{equation*}
\end{itemize}

As we noted earlier, in \cite[Section 11]{BFM} an inductive algorithm
was given for computing all products in $\mf{R}_{\kappa}(\mf{g})$
(with $\mf{g} = \mf{sl}_m$). For the rings $\mfrfm_{\kappa}(\mf{g})$ we
give an explicit formula for the product of pairs of arbitrary basis
elements. For this recall that for dominant $\lambda,\mu,\nu$, ${\ExpM}((\lambda,\mu),\nu^*)$ is defined to be
$s(\lambda +\mu - \nu)$.
\begin{theorem}\label{t:G}
  Let $\mathfrak{g}=\mf{sl}_m$ and let $\kappa\ge2$ be an
  integer.\footnote{The assumption that $\kappa$ be an integer does
    not lead to any loss of generality, see Remark \ref{r:kappa}.}
  Then for dominant integral weights $\lambda,\mu,\nu$, we have
\begin{align}
FM^\mr{W}_\kappa((\lambda,\mu),\nu^*)
={}&{c_{\lambda\mu}^{\nu}\ExpM}((\lambda,\mu),\nu^*)\notag\\
&+\sum_{\alpha\in R^+}
 \sum_{\substack{r\ge1\\r\kappa<\langle\lambda+\rho,\alpha^\vee\rangle}}
  {c}_{\lambda-r\kappa\alpha,\mu}^{\nu}\notag\\
&+\sum_{\alpha\in R^+}
 \sum_{\substack{r\ge1\\r\kappa<\langle\mu+\rho,\alpha^\vee\rangle}}
  {c}_{\lambda,\mu-r\kappa\alpha}^{\nu}\notag\\
&-\sum_{\alpha\in R^+}
 \sum_{\substack{t\in\mathbb Z\\
 t<0\ \mathrm{or}\ t\kappa>\langle\nu+\rho,\alpha^\vee\rangle}}
  {c}_{\lambda,\mu}^{\nu-t\kappa\alpha}.
\label{e:G}
\end{align}
\end{theorem}
This is proved in Section \ref{s:fm}. A similar formula for
$FM_{\kappa}^{\mr{W}}(\vec{\lambda},\nu^*)$ for arbitrary
$\vec{\lambda}$ is given in Lemma \ref{l:fmwn}.

For $\mf{g} = \mf{sl}_2$, since $KZ_{\kappa}((\lambda,\mu),\nu^*)$
has rank at most $1$---and for general $\mf{sl}_m$ whenever
$c_{\lambda,\mu}^{\nu} =1$---the above is equivalent to an explicit
formula for the weight. For $\mf{sl}_2$ we make the formula in the
theorem even more explicit in Proposition \ref{p:A1-floor}.

\subsubsection{Evidence for the quantum group conjectures}

Conjecture \ref{c:tilt} gives an explicit prediction for all the
weights of $KZ_{\kappa}(\vec{\lambda},\nu^*)$ when the
$\dq(\lambda_i)$ are simple so, in particular, it gives a prediction
for $FM^{\mr{W}}_{\kappa}(\vec{\lambda},\nu^*)$. As the first evidence
for this conjecture we prove
\begin{theorem}\label{t:fmpred}
  The prediction following from Conjecture \ref{c:tilt} for
  $FM^{\mr{W}}_{\kappa}(\vec{\lambda},\nu^*)$ is true for $\mf{g} = \mf{sl}_m$,
  $m \geq 2$.
\end{theorem}
This is proved in an explicit form in Theorem
\ref{t:many}. 

\smallskip
As further evidence, we prove
\begin{theorem}\label{t:evid}
  $ $
  \begin{enumerate}
  \item If $\mf{g} = \mf{sl}_2$ and $\vec{\lambda}$ is such that all
    $\dq(\lambda_i)$ are simple then the weights of
    $KZ_{\kappa}(\vec{\lambda},\nu^*)$ are in the range
    $[\ExpM(\vec{\lambda},\nu^*),\ExpM(\vec{\lambda},\nu^*) +1]$
  \item For general $\mf{sl}_m$ assume that $\vec{\lambda}$ is such
    that each $\lambda_i$ is of level $\ell$, i.e., a sum of at most
    $\ell$ fundamental weights. If
    $\ell + m > (\rho + \varpi_1, \rho + \varpi_1)/2$, then the
    weights of $KZ_{\kappa}(\vec{\lambda},\nu^*)$ are at least
    $\ExpM(\vec{\lambda},\nu^*)$.
  \end{enumerate}
\end{theorem}
See Theorems \ref{E22} and \ref{purity} which also contain other
results supporting Conjecture \ref{c:tilt}.

\smallskip

For Conjecture \ref{c:simple} our main evidence is: 
\begin{theorem}\label{t:simple}
  The prediction implied by Conjecture \ref{c:simple} for
  $FM^{\mr{W}}_{\kappa}(\vec{\lambda},\nu^*)$ is true for
  $\mf{g} = \mf{sl}_2$ and all $\vec{\lambda}$, $\nu$.
\end{theorem}
See Theorem \ref{t:sl2fm}. It follows from this that Conjecture
\ref{c:simple} is true for $\mf{g} = \mf{sl}_2$ and $n=2$.

\smallskip
We also prove that an analogous statement holds for all $\mf{sl}_m$
when one of the $\lambda_i$ or $\nu$ is such that the corresponding
Weyl module is projective, see Proposition \ref{p:N_0}.

\smallskip

Finally, in Section \ref{s:compute} we give a simple algorithm based
on Conjecture \ref{c:simple} for computing the ranks of the layers of
the weight filtration for $\mf{g} = \mf{sl}_2$ and arbitrary
$\vec{\lambda}$ and $\nu$. This relies on a degeneration result,
Theorem \ref{t:sl2deg}, for the $\Ext$-spectral sequence. We do not
know any non-conjectural method for computing these numbers for
general $\vec{\lambda}$ and $\nu$.

\subsection{Methods}

The proofs in the paper rely on several distinct sets of inputs. The
first is the basic factorisation result for KZ variations proved in
\cite{BFM} and some of its consequences related to the rings
$\mf{R}_{\kappa}(\mf{g})$. A refinement of a key result from
\cite{BFM} plays an important role here and may be of independent
interest, so we discuss this briefly here; for details see Theorem
\ref{t:rational-pieri}.

In the case $\mf{g} = \mf{sl}_m$, as noted earlier an inductive
algorithm for computing all products in $\mf{R}_{\kappa}(\mf{g})$ was
given in \cite{BFM}. It was based on an explicit formula for products
of the form $[\varpi_1]\star_{\kappa} [\lambda]$, where $\lambda$ is
any dominant weight. A key new ingredient in the proof of Theorem
\ref{t:G} is an explicit formula for all products of the form
$[\varpi_i]\star_{\kappa}[\lambda]$, $1 \leq i \leq m-1$, which then
leads to an explicit formula for
$[\varpi_i]\diamond_{\kappa}[\lambda]$. For $i=1$ this formula was
derived in \cite[Section 11.2]{BFM} by an explicit geometric
computation of the corresponding motives/MHS, but for $i > 1$ we
derive these using formal arguments based on the algorithm in
\cite[Section 11.3]{BFM}. Other inputs in the proof of Theorem
\ref{t:rational-pieri} come from constructions/formulas related to the
Jantzen and Andersen sum formulas in representation theory.

We also use various results from the representation theory of quantum
groups at roots of unity. In particular, we use heavily the theory of
tilting modules and the Andersen sum formula \cite{and_sum} (with
improvements by Andersen--Kulkarni \cite{AK}) plays a key role in the
proof of Theorem \ref{t:fmpred}. While tensor products of tilting
modules may be computed algorithmically, it is not known in general
how to compute the (generalised) Andersen filtration for tensor
products of simple modules. In the case of $\mf{sl}_2$ though this is
not difficult and the proof of the degeneration of the spectral
sequence of Theorem \ref{t:sl2deg}, a key ingredient in the proof of
Theorem \ref{t:simple}, relies on such explicit computations.

The proof of Theorem \ref{t:evid} depends more directly on the
geometry underlying the construction of the KZ variations. A basic
result, Lemma \ref{l:D}, is proved using this geometry and the general
results are proved by examining the local monodromy as some points in
$\mc{C}_n$ coalesce or approach infinity using results from
\cite{BFM}. Schmid's theorem on the limiting mixed Hodge structure \cite[Theorem 6.16 and Lemma 6.4]{schmid} is then used to 
derive properties of the weight filtration from properties of the
monodromy.

\subsection{Organisation of the article}
In Section \ref{s:conj} we formulate our main conjectures relating the
weight filtration and the representation theory of quantum
groups. Then in Section \ref{s:sl2} we specialise to the case of
$\mf{sl}_2$.  We prove Theorem \ref{t:sl2deg} and explain how the
conjectures lead to an explicit prediction for the ranks of the layers
of the weight filtration.

Section \ref{s:pieri} is devoted to the proof of Theorem
\ref{t:rational-pieri} which gives the Pieri formula in
$\mf{R}_{\kappa}(\mf{sl}_m)$ for all fundamental weights. This is then
used in Section \ref{s:fm} to prove Theorem \ref{t:G} which gives an
explicit formula for the $\diamond_{\kappa}$-product. We also make
this formula more explicit in the case of $\mf{sl}_2$; this is needed
for the proof of Theorem \ref{t:simple}.

Section \ref{s:wmi} is devoted to the proof of Theorem \ref{t:fmpred}
by deriving explicit formulas for both relevant quantities and showing
that they agree.

Section \ref{s:exp} is devoted to the proof of Theorem \ref{t:evid} as
well as some other related results.

\subsection{Notation}\label{s:notations}

\subsubsection{Lie theoretic notation}\label{s:lie}

As in the introduction, let $\mf{g}$ be a split finite dimensional
simple Lie algebra over a field $\msk\subseteq \C$ and let
$\mf{h}\subset \mf{g}$ be a Cartan subalgebra. Let
$R=R^+\cup R^{-} \subseteq\mf{h}^*$ be the set of roots and
$\Delta = \{\alpha_1,\alpha_2,\dots,\alpha_{r}\} \subseteq R^+$, with
$r = \rank(\mf{h})$, be the set of simple roots corresponding to a
fixed Cartan decomposition. We let $W$ denote the Weyl group.

Let $(\ ,\ )$ be the Killing form on
$\mf{h}^*$ normalised by $(\theta,\theta)=2$, where
$\theta\in \mf{h}^*$ is the highest root. We let $\{X_i\}$, be an
orthonormal basis of $\mf{g}$ with respect to the Killing form. Let
$h^{\vee}$ be the dual Coxeter number of $\mf{g}$; this is half the
scalar by which the normalised Casimir element
$\Omega= \sum_i X_i\tensor X_i\in U(\mf{g})\tensor U(\mf{g})$ acts on
the adjoint representation of $\mf{g}$ by its image under the
multiplication map on $U(\mf{g})$. Note that $h^{\vee}=m$ for
$\mf{g}=\mf{sl}_m$.

Let $P$ be the lattice of integral weights
and $P^+ \subset P$ be the set of dominant integral weights of
$\mf{g}$.  Let $\{\varpi_1,\varpi_2,\dots,\varpi_r\} \subset P^+$ be
the set of fundamental dominant weights. For $\lambda\in P^+$, let
$V_{\lambda}$ denote the corresponding irreducible representation of
$\mf{g}$. For $\lambda\in P^+$, let
$c(\lambda)=(\lambda,\lambda+2\rho)$ where $\rho$ is half the sum of
positive roots.  Let $s:P \to \Q$ be the linear function defined by
$s(\alpha_i) = 1$ for each simple root $\alpha$. For
$\vec{\lambda} \in P^{\, n}$ and $\nu \in P$ let
$\ExpM(\vec{\lambda},\nu^*) := s((\sum_i \lambda_i) - nu)$.

Let
$$m_{\mf{g}}=\frac{(\alpha_{\ell},\alpha_{\ell})}{(\alpha_s,\alpha_s)}, $$
where $\alpha_{\ell}$ is a long root and $\alpha_{s}$ a short root be
the lacing number of $\mf{g}$. Note that $m_{\mf{g}}\in\{1,2,3\}$ and
$m_{\mf{g}}=1$ if $\mf{g}$ is simply laced. Since
$(\alpha_{\ell},\alpha_{\ell})=2$, $m_{\mf{g}}\cdot(\alpha,\alpha)$ is
an even integer for all roots $\alpha$.

We denote by $\mf{R}(\mf{g})$, the representation ring of $\mf{g}$ with basis $[\lambda]$ with $\lambda\in P^+$. There is 
a character isomorphism 
\begin{equation}\label{chara}
\operatorname{Char}:\mf{R}(\mf{g})\leto{\sim} \mathbb{Z}[P]^W,\ \ \operatorname{Char}([\lambda])=\chi_{\lambda}.
\end{equation}

\subsection{Acknowledgements} 
P.~B.~would like to thank TIFR Mumbai, ISI Bengaluru, and IHES where
parts of this work were carried out for hospitality.  P.~B.~was
supported in part by NSF grant DMS - 2302288. We thank H.~H.~Andersen,
J.~Gruber and D.~Tubbenhauer for useful communication. P.~B.~thanks
R.~Rim\'anyi for some early help with AI.

\subsection{Statement on AI use}
This project began in 2023, and we began using AI (ChatGPT 5.6) in June of 2026. The Oberwolfach report  from 2025 \cite{Ober} records some of the early work.

The following were done before AI use:
\begin{enumerate}
\item[(a)] The formulation of  Conjectures \ref{c:tilting} and \ref{c:simple}.
\item[(b)] Construction and proof of associativity of the first moment weight deformation rings.
\item[(c)] The results of Section \ref{s:sl2} and a
  large part of Section \ref{s:exp} (e.g., Theorem \ref{E22}),
  though some of our original proofs in Section \ref{s:exp} were
  different.
\end{enumerate}
Our AI use began with an effort to explicitly determine structure
constants in the $\mf{sl}_2$ first moment weight deformation rings
from a recursion (this is implied by more general results in the paper now). Next we had the AI guess and prove an explicit
general type $\mr{A}$ Pieri rule, Theorem
\ref{t:rational-pieri} for $\mf{R}_{\kappa}(\mf{g})$
(\cite{BFM} had determined them implicitly). After this the AI could
guess and prove a general formula, Theorem \ref{t:G}, for the first moment product
$\diamond_{\kappa}$ in type $\mr{A}$ after connecting it to the
conjectural expectations given in the report \cite{Ober}. We have also
used AI for copy editing and literature search.

\section{The weight filtration and quantum groups}\label{s:conj}

In this section we formulate our conjectures relating the weight
filtration of KZ variation to the structure of tensor products of
modules over Lusztig's quantum groups at a root of unity. We first
consider the special case where all the dominant weights at finite
points correspond to simple Weyl modules. A key ingredient is the
Andersen filtration for tilting modules \cite{and-filt}. For general
dominant weights Weyl modules need not be simple and one must also
consider a filtration on these modules. The general conjecture uses
both the Andersen filtration, which we extend to arbitrary modules,
and filtrations of Weyl modules.

\subsection{More notation and some definitions}\label{s:mn}

Throughout this section we fix a finite-dimensional simple complex Lie
algebra $\mathfrak{g}$ with associated root system $\Phi$ in the
Euclidean space $E$. Also fix $\kappa \in \Q_{> 0}$ and let
$q = \exp(-i\pi/m_{\mf{g}}\kappa) \in \C^{\times}$, so $q$ is a root
of unity\footnote{If $\mf{g}$ is not simply laced then the analogue of
  the Kazhdan--Lusztig equivalence does not always hold with
  $q = \exp({-i\pi/\kappa})$ for all $\kappa$ but is expected to hold
  with $q = \exp({-i\pi/m_{\mf{g}}\kappa})$, see, e.g.,
  \cite{gaitsgory-kl}.}  and we let $l$ be the order of $q^2$. We
assume throughout that $l > 1$.

For background and the basic results on quantum groups that we will
use in this section as well as further references, we refer the reader
to \cite{AST} (but see also \cite{LTV}), whose notation we follow to a
large extent.

Let $\uq = \uqg$ be Lusztig's quantum group over $\Q(q)$ associated to
$\mf{g}$ (or $\Phi$) and specialised to $q$ and let $\uq$-$\Mod$ be
the category of finite dimensional $\uq$-modules\footnote{
  All the $\uq$-modules we consider will be objects of $\uq$-$\Mod$
  and when we say that such a module is projective/injective we mean
  that it is so in the abelian category $\uq$-$\Mod$.} of type 1. By
work of Lusztig, $\uq$-$\Mod$ has the structure of a braided monoidal
category with a duality functor that we denote by $D$. Tensor products
of $\uq$-modules will always be assumed to be over $\Q(q)$ which we
will not include in the notation.

Let $P \subset E$ denote the group of integral weights and
$P^+ \subset P$ the cone of dominant integral weights. For
$\lambda \in P^+$ we have associated $\uq$-modules:
\begin{enumerate}
\item the Weyl module $\dq(\lambda)$,
\item the dual Weyl module  $\nq (\lambda)$,
\item the simple module $\lq (\lambda)$, and
\item the tilting module $\tq (\lambda)$.
\end{enumerate}
These modules are related as follows:
\begin{enumerate}
\item  $\nq (\lambda) \cong D(\dq(\lambda))$.
\item $\lq (\lambda)$ is the socle of $\nq (\lambda)$ and the head of
  $\dq(\lambda)$.
\item $\dq(\lambda)$ is a submodule of $\tq (\lambda)$ and
  $\nq (\lambda)$ is a quotient module of $\tq (\lambda)$.
\item $\tq (\lambda)$ and $\lq (\lambda)$ are self-dual.
\end{enumerate}
We note that any simple object in $\uq$-$\Mod$ is of the form
$\lq (\lambda)$ for some $\lambda \in P^+$ and any indecomposable tilting
module is of the form $\tq (\lambda)$ for some $\lambda \in
P^+$. Furthermore, the tensor product of two tilting modules is also a
tilting module and a tilting module can be decomposed uniquely as a
sum of indecomposable tilting modules.

For any root $\alpha$ we let $s_{\alpha}$ be the
  corresponding reflection (acting on $E$) and let $W$ be the Weyl
  group, i.e.,~the group generated by all the $s_{\alpha}$. Let
  $d_{\alpha} = 1$ if $\alpha$ is a short root and
  $d_{\alpha} = m_{\mf{g}}$ if $\alpha$ is a long root. For
  $1 \leq i \leq r$, let $d_i = d_{\alpha_i}$ and let
  $l_i = l/(l,d_i)$.

  We let $\ti{W}$ be the affine Weyl group associated to $\Phi$ if
  $l_i = l$ for all $1 \leq i \leq r$ and the affine Weyl group of the
  dual root system if $l_i \neq l$ for some $i$. 
    Following \cite[Section 3]{and-linkage}, we let $W_l$ be the group
    generated by the affine reflections $s_{\alpha,m}$,
    $\alpha \in R$, $m \in \Z$, given by
  \[
    s_{\alpha,m}\cdot x = s_{\alpha}\cdot x  + m l_{\alpha} \alpha
  \]
  for all $x \in E$, where the action of $W$ is the dot-action,
  i.e., $s_{\alpha}\cdot x = s_{\alpha}(x + \rho) - \rho$. As an
  abstract group, $W_{l}$ is isomorphic to $\ti{W}$.

  Let $C\subset E$ be the dominant cone. It is a fundamental domain
  for the action of $W$ on $E$. Let $A^+ \subset C$ be the fundamental
  alcove; it is a fundamental domain for the action of $\ti{W}$ on
  $E$. Let $\mcl{A}$ be the set of all alcoves in $E$ and $\mcl{A}^+$
  the set of dominant alcoves, i.e., the ones contained in $C$. We let
  $\mcl{A}^{++}$ denote the set $\{\rho + A \mid A \in \mcl{A}^+\}$.

  Let $\rho_l = \sum_i (l_i-1)\varpi_i$ be the Steinberg weight. The
  corresponding Weyl module $\dq(\rho_l)$ is simple and both
  projective and injective  \cite{Negron}. More
  generally, the set
 \[
   N_0 = \{\rho_l + \sum_i b_i l_i \varpi_i \mid b \in \Z^+\}
 \]
 consists of weights $\lambda$ with the property that $\dq(\lambda)$
 is simple, injective and projective. This is proved
 in \cite{APW} under some additional hypotheses and follows in general
 from the results in \cite{Negron}.

\subsection{The conjecture for simple Weyl modules}

\subsubsection{The Andersen filtration for tilting modules}

For any tilting module $Q$ and $\lambda \in P^+$, Andersen defines in
\cite[Section 1.3]{and-filt} a decreasing and separated filtration
(which we denote by) $\Hom_{\uq}(\dq(\lambda),Q)^{\bullet}$ on
$\Hom_{\uq}(\dq(\lambda),Q)$. Since we will need the fact that the
Andersen filtration functorial in $Q$ (which is not explicitly stated
by Andersen) we briefly explain the construction. Note that most of
Andersen's article is written in the setting of algebraic groups, but
it is explained in \cite[Section 4]{and-filt} how the constructions
can be adapted to the quantum group setting.

We let $A$ be the localisation of $\mathbb{Q}[v]$ at the ideal
generated by $\phi_{q}$, the minimal polynomial of $q$.
The quantum groups and modules we have discussed above are defined
over the cyclotomic field $k = \Q(q) = A/(\phi_{q})$, but they can all
be lifted to $A$, i.e., we have a quantum group $\mathbf{U}_A$ and
$\mathbf{U}_A$-modules $\Delta_A(\lambda)$, $\nabla_A(\lambda)$ and
$T_A(\lambda)$ which reduce to the usual objects modulo $\phi_{q}$;
for the construction of the $T_A(\lambda)$ without any assumption on
$l$ the reader may consult \cite{fiebig}. (We recall that tilting
modules for $\mathbf{U}_A$ are defined, as usual, to be modules which
have both a Weyl and dual Weyl filtration.) Furthermore, the
classification of tilting $\uq$-modules implies that any such module
$Q$ has a unique lift to a tilting module $Q_A$ over $\mathbf{U}_A$.

\begin{lemma}\label{l:tiltlift}
  Any map of $\uq$-tilting modules $f: Q \to Q'$ lifts
  (noncanonically) to a map $f_A: Q_A \to Q'_A$ of
  $\ua$-tilting modules.
\end{lemma}
\begin{proof}
  For any $\uq$-module $M$ with a Weyl filtration let
  $(M:\dq(\lambda))$ be the multiplicity of $\dq(\lambda)$ in a
  Weyl filtration of $M$ and define $(M:\nq (\lambda))$ similarly for
  modules with dual Weyl filtrations. It then follows from
  \cite[(18)]{AST} that
  \[
    \dim(\Hom_{\uq}(Q,Q')) = \sum_{\lambda \in
      P^+}(Q:\dq(\lambda))\cdot (Q':\nq (\lambda)) .
  \]
  Let $K$ be the fraction field of $A$ and we denote objects over $K$
  obtained by by base change from objects over $A$ by using the
  subscript $K$. It then follows from \cite{fiebig}Lemma 4.4] that
  \begin{equation}\label{e:equality}
    \dim(\Hom_{\mathbf{U}_K}(Q_K,Q'_K)) = \dim(\Hom_{\uq}(Q,Q')) .
  \end{equation}
  Both $Q_A$ and $Q_A'$ are finite free as $A$-modules, so
  $\Hom_{\mathbf{U}_A}(Q_A,Q'_A)$ is the intersection of
  $\Hom_{\mathbf{U}_K}(Q_K,Q'_K)$ and $\Hom_A(Q_K,Q'_K)$ in
  $\Hom_{K}(Q_K,Q'_K)$. Since $A$ is a DVR, it follows that
  $\Hom_{\mathbf{U}_A}(Q_A,Q'_A)$ is an $A$-lattice inside
  $\Hom_{\mathbf{U}_K}(Q_K,Q'_K)$ and it is a saturated submodule of
  $\Hom_A(Q_K,Q'_K)$. This, implies that the image of the reduction map
  $\Hom_{\mathbf{U}_A}(Q_A,Q'_A) \to \Hom_{\uq}(Q,Q')$ has dimension
  equal to the rank of $\Hom_{\mathbf{U}_K}(Q_K,Q'_K)$.  The
  equality \eqref{e:equality} then implies that this reduction map is
  surjective.
\end{proof}

For each $\nu \in P^+$ we fix a generator
$c_{\nu} \in
\Hom_{\mathbf{U}_A}(\Delta_A(\nu),\nabla_A(\nu)) \cong A$ and
for a $\uq$-tilting module $Q$ we set
\[
  F_{\nu}(Q) = \Hom_{\mathbf{U}_A}(\Delta_A(\nu),Q_A) .
\]
We have
\[
  \bar{F}_{\nu}(Q) := F_{\nu}(Q)\otimes_A k \simeq
  \Hom_{\uq}(\dq (\nu),Q) .
\]
We define a filtration $F_{\nu}(Q)^{\bullet}$ of
$F_{\nu}(Q)$
by setting
\[
  F_{\nu}(Q)^j = \{\beta \in F_{\nu}(Q) \ | \ \gamma \circ   \beta 
   \in \phi_{l}^jAc_{\nu} \mbox{ for all } \gamma \in
  \Hom_{\mathbf{U}_A}(Q_A,\nabla_A(\nu))\} 
\]
for $j \geq 0$.  We denote by $\bar{F}_{\nu}(Q)^j$ the subspace of
$\bar{F}_{\nu}(Q)$ given by the image of $F_{\nu}(Q)^j$ in
$\bar{F}_{\nu}(Q)$ under the reduction map.
\begin{lemma}\label{l:func}
  If $f:Q \to Q'$ is a map of $\uq$-tilting modules then
  $f(\bar{F}_{\nu}(Q)^j) \subset \bar{F}_{\nu}(Q')^j$ for all
  $j \geq 0$. In particular, any automorphism of $Q$ induces a
  canonical automorphism of $\bar{F}_{\nu}(Q)^{\bullet}$ as a filtered
  vector space.
\end{lemma}
\begin{proof}
  Let $f_A:Q_A \to Q'_A$ be a lift of $f$ (which exists by Lemma
  \ref{l:tiltlift}) and let $\beta \in F_{\nu}(Q)^j$, so
  $ \gamma \circ \beta \in \phi_{q}^jAc_{\nu}$ for all
  $\gamma \in \Hom_{\mathbf{U}_A}(Q_A,\nabla_A(\nu))$. Now for any
  $\gamma' \in \Hom_{\mathbf{U}_A}(Q'_A,\nabla_A(\nu))$ we have
  $ \gamma \circ f_A \in \Hom_{\mathbf{U}_A}(Q_A,\nabla_A(\nu))$, so
  \[
    \gamma' \circ (f_A \circ \beta) = (\gamma' \circ f_A) \circ \beta
    \in \phi_q^jAc_{\nu} ,
  \]
  hence $f_A \circ \beta \in F_{\nu}(Q')^j$.
  
  The first part of the lemma follows by reducing this inclusion
  modulo $\phi_{q}$ and the statement about automorphisms is then
  obvious.
\end{proof}

\subsubsection{The Andersen filtration and Kazhdan--Lusztig
  polynomials}

The Andersen filtration gives rise to a polynomial
$f_{\nu}(Q) \in \mathbb{Z}[w]$ with non-negative coefficients which
encodes the dimensions of the layers of the filtration
\cite[1.5]{and-filt}. We now formulate a conjecture, essentially due
to Andersen, describing these polynomials when $Q = \tq (\lambda)$ for
$\lambda \in P^+$ in terms of parabolic Kazhdan--Lusztig polynomials
for the affine Weyl group $\ti{W}$. These are elements
$n_{A,B} \in \Z[w]$ with non-negative coefficients parametrised by
pairs $A,B$ of dominant alcoves. For the construction and properties
of these polynomials we refer the reader to \cite{soergel-kl}, whose
notation we follow\footnote{Soergel uses $v$ for the variable but we
  find it more convenient to use $w$.}. A basic property of these
polynomials is that $n_{A,A} = 1$ for any alcove $A$ and $n_{B,A} = 0$
if $B \not \leq A$, where the order is induced by the Bruhat order on
$\ti{W}$.

Recall that we have fixed an integer $l > 1$ and have defined the
group $W_{l}$ which acts on $E$ by the scaled dot-action.
 Let $\sigma_l: E \to E$ be the unique linear map such
  that $\sigma_l(\varpi_i) = l_i \varpi_i$ for $1 \leq i \leq r$.
  Given any alcove $A$ for the affine Weyl group $\ti{W}$ we let
  $A_{l}$ denote the subset $\sigma_l(A) - \rho \subset E$. Each such
  $A_{l}$ is a fundamental domain for the action of $W_{l}$ on $E$.
  To any $\lambda$ in $P^+$, we associate an alcove
  $A(\lambda) \in \mcl{A}^+$ as follows:

If $\lambda$ is an affine dot-regular weight, i.e., it is contained in
$A_{l}$ for some alcove $A$, we let $A(\lambda) = A$. If $\lambda$ is
not affine dot-regular we use a construction from \cite[p.~26]{AST}:
Consider the finite set of alcoves defined by
\[
  S = \{B \mid \lambda \in \ov{B_{l}} \} .
\]
The stabiliser of $\lambda$ in $W_{l}$ acts simply transitively
on $S$, so $S$ has a unique maximal element for the order induced by the
Bruhat order on $W^{aff}$ (see, e.g., \cite[Corollary
2.4.5]{bjorner-brenti})  and we let $A(\lambda)$ be this element.

We then have:
\begin{conj}\label{c:akl}
  For $\lambda \in P^+$ and $\nu \in P^+ \cap W_{l}\cdot\lambda$ we
  have
  \[
    f_{\nu}(\tq (\lambda)) = n_{A(\nu),A(\lambda)} .
  \]
\end{conj}
For affine dot-regular $\lambda$ this is due to Andersen and is a
consequence of the (quantum group analogue) of the conjecture in
\cite[Section 3.1]{and-filt}. The extension to all $\lambda$ is
motivated by \cite[Theorem 3.26]{AST}. As far as we are aware, this
conjecture is not known in general even for affine dot-regular
$\lambda$ and $\mf{g} = \mf{sl}_m$, but it is easy for
$\mf{sl}_2$. However, it follows from a theorem of Soergel
\cite[Section 8]{soergel-kma} that the equality holds after evaluating
both polynomials at $1$ if $l$ is sufficiently large.

\subsubsection{The conjecture for simple Weyl modules}

Recall that for
$\vec{\lambda} = (\lambda_1, \lambda_2,\dots,\lambda_n) \in (P^+)^n$
we let
\[
  \dq(\vec{\lambda}) = \dq(\lambda_1) \otimes \dq(\lambda_2) \otimes \dots
  \otimes \dq(\lambda_n) .
\]

The main goal of this section is to formulate some conjectures on the
weight filtration $\mr{W}_{\bullet}\, KZ_{\kappa}(\vec{\lambda},\nu^*)$ in
terms of the representation theory of quantum groups. We first do this
for $\vec{\lambda}$ such that $\dq(\lambda_i)$ is simple for each
$i$.  Note that in this case $\dq(\lambda_i) = \tq (\lambda_i)$ for
each $i$, so
$\dq(\vec{\lambda})$ is a tilting module.

\begin{conj}\label{c:tilting}
  If $\vec{\lambda}$ is such that $\dq(\lambda_i)$ is a simple module
  for each $i$, then the weight filtration on
  $KZ_{\kappa}(\vec{\lambda}, \nu^*)$ coincides, via the
  Kazhdan--Lusztig equivalence, up to a shift by
  $D := {\ExpM}(\vec{\lambda},\nu^*)$ with the dual Andersen filtration
  on $\Hom_{\uq}(\dq(\vec{\lambda}), \nq (\nu))$. In particular, all
  weights of $KZ_{\kappa}(\vec{\lambda}, \nu^*)$ are at least $\ExpM$
  and the weight polynomial is equal to
  $v^{\ExpM}f_{\nu}(\dq(\vec{\lambda}))$.
\end{conj}

We emphasize that there is no assumption on $\nu$ in the conjecture,
but the conjecture simplifies considerably in some special cases.

\begin{remark}\label{r:fink}
  One of the motivations for Conjecture \ref{c:tilting} comes from the
  work of Finkelberg \cite{Finkelberg} which, together with Conjecture
  \ref{c:tilting}, implies that if all the $\lambda_i$ and $\nu$ are
  of ``level $\ell\,$''---see, e.g., \cite{BF-survey} for what this
  means---then $\mr{W}_D(KZ_{\kappa}(\vec{\lambda}, \nu^*))$ is equal
  to the subsheaf of conformal blocks (which is determined
  representation theoretically). We note that it is already known, see
  \cite[Proposition 4.17]{BF-survey}, that the subsheaf of conformal
  blocks is contained in
  $\mr{W}_D(KZ_{\kappa}(\vec{\lambda},\nu^*))$. This implication will
  be proved unconditionally for $\mf{sl}_2$ in Corollary \ref{c:fink}.
\end{remark}

\begin{example}\label{ex:proj}
  $ $
  \begin{enumerate}
  \item If $\dq(\nu)$ is projective then it is also tilting and the
    only indecomposable tilting module with a nonzero map to it is
    $\dq(\nu)$ itself (since the map would split). It follows that in
    this case $f_{\nu}(\dq(\vec{\lambda}))$ is a constant, equal to
    $1$ if $\lambda = \nu$ and $0$ otherwise, so the conjecture
    predicts that $KZ_{\kappa}(\vec{\lambda},\nu^*)$ is pure of weight
    ${\ExpM}$.
  \item If some $\dq(\lambda_i)$ is projective, then so is
    $\dq(\vec{\lambda})$. If $\nq (\nu)$ is simple, then the only
    indecomposable projective with a nonzero map to it is its
    projective envelope. Thus, if both the above conditions hold then
    $f_{\nu}(\dq(\vec{\lambda}))$ is a monomial of degree $d_{\nu}$
    (independent of $\vec{\lambda}$) so
    $KZ_{\kappa}(\vec{\lambda},\nu^*)$ is predicted to be pure of
    weight ${\ExpM} + d_{\nu}$.
  \end{enumerate}
\end{example}

If we combine Conjectures \ref{c:tilting} and \ref{c:akl} we get the
following:
\begin{conj}\label{c:wkl}
  Let
  $\dq(\vec{\lambda}) \cong \oplus_{\mu \in P^+}
  \tq (\mu)^{m_{\vec{\lambda},\mu}}$. Then the ranks of the terms in the associated
  graded of the weight filtration on $KZ_{\kappa}(\vec{\lambda}, \nu^*)$
  are given by the coefficients of the polynomial
  \[
    w^{\ExpM} \sum_{\mu \in P^+} m_{\vec{\lambda},\mu}n_{A(\nu),A(\mu)} .
  \]
\end{conj}
The decomposition of any tilting module as a sum of indecomposable
tilting modules can be computed algorithmically using characters and
Kazhdan--Lusztig polynomials can also be computed algorithmically, so
this conjecture gives an algorithm for computing the ranks of the
associated graded of the weight filtration.

\smallskip

In order to get an upper bound on the weight filtration, we would
like to know the answer to the following question:
\begin{question}\label{q:kl}
  Is the degree of $n_{A,B}$ bounded above by a constant $d(\mf{g})$?
  Can $d(\mf{g})$ be taken to be $|\Phi^+|$?
\end{question}
By \cite[Lemma 4.21, Theorem 5.3 (1)]{soergel-kl}, the degree of
$n_{A,B}$ is bounded by $|\Phi^+|$ if $A \subset \rho + C$.  If
the question has a positive answer and Conjecture \ref{c:wkl} holds
then we have
$\mr{W}_a(KZ_{\kappa}(\vec{\lambda},\nu^*)) =
KZ_{\kappa}(\vec{\lambda},\nu^*)$ for $a \geq {\ExpM} +
d(\mf{g})$. For $\mf{sl}_2$ this holds with $d(\mf{sl}_2) = 1$, see
Theorem \ref{E22} (b).

\subsection{The conjecture for arbitrary Weyl modules}

\subsubsection{A filtration on $\Hom_{\uq}(M,\nq (\mu))$ for arbitrary $M$}

Let $M$ be an arbitrary object in $\uq$-$\Mod$. We will define a
decreasing filtration $\Hom_{\uq}(\dq (\nu),M)^{\bullet}$ on
$\Hom_{\uq}(\dq (\nu),M)$ which agrees with the Andersen filtration
when $M$ is tilting.

\begin{definition}\label{d:filt}
  Let $C_{\mr{min}}(M)$ be the minimal tilting complex of $M$
  constructed by Gruber \cite[Definition 3.9]{gruber}. It is a
  uniquely determined bounded complex of tilting modules
  quasi-isomorphic to $M$ viewed as a complex concentrated in degree
  zero. For any $\nu$, we have
  $H^i(\Hom_{\uq}(C_{\mr{min}}(M),\nq (\nu))) = \Ext^i(M,\nq (\nu))$
  since for any tilting module $T$ we have $\Ext^i(T,\nq (\nu)) = 0$
  for $i\neq 0$. By the functoriality of the Andersen filtration for
  tilting modules, the complex $\Hom_{\uq}(\dq(\nu),C_{\mr{min}}(M))$
  is a complex of filtered vector spaces and we define the Andersen
  filtration $\Ext^i(\dq (\nu),M)^{\bullet}$ on $\Ext^i(\dq (\nu),M)$
  to be the filtration induced on its cohomology.\footnote{Instead of
    $C_{\mr{min}}(M)$ we may actually consider any bounded complex of
    tilting modules quasi-isomorphic to $M$; this follows from
    \cite[Lemma 3.6]{gruber}.}
\end{definition}
This filtration is also functorial in $M$ since maps $M \to M'$
corresponds to homotopy classes of maps
$C_{\mr{min}}(M) \to C_{\mr{min}}(M')$ and homotopies between such
maps induce filtered homotopies on the $\Hom_{\uq}$-complexes.

If $M$ is a tilting module then $C_{\mr{min}}(M) \cong M$, so the
filtration defined above is equal to the original Andersen
filtration. The following example is less trivial.
\begin{example}\label{e:simple}
  Let $\mathfrak{g}$ be arbitrary and $M = \lq (\lambda)$, so
  $\Hom_{\uq}(\dq (\lambda),\lq (\lambda))$ is one dimensional and
  $\Hom_{\uq}(\dq (\lambda),\lq (\nu)) = 0$ for $\nu \neq \lambda$. By
  \cite[Proposition 4.3]{gruber}, for all $i \in \Z$ if $\tq (\mu)$ is
  a direct summand of $C_{\mr{min}}(\lq (\lambda))_i$, then
  $\mu \leq \lambda$ in the usual order on dominant weights. If
  $\mu < \lambda$, then $\Hom_{\uq}(\dq (\lambda), \tq (\mu)) = 0$ and
  we have $\Hom_{\uq}(\dq (\lambda),\tq (\lambda))^1 = 0$, so the
  filtration on
  $\Hom_{\uq}(\dq (\lambda),C_{\mr{min}}(\lq (\lambda)))$ satisfies
  $\Hom_{\uq}(\dq (\lambda),C_{\mr{min}}(\lq (\lambda)))^1= 0$. It
  then follows from the definition that
  $\Hom_{\uq}(\dq (\lambda),\lq (\lambda))^1 = 0$.
\end{example}

\subsubsection{Filtrations on Weyl modules}

In order to formulate a conjecture on the weight filtration of
$KZ_{\kappa}(\vec{\lambda},\nu^*)$ for general $\vec{\lambda}$ we will
have to consider filtrations on Weyl modules. There are various
natural choices that can be made, so we will state the conditions that
we want our filtrations to satisfy as an assumption and then discuss
examples.

\begin{ass}\label{a:filtweyl}
  Any Weyl module $\dq(\lambda)$ for $\lambda \in P^+$ has a finite
  decreasing filtration $\dq(\lambda)^{\bullet}$ with the following
  properties:
  \begin{enumerate}
  \item The layers of the filtration are semi-simple.
  \item For any $\mu \in P^+$ the multiplicity of $\lq (\mu)$ in any layer
    of the filtration is equal to a suitable coefficient of a
    (non-parabolic) Kazhdan--Lusztig polynomial associated to
    $W^{aff}$.
  \end{enumerate}
\end{ass}
For a precise description of the layers in terms of (inverse)
Kazhdan--Lusztig polynomials see \cite[Theorem 7.5]{parshall-scott},
where the existence of such a filtration is proved in types $A$ and
$D$ (where $l > 2$ for type $D_{2n+1}$).  If $\lambda$ is an
affine dot-regular
weight then the filtration of \cite{parshall-scott} is simply the
radical filtration; it is expected, but not known (as far as we are
aware), that this holds for arbitrary $\lambda$. One may hope that the
radical filtration satisfies Assumption \ref{a:filtweyl} for general
$\mf{g}$. One also hopes that the
Jantzen filtration, as defined, e.g., in \cite[Section 7.6]{and-fdr}, is
equal to the radical filtration. We will use this filtration in
Section \ref{s:jantzen}.
\begin{remark}\label{r:ll}
  The length of the radical filtration of any $M$ in $\uq$-$\Mod$ is
  bounded above by $2|\Phi^+| + 1$ by \cite[Corollary
  3.19]{andersen-kaneda}. For most Weyl modules $2|\Phi^+| + 1$ can be
  replaced with $|\Phi^+| + 1$ by \cite[Theorem
  3.17]{andersen-kaneda}. We do not know if this holds for all Weyl
  modules.
\end{remark}

\subsubsection{An $\Ext$ spectral sequence}

Our conjecture on the weight filtration of
$KZ_{\kappa}(\vec{\lambda},\nu^*)$ for general $\vec{\lambda}$ will be
formulated using the tensor product filtration on $\dq(\lambda)$
associated to any filtration satisfying Assumption \ref{a:filtweyl}.
For this we recall the following construction.

Let $R$ be any ring and $M$, $N$ two left $R$-modules. For
a finite filtration

\[
  0 = M_r \subset M_{r-1} \subset \dots M_1 \subset M_0 = M
\]
of $M$ by $R$-submodules, there is a standard $E_1$-spectral
sequence---constructed, for example, using an injective resolution of
$N$ and the induced filtered complex of $\Hom$s---with
\begin{equation}\label{e:ess}
  E_1^{p,q} = \Ext_R^{p+q}(M_{-p}/, M_{-p+1},N)
\end{equation}
converging to $\Ext^*_R(M,N)$.

Note that for all $p$, $E_{\infty}^{p,-p}$ is a subgroup of
$E_1^{p,-p} = \Hom_{\uq}(M_{-p}/M_{-p+1},N)$ which we denote by
$\Hom_{\uq}^0(M_{-p}/M_{-p+1},N)$.

We will use this spectral sequence for $R = \uq$, $M = \dq(\vec{\lambda})$ and
$N = \nq(\nu)$. Since $\Ext^i(\dq(\mu), \nq (\nu)) = 0$
for all $\mu,\nu \in P^+$ and $i \neq 0$ (\cite[Theorem 3.1 ]{AST}),
and $\dq(\vec{\lambda})$ has a filtration with each layer a
Weyl module \cite[Theorem 3.3]{paradowski}, it follows that
$\Ext^i(\dq(\vec{\lambda}),\nq(\nu)) = 0$ for $i > 0$. However,
for say the radical filtration on $\dq(\vec{\lambda})$, the spectral
sequence can be quite complicated and, in particular, need not
degenerate at $E_1$ even for $\mf{sl}_2$.

\subsubsection{A filtration on KZ for general weights}

We assume that a filtration satisfying Assumption \ref{a:filtweyl}
exists for $\mf{g}$ and we fix one such for each $\dq(\lambda)$.
Given an arbitrary
$\vec{\lambda} = (\lambda_1,\lambda_2,\dots,\lambda_n)$, we consider
the tensor product filtration $\dq(\vec{\lambda})^{\bullet}$ on
$\dq(\vec{\lambda})$ defined by
\begin{equation}\label{e:filt}
  \dq(\vec{\lambda})^i = \sum_{(i_i,_2,\dots,i_n), \sum_j i_j = i}
  \dq(\lambda_1)^{i_1} \otimes \dq(\lambda_2)^{i_2} \otimes \dots
  \otimes \dq(\lambda_n)^{i_n} .
\end{equation}
It follows that $\dq(\vec{\lambda})^i/\dq(\vec{\lambda})^{i+1}$
is a direct sum of tensor products of simple modules with known
multiplicities.

By taking the module $M$ in Definition \ref{d:filt} to be
$D(\dq(\vec{\lambda})^i/\dq(\vec{\lambda})^{i+1})$ and dualising the
resulting filtration, we obtain a canonical increasing filtration
$\Hom_{\uq}(\dq(\vec{\lambda})^i/\dq(\vec{\lambda})^{i+1},\nq(\nu))^{\bullet}$
on
$\Hom_{\uq}(\dq(\vec{\lambda})^i/\dq(\vec{\lambda})^{i+1},\nq(\nu))$
for each integer $i \geq 0$ and hence, by restriction, an increasing
filtration
$\Hom_{\uq}^0(\dq(\vec{\lambda})^i/\dq(\vec{\lambda})^{i+1},\nq(\nu))^{\bullet}$
on
$\Hom_{\uq}^0(\dq(\vec{\lambda})^i/\dq(\vec{\lambda})^{i+1},\nq(\nu))$
for each integer $i \geq 0$.

Assuming $\Hom_{\uq}(\dq(\vec{\lambda}),\nq (\nu))$ is nonzero, recall
that we have defined the integer $\ExpM =
\ExpM(\vec{\lambda},\nu^*)$. We then reindex our filtrations and
define candidate ``weight'' filtrations as follows:
\[
  \mr{W}'_r\,
  \Hom_{\uq}(\dq(\vec{\lambda})^i/\dq(\vec{\lambda})^{i+1},\nq (\nu))
  := \Hom_{\uq}(\dq(\vec{\lambda})^i/\dq(\vec{\lambda})^{i+1},\nq
  (\nu))_{r+{\ExpM} -i}
\]
and analogously for
$\Hom_{\uq}^0(\dq(\vec{\lambda})^i/\dq(\vec{\lambda})^{i+1},\nq(\nu))$.

\begin{conj}\label{c:simple}
  The weight filtration on
  $\mr{W}_{\bullet}\,KZ_{\kappa}(\vec{\lambda},\nu^*)$ for arbitrary
  $\vec{\lambda}$ and $\nu$ is compatible with the filtrations
  $\mr{W}_{\bullet}' \,
  \Hom_{\uq}(\dq(\vec{\lambda})^i/\dq(\vec{\lambda})^{i+1},\nq (\nu))$.
  More precisely, this means the following:

  Let $\mr{W}_{\bullet}\, \Hom_{\uq}(\dq(\vec{\lambda}),\nq (\nu))$ be
  the filtration on $\Hom_{\uq}(\dq(\vec{\lambda}),\nq (\nu))$
  obtained by transporting the Hodge-theoretic weight
  filtration $\mr{W}_{\bullet}\,KZ_{\kappa}(\vec{\lambda},\nu^*)$ to
  $\Hom_{\uq}(\dq(\vec{\lambda}),\nq (\nu))$ via the Kazhdan--Lusztig
  equivalence. For $i \geq 0$ this induces increasing filtrations on the subquotients
  $\Hom_{\uq}^0(\dq(\vec{\lambda})^i/\dq(\vec{\lambda})^{i+1},\nq(\nu))$
  of $\Hom_{\uq}(\dq(\vec{\lambda}),\nq (\nu))$ induced by the
  spectral sequence \eqref{e:ess} with $R = \uq$, $M =
  \dq(\vec{\lambda})$ with the filtration \eqref{e:filt}, and $N =
  \nq(\nu)$, which we also denote by $\mr{W}_{\bullet}$. Then
  the conjecture is that
  \begin{equation}\label{e:conj}
    \mr{W}'_r\,
    \Hom_{\uq}^0(\dq(\vec{\lambda})^i/\dq(\vec{\lambda})^{i+1},\nq (\nu))
    = \mr{W}_r\,
    (\dq(\vec{\lambda})^i/\dq(\vec{\lambda})^{i+1},\nq (\nu))
  \end{equation}
  for all $i$ and $r$.
\end{conj}
When $\mf{g} = \mf{sl}_2$ we provide strong evidence for this
conjecture (Theorem \ref{t:sl2fm}) and, in particular, prove it when
$n=2$ (Corollary \ref{c:sl22}).

The $\Ext$ spectral sequence does not always degenerate at
$E_1$---there are such examples even for $\mathfrak{sl}_2$ and two
finite points---so Conjecture \ref{c:simple} is not in general
sufficient to compute the ranks of the layers of the weight filtration
of $KZ$.  However, it does reduce the computation to a purely
algebraic, but perhaps non-trivial, computation in the category
$\uq$-$\Mod$. For $\mf{sl}_2$ we show in Section \ref{s:nsl2} that the
spectral sequence degenerates at $E_2$ and Conjecture \ref{c:simple}
then leads to explicit computable predictions for the ranks.
\begin{remark}\label{r:ell}
  $ $
  \begin{enumerate}
  \item If $\vec{\lambda} = (\lambda)$ and $\nu = \lambda$ then the
    truth of Conjecture \ref{c:simple} is implied by Example
    \ref{e:simple}.
  \item If all $\dq(\lambda_i)$ are simple then Conjecture
    \ref{c:simple} reduces to Conjecture \ref{c:tilting}.
  \item Conjecture \ref{c:simple} and Remark \ref{r:ll} lead us to
    expect that for any $\vec{\lambda}$,
    $\mr{W}_a KZ_{\kappa}(\vec{\lambda},\nu^*) = 0$ for
    $a < {\ExpM} - n|\Phi^+|$. A positive answer to Question
    \ref{q:kl} together with Conjecture \ref{c:simple} would imply
    that
    $\mr{W}_a KZ_{\kappa}(\vec{\lambda},\nu^*) =
    KZ_{\kappa}(\vec{\lambda},\nu^*)$ for $a \geq {\ExpM} + |\Phi^+|$.
  \end{enumerate}
\end{remark}

\begin{example}\label{e: GenQGN0}
  Suppose $\lambda_1 \in N_0$, so $\dq(\lambda_1)$ is simple as well
  as projective. Since the tensor product of a projective module with
  any other module in $\uq$-$\Mod$ is projective, this implies that
  $\dq(\vec{\lambda})^i/\dq(\vec{\lambda})^{i+1}$ is projective for
  all $i$ so there are no $\Ext^j$ terms in the spectral sequence for
  $j > 0$ and it degenerates at $E_1$. Since projective modules are
  also tilting, $\dq(\vec{\lambda})^i/\dq(\vec{\lambda})^{i+1}$ is
  also tilting, so the filtration on
  $\Hom_{\uq}(\dq(\vec{\lambda})^i/\dq(\vec{\lambda})^{i+1},\nq
  (\nu))$ is the dual Andersen filtration. Conjectures \ref{c:akl} and
  \ref{c:simple} jointly imply that all ranks of the layers in the
  weight filtration
  $\mr{W}_{\bullet}\,KZ_{\kappa}(\vec{\lambda},\nu^*)$ can be computed
  algorithmically in this case. A similar statement holds if
  $\nu \in N_0$.
\end{example}

The utility of Conjecture \ref{c:simple} for arbitrary $\vec{\lambda}$
and $\nu$ is limited to some extent by the fact that a tensor product
of simple modules might have a complicated structure, so computing the
corresponding Andersen filtration is a nontrivial problem. However,
there are no difficulties for $\mf{sl}_2$ and we discuss this case in
more detail in the next section.

\section{The case of $\mf{sl}_2$}\label{s:sl2}
In this section we assume that $\mf{g} = \mf{sl}_2$,
$q = \exp(-i \pi/\kappa)$, and $l$, the order of $q^2$, is $>1$.  We
will show how Conjecture \ref{c:simple} leads to explicit computable
predictions for the ranks of the layers of the weight filtration of
all KZ variations.  The key result here is Theorem \ref{t:sl2deg}
which proves the degeneration of the $\Ext$ spectral sequence at
$E_2$.

\subsection{Recall of some basic properties of $\uq(\mf{sl}_2)$-$\Mod$.}

We will identify $P$ with $\Z$ and $P^+$ with $\Z_{\geq 0}$ and use Latin
letters $a,b,\dots$ to denote weights. Any weight $a$ can be written
as $a = a'l + a''$ with $0 \leq a'' \leq l -1$. We define maps
$d, u: P^+ \to P^+$ by
\begin{equation*}
  d(a) =
  \begin{cases}
    a & \mbox{if }  a \leq l -1 \mbox{ or } a \in N_0, \\
    a'l - (2 + a'') &  \mbox{otherwise}.
  \end{cases}
\end{equation*}
and
\begin{equation*}
  u(a) =
  \begin{cases}
    a  & \mbox{if }  a \in N_0, \\
    (a'+ 2)l - (2 + a'') & \mbox{otherwise}.
  \end{cases}
\end{equation*}

A Weyl module $\dq(a)$ is simple iff $a \leq l - 2$ or $a \in
N_0$. If $\dq(a)$ is not simple then there is an exact sequence
\begin{equation}\label{e:weyl}
  0 \to \lq (d(a)) \to \dq(a) \to \lq (a) \to 0 .
\end{equation}
A tilting module $\tq (a)$ is simple also iff $a \leq l - 2$ or
$a \in N_0$. If $\tq (a)$ is not simple then there is an exact sequence
\begin{equation}\label{e:tilt}
  0  \to \dq(a) \to \tq (a) \to \dq(d(a)) \to 0 .
\end{equation}
See, e.g., \cite[Section 2]{AndersenTubbenhauer}, for the proofs of
these statements, from which the statements below can be deduced
easily.

We note that the Andersen filtration of a tilting module in the case
of $\mf{sl}_2$ is simple to describe. If $\tq (a) = \dq(a)$, i.e.,
$a \leq l -2$ or $a \in N_0$, then $f_{b}(\tq (a)) = 1$ if $a = b$
and is $0$ otherwise. If $\tq (a) \neq \dq(a)$ then $f_{a}(\tq (a)) = 1$,
$f_{d(a)}(\tq (a)) = w$ and $f_b(\tq (a)) = 0$ for all other $b$. Conjecture
\ref{c:akl} is trivial to check and Question \ref{q:kl} has a positive
answer with $d(\mf{sl_2})$ being $1$. Furthermore, it is trivial that
the radical filtration for a Weyl module satisfies Assumption
\ref{a:filtweyl} and its length is at most $2$.

\subsection{The $\Ext$ spectral sequence and the weight filtration
  explicit}\label{s:nsl2}

\subsubsection{Degeneration of the $\Ext$ spectral sequence}

 Suppose $\vec{a} = (a_1,a_2,\dots,a_n)$ is such that
$\dq(a_i)$ is simple, equivalently tilting, for all $i$, so we may
write
\[
  \bigotimes_i\dq(a_i) \cong \bigoplus_b \tq (b)^{m_{\vec{a},b}} .
\]
with $m_{\vec{a},b}$ being $0$ for all but finitely many $b$. Then for
any $b$, Conjecture \ref{c:tilting} predicts that the weight
${\ExpM}(\vec{a},b)+ 1$ layer of $KZ_{\kappa}(\vec{a},b)$ has rank
$m_{\vec{a},u(b)}$ and this is also the predicted rank of the weight
${\ExpM}(\vec{a},u(b))$ layer of $KZ_{\kappa}(\vec{a},u(b))$. We will
prove that this ``reflection'' property does indeed hold, and this
will be a key for proving that the weights of $KZ_{\kappa}(\vec{a},b)$
are in the range predicted by Conjecture \ref{c:tilting}.

\smallskip

As a direct consequence of Conjecture \ref{c:simple}  and
  the properties of the Andersen filtration for tilting and simple
  modules for $\mf{sl}_2$ we then get:
  \begin{conj}\label{c:sl2a}
  Let $\vec{a} = (a_1,a_2,\dots,a_n)$ and assume $\tq (a_i) =
  \dq(a_i)$ for $1 \leq i \leq n_0$. Then the possible nonzero
  weights of $KZ_{\kappa}(\vec{a},b)$ are contained in the interval
  $[\ExpM + n_0 - n, \ExpM +1]$, where $\ExpM = \ExpM(\vec{a},b)$.
\end{conj}
We will prove the special case of this when $n = n_0$ in Section
\ref{s:exp} and the lower bound in the general case (using very
different methods) in a work currently under preparation.

\smallskip
As noted earlier, the spectral sequence for
$\Ext_{\uq}^*(\dq(\vec{\lambda}),\nq (\nu))$ with respect to the
filtration $\dq(\vec{\lambda})^{\bullet}$ can in principle be
complicated and need not degenerate at $E_1$, but for $\mf{sl}_2$ we
have:
\begin{theorem}\label{t:sl2deg}
  For any $\vec{a}$ and $b$, the spectral sequence for
  $\Ext_{\uq}^*(\dq(\vec{a}),\nq (b))$ with respect to the filtration
  $\dq(\vec{a})^{\bullet}$ degenerates at $E_2$.
\end{theorem}

Theorem \ref{t:sl2deg} will be proved in Section \ref{s:sl2deg}.

\subsubsection{}
We now analyse in detail the terms in the spectral sequence. To do
this we need to know the structure of the layers of the filtration
$\dq(\vec{a})^{\bullet}$. This follows from:
\begin{lemma}\label{l:ab}
  For any $a,b \geq 0$ $\lq (a) \otimes \lq (b)$ is a direct sum of simple
  and projective tilting modules.
\end{lemma}

\begin{proof}
  Let $a = a'l + a''$ and $b = b'l + b''$.  By the Lusztig
  tensor product theorem \cite[Theorem 7.4]{LTT} we have
  $\lq (a) \cong \lq (a'l) \otimes \lq (a'')$ and similarly for $\lq (b)$. It
  follows that
  \[
    \lq (a) \otimes \lq (b) \cong (\lq (a'l) \otimes \lq (b'l)) \otimes
    (\lq (a'') \otimes \lq (b'')) .
  \]
  Now $\lq (a'')$ and $\lq (b'')$ are tilting modules since
  $a'', b'' < l$, therefore so is $\lq (a'') \otimes \lq (b'')$,
  and it can be written as a direct sum of simple (tilting) modules
  and projective tilting modules; the former are the ones with highest
  weight $< l -1$ and the latter are the other ones. Using Lusztig's
  theorem once again, we have
  \begin{equation}\label{e:a'b'}
    \lq (a'l) \otimes \lq (b'l) \cong \bigoplus_{|a'-b'| \leq c' \leq
      a' + b'} \lq (c'l) .
  \end{equation}
  Since the tensor product of a projective $\uq$-module with any
  $\uq$-module is projective, hence also tilting, the lemma follows
  from the above by a third application of Lusztig's theorem.
\end{proof}
The proof of the lemma leads to an explicit description of the
indecomposable summands, we will work this out below in Section
\ref{s:lab}.

\smallskip

It follows from Lemma \ref{l:ab} and induction that the tensor product of any
number of simple modules is a direct sum of simple modules and
projective tilting modules, therefore all the layers in the filtration
$\dq(\vec{a})^{\bullet}$ are sums of simple modules and projective
tilting modules. Combining this with the next lemma allows us to
compute all terms of the spectral sequence.
\begin{lemma}\label{l:exts}
  Let $b \geq 0$ be any integer.
  \begin{enumerate}
  \item If $b \in N_0$ then $\Hom_{\uq}(\lq (b),\nq (b))$ is of rank $1$ and
    this is the only possible nonvanishing $\Ext_{\uq}^i(\lq (a),\nq (b))$ for
    any $i,a \geq 0$.
  \item If $b \notin N_0$ then $\Ext_{\uq}^i(\lq (u^i(b)),\nq (b))$ is of rank
    $1$ for each $i \geq 0$ and these are the only possible
    nonvanishing $\Ext_{\uq}^i(\lq (a),\nq (b))$ for any $i,a \geq 0$.
  \end{enumerate}
  In particular, given any $a$ there are only finitely many $b$ such
  that $\Ext_{\uq}^i(\lq (a),\nq (b)) = 0$ for some $i$.
\end{lemma}
\begin{proof}
  Both statements for $\Hom_{\uq}(\lq (a),\nq (b))$ are clear. The vanishing
  of $\Ext_{\uq}^i(\lq (a),\nq (b))$ for $b \in N_0$ and $i >0$ is also clear
  if $b \in N_0$ since $\nq (b)$ is injective in this case.

  If $b \notin N_0$ the statement for $\Ext^i$ with $i > 0$ is reduced
  to the statement for $\Hom$ by induction on $i$ using the exact
  sequences
  \[
    0 \to \lq (u^j(a)) \to \dq(u^{j+1}(a)) \to \lq (u^{j+1}(a)) \to 0
  \]
  for all $j \geq 0$ and the fact that $\Ext_{\uq}^i(\dq(a),\nq (b)) =
  0$ for any $i > 0$ and $a \geq 0$.
\end{proof}

\subsubsection{Computing the weight filtration}\label{s:compute}
We now explain how, assuming Conjecture \ref{c:simple} and Theorem
\ref{t:sl2deg}, and using the explicit computations above, one may
compute all the ranks of the layers of the weight filtration of
$KZ_{\kappa}(\vec{a},b)$.

\smallskip

Firstly, all the $E_1^{r,s}$ in the spectral sequence can be computed
explicitly by Lemma \ref{l:exts} since the layers in the filtration
$\dq(\vec{a})^{\bullet}$ are direct sums of tilting modules and simple
modules (by Lemma \ref{l:ab} and induction on the number of
factors). Also, without assuming any conjecture, we know that
$E_{\infty}^{r,s}$ is zero whenever $r + s \neq 0$. The degeneration
at $E_2$ assumption then implies that $E_2^{r,s}$ is zero whenever
$r + s \neq 0$.

The differentials $d_1^{r,s}$ are maps
\begin{equation}\label{e:ss}
  \Ext_{\uq}^{r+s}(\dq(\vec{a})^{-r}/\dq(\vec{a})^{-r+1}, \nq (b)) \to
  \Ext_{\uq}^{r+s +1}(\dq(\vec{a})^{-r+1}/\dq(\vec{a})^{-r}, \nq (b))
  .
\end{equation}
The targets of the maps are zero for large $r+s$, so the degeneration
at $E_2$ and the knowledge of the ranks of all the $E_1^{r,s}$, lets
us inductively compute the ranks of all the $d_1^{r,s}$. In
particular, we can compute the ranks of all $d_1^{r,-r}$ and therefore
the ranks of the $E_2^{-r,r} = E_{\infty}^{-r,r} =
\Hom_{\uq}^0(\dq(\vec{a})^{r}/\dq(\vec{a})^{r+1}, \nq (b))$ for all
$r \geq 0$.

The dual Andersen filtration on
$\Hom_{\uq}(\dq(\vec{a})^{r}/\dq(\vec{a})^{r+1}, \nq (b))$ can be
computed using Example \ref{e:simple} and Lemma \ref{l:ab}.  However,
$\Hom_{\uq}^0(\dq(\vec{a})^{r}/\dq(\vec{a})^{r+1}, \nq (b))$ can be a
proper subspace of
$\Hom_{\uq}(\dq(\vec{a})^{r}/\dq(\vec{a})^{r+1}, \nq (b))$ and to
compute the ranks of the dual Andersen filtration restricted to
$\Hom_{\uq}^0(\dq(\vec{a})^{r}/\dq(\vec{a})^{r+1}, \nq (b))$ we need
to examine the structure of $\dq(\vec{a})^{r}/\dq(\vec{a})^{r+1}$. As
we have seen, this module is a sum $S_r \oplus P_r$ where $S_r$ is a
sum of (non-projective) simple modules $S_r$ and $P_r$ is a projective
module. Now projective modules in $\uq$-$\Mod$ are also injective in
$\uq$-$\Mod$, so this summand splits off in the filtration
$\dq(\vec{a})^{\bullet}$, which implies that $d_1^{-r,r}$ must vanish
on the subspace $\Hom_{\uq}(P_r, \nq (b))$ of
$\Hom_{\uq}(\dq(\vec{a})^{r}/\dq(\vec{a})^{r+1}, \nq (b))$. It follows
that the rank of the second layer of the dual Andersen filtration of
$\Hom_{\uq}(\dq(\vec{a})^{r}/\dq(\vec{a})^{r+1}, \nq (b))$ equals the
rank of its restriction to
$\Hom_{\uq}^0(\dq(\vec{a})^{r}/\dq(\vec{a})^{r+1}, \nq (b))$ and the
rank of the first layer decreases by the rank of $d_1^{-r,r}$. Since
we can compute the rank of $d_1^{-r,r}$, it follows that we can
compute the ranks of the LHS of the equality in Conjecture
\ref{c:simple}, and therefore also the RHS, if we assume that the
conjecture holds.

\subsubsection{}

In order to prove Theorem \ref{t:sl2deg} we will need to know the
structure of $\dq(\vec{a})$ as a sum of indecomposables and how
this structure relates to the tensor product filtration. For this we
will need several lemmas which we prove next.
\begin{lemma}\label{l:extweyl}
  Let $a,b \geq 0$. If $\Ext_{\uq}^1(\dq(a),\dq(b)) \neq 0$ then $b =
  u(a)$ and if $a \notin N_0$ then $\Ext_{\uq}^1(\dq(a),\dq(u(a)))$ has
  rank one.
\end{lemma}

\begin{proof}
  We will prove this by induction.  If $b = b''$ then
  $\dq(b) = \tq (b)$ is a tilting module, so
  $\Ext_{\uq}^1(\dq(a),\dq(b))= 0$ for all $a$.

  Assuming the statement for $b$ we will prove it for $u(b)$. There is
  an exact sequence
  \[
    0 \to \dq(u(a)) \to \tq (u(a)) \to \dq(a) \to 0
  \]
  and $\tq (u(a))$ is both projective and injective. Given any $b$ we
  therefore get an exact sequence
  \begin{multline*}
    \Hom_{\uq}(\dq(a),\dq(b)) \to \Hom_{\uq}(\tq (u(a)),\dq(b)) \to
    \Hom_{\uq}(\dq(u(a)),\dq(b)) \\
    \to \Ext_{\uq}^1(\dq(a), \dq(b)) \to 0
    .
  \end{multline*}
  Now $\Hom_{\uq}(\dq(u(a)),\dq(b)) \neq 0$ iff $u(a) = b$ or
  $u(a) = d(b)$ and in both cases the space has rank one.

  If $u(a) = b$ we see that $\Ext_{\uq}^1(\dq(a),\dq(u(a)))$ has rank
  at most one and the non-split extension above for $\tq (u(a))$ implies
  that the rank is equal to one.

  If $u(a) = d(b)$ then we have a sequence of maps
  \[
    \dq(u(a)) \to \tq (u(a)) \twoheadrightarrow \nq (u(a))=
    \nq (d(b)) \twoheadrightarrow \lq (d(b)) \hookrightarrow \dq(b)
  \]
whose composite is nonzero, so we see that
$\Ext_{\uq}^1(\dq(a),\dq(u^2(a))) = 0$.
\end{proof}

Recall that a filtration
\[
  0 = M_0 \subset M_1 \subset \dots \subset M_r = M
\]
of a $\uq$-module $M$ is called a $\dq$-filtration if $M_i/M_{i-1}$
is a Weyl module for $i \leq i \leq r$.
\begin{lemma}\label{l:donkin}
  $ $
  \begin{enumerate}
  \item If a module $M$ has a $\dq$-filtration then $M$ is a direct
    sum of projective modules and Weyl modules.
  \item A tensor product of Weyl modules is a direct sum of projective
  modules and Weyl modules.
\end{enumerate}
\end{lemma}
\begin{proof}
  We first prove (1).  By Donkin's criterion for the existence of a
  $\dq$-filtration \cite[Proposition 3.5]{AST} we may and do assume
  that $M$ is indecomposable.  We now use induction on $r$.

  If $r = 1$ then $M$ is a Weyl module by definition of a
  $\dq$-filtration, so suppose $r > 1$. Then $M_1 \cong \dq(a)$
  for some $a$ and by induction $M/M_1$ is a direct sum of projective
  tilting modules and Weyl modules. If there is a nonzero summand
  which is a projective module, then the summand must split
  off from $M$. Since $M$ is indecomposable this is not possible. Thus
  $M/M_1 \cong \oplus_{j=1}^s \dq(a_j)$ is a sum of Weyl modules.
  This gives us submodules $N_j$ of $M$ fitting into exact sequences
  \[
    0 \to \dq(a) \to {N}_j \to \dq(a_j) \to 0 .
  \]
  Since $M$ is indecomposable, there exists $j$ such that this
  sequence doesn't split. Then $N_j$ is a projective (therefore also
  injective) module by Lemma \ref{l:extweyl} and then
  indecomposability implies $M=N_j$, so (1) is proved.

  Part (2) follows from (1) and the fact that the tensor product of
  Weyl modules has a $\dq$-filtration by the main theorem of
  \cite{paradowski} (but see \cite[Claim 3.10.1]{AST} for a simpler
  proof in type $A$).
\end{proof}

\begin{lemma}\label{l:socle}
  Let $M$ be a $\uq$-module which is a direct sum of simple modules,
  Weyl modules and indecomposable tilting modules. If $L$ is a simple
  submodule of $M$ then there is an indecomposable summand $M'$ of $M$
  with $L = \mr{soc}(M')$.
\end{lemma}
\begin{proof}
  Let $L\cong \lq (a)$ for some $a \geq 0$. The only non-simple Weyl
  module with socle containing $\lq (a)$ is $\dq(u(a))$ and the only
  non-simple indecomposable tilting module with socle containing $\lq (a)$
  is $\tq (u(a))$ and in both these cases the socle is equal to $\lq (a)$.
  We may therefore assume that
  $M \cong \lq (a)^{m_1} \oplus \dq(u(a))^{m_2} \oplus \tq (u(a))^{m_3}$,
  for some $m_1, m_2, m_3 \geq 0$.

  If the projection of $L$ to one of the $\lq (a)$ summands is nonzero,
  then we can split if off and take $M' \cong \lq (a)$. If all such
  projections are zero, we may assume $m_1 = 0$. Any map from $\lq (a)$
  to $\dq(u(a))$ (resp. $\tq (u(a))$) extends (uniquely) to a map from
  $\dq(u(a))$ to $\dq(u(a))$ (resp. $\tq (u(a))$), so if the
  projection of $L$ to one of the $\dq(u(a))$ summands is nonzero
  we may take $M' \cong \dq(u(a))$. If all such maps are zero we
  may assume $m_2=0$, and any map $\lq (a) \to \tq (u(a))$ extends to a map
  from $\tq (u(a))$ to itself, so in this case we may take $M' \cong \tq (u(a))$.
\end{proof}

  \subsubsection{}

To analyse how the filtration $\dq(\vec{a})^{\bullet}$ relates to
the structure of $\dq(\vec{a})$ as given by Lemma \ref{l:donkin}
(2), we will need some elementary lemmas about splitting filtrations.
They apply in arbitrary abelian categories, so in this section we fix
an abelian category $\bf{A}$ and all objects and morphisms
will be in this category.

\begin{lemma}\label{l:fil1}
  Let $M$ be an object of $\bf{A}$ with a finite descending
  separated filtration $M^{\bullet}$. Let $N$ be a subobject of $N$
  and endow it with the induced filtration $N^{\bullet}$, and also
  endow the quotient $Q$ with the quotient filtration $Q^{\bullet}$, i.e., $Q^i :=
  (M^i + N)/N$.
  \begin{enumerate}
  \item If the quotient map $M \to Q$ has a section which is
    compatible with the filtration then
    $M^{\bullet} \cong N^{\bullet} \oplus Q^{\bullet}$ as filtered
    objects.
  \item If $Q^i = 0$ and $N^{i-1} = N$ then any section of the
    quotient map is compatible with the filtration.
  \end{enumerate}
\end{lemma}
\begin{proof}
  For (1), the section gives a map
  $N^{\bullet} \oplus Q^{\bullet} \to M^{\bullet}$ of filtered objects
  and we must show that for each $i$, $N^i \oplus Q^i$ maps onto
  $M^i$. But by definition of the induced and quotient filtrations, we
  have an exact sequence
  \[
    0 \to N^i \to M^i \to Q^i \to 0
  \]
  so this is clear.

  We now prove (2).  Let $\sigma: Q \to M$ be a section. Since
  $Q^i = 0$ it is clear that for any $j \geq i$ we have
  $\sigma(Q^j) \subset M^j$. For general $j$ we must have
  $\sigma(Q^j) \subset M^j + N = M^j + N^{i-1}$. If $j < i$ then we
  have $M^j + N^{i-1} = M^j$ so (2) follows.
\end{proof}

\begin{lemma}\label{l:fil2}
  Let $M$ be an object of $\bf{A}$ with a finite descending separated
  filtration $M^{\bullet}$. Suppose for some $i$ we have
  $M^{i-1}/M^{i} \cong P \oplus R$ where $P$ is a projective and
  injective object of $\bf{A}$. Let $P^{\bullet}$ be the filtration
  on $P$ such that $P^{i-1} = P$ and $P^{i} = 0$. Then there is a
  filtered object $\widetilde{M}^{\bullet}$ such that
  $M^{\bullet} \cong \widetilde{M}^{\bullet} \oplus P^{\bullet}$,
  $\widetilde{M}^{j-1}/\widetilde{M}^{j} \cong M^{j-1}/M^{j}$ for
  $j \neq i$ and $\widetilde{M}^{i-1}/\widetilde{M}^{i} \cong R$.
\end{lemma}

\begin{proof}
  We first consider the case $M = M^{i-1}$, so we have maps
  $M \to M^{i-1}/M^{i} \to P$ whose composite we denote by
  $\pi:M \to P$. We let $\widetilde{M} = \ker(\pi)$ with the induced
  filtration $\widetilde{M}^{\bullet}$. Since $P$ is projective, $\pi$
  has a section. It follows from Lemma \ref{l:fil1} that the section
  is compatible with the filtration, so
  $M^{\bullet} \cong \widetilde{M}^{\bullet} \oplus P^{\bullet}$. It
  is also clear that the filtered object $\widetilde{M}^{\bullet}$ has
  the desired properties.

  Now we prove the general case. We may apply the previous case to
  $N := M^{i-1} \subset M$ with its induced filtration $N^{\bullet}$,
  so we get
  $N^{\bullet} \cong \widetilde{N}^{\bullet} \oplus P^{\bullet}$.
  Using such an isomorphism, we may view $P^{\bullet}$ as a subobject
  of $M^{\bullet}$ with the induced filtration, and the quotient map
  $M \to M/P$ has a section $\sigma$ since $P$ is
  injective. Furthermore, the injectivity of $P$ also implies that any
  map $N/P \to P$ extends to a map $M/P \to P$. This implies that we
  may choose $\sigma$ such that its restriction to $N/P$ is equal to
  the original splitting of the map $N \to N/P$. Such a splitting is
  compatible with the filtrations so letting $\widetilde{M}^{\bullet}$
  be $M/P$ with its induced filtration we get an isomorphism of
  $M^{\bullet} \cong \widetilde{M}^{\bullet} \oplus P^{\bullet}$.
\end{proof}

\subsubsection{}\label{s:lab}

We now return to analysing the filtered object
$\dq(\vec{a})^{\bullet}$.
\begin{lemma}\label{l:st}
  Let $a = a'l $. Then for $b$ such that
  $l -1 \leq b < 2l -1$ we have
  $\lq (a) \otimes \tq (b) = \tq (a+b)$.
\end{lemma}

\begin{proof}
  If $a = 0$ the lemma is obvious so we assume $a>0$.

  The condition on $b$ implies that $\tq (b)$ is projective, so
  $\lq (a) \otimes \tq (b)$ is also a projective module, hence it is a
  direct sum of tilting modules with highest weight $\geq l
  -1$. It is then clear that $\tq (a + b)$ is a direct summand, so we
  need to show that there are no other summands.

  If $b = l -1$ then $\tq (b) = \lq (b)$ and the lemma follows
  immediately from Lusztig's tensor product theorem \cite[Theorem
  7.4]{LTT}. If not, the simple constituents of $\tq (b)$ are
  $\lq (b)$ and two copies of $\lq (d(b))$.  It then follows from
  Lusztig's tensor product theorem that $\lq (a) \otimes \tq (b)$ has
  four simple constituents since $\lq (a) \otimes \lq (b)$ has two
  simple components and $\lq (a) \otimes \lq (d(b))$ is simple. The
  number of simple constituents of $\tq (a'l + b)$ is also four, so
  the lemma is proved.
\end{proof}

We now make explicit the decomposition of $\lq (a) \otimes \lq (b)$ in Lemma
\ref{l:ab}. For this we need to consider two cases:
\begin{itemize}
\item
If $a'' + b'' \geq l$, then $\lq (a'') \otimes \lq (b'') \cong T \oplus L$
with \[
  T = \tq (a'' + b'') \oplus \tq (a'' +b'' -2) \oplus \dots \oplus \tq (p) ,
\]
\[
  L = \lq (2l - 2 -a'' - b'' -2) \oplus \lq (2l - 2 -a'' - b'' -4) \oplus
  \dots \oplus \lq (|a'' - b''|) ,
\]
and $p \in \{l,l -1\}$ has the same parity as $a'' + b''$.
\item
If $a'' + b'' < l$ then
\[
  \lq (a'') \otimes \lq (b'') \cong \lq (a'' + b'') \oplus \lq (a'' + b'' -2) \oplus \dots \oplus
  \lq (|a'' -b''|) .
\]
\end{itemize}
Using \eqref{e:a'b'}, the proof of Lemma \ref{l:ab} gives
\begin{align}
  \lq (a) \otimes \lq (b) &=  (\lq (a'') \otimes \lq (b'')) \otimes (\lq (a'l)
  \otimes \lq (b'l)) \nonumber \\
 &\cong (T \oplus L) \otimes \Bigl (\bigoplus_{|a'-b'| \leq c' \leq
      a' + b'} \lq (c'l) \Bigr ) . \label{l:tl}
\end{align}
We can now apply Lemma \ref{l:st} and Lusztig's theorem to write
$\lq (a) \otimes \lq (b)$ as a sum of explicit indecomposable modules.

\bigskip

The lemma below is proved by elementary arithmetic computations using
the description of $\lq (a) \otimes \lq (b)$ given above. The detailed
proof is given in Appendix \ref{a:A}.

\begin{lemma}\label{l:u}
  Suppose $\vec{a} = (a_1,a_2)$. If $\lq (c)$, $c \notin N_0$, is any
  simple summand of $\dq(\vec{a})^2/ \dq(\vec{a})^3$ then
  $\lq (u(c))$ is not a summand of
  $\dq(\vec{a})/ \dq(\vec{a})^1$.
\end{lemma}

\subsubsection{}

The following technical lemma is the key to the proof of Theorem
\ref{t:sl2deg}.
\begin{lemma}\label{l:veca2}
  There is a finite separated decreasing filtration
  $\fil^{\bullet}\dq(\vec{a})$ of $\dq(\vec{a})$ such that each layer
  of the filtration is a projective module or a Weyl module and the
  filtration induced on each layer of $\fil^{\bullet}\dq(\vec{a})$
  from the filtration $\dq(\vec{a})^{\bullet}$ has length at most
  two. Furthermore, if a layer of $\fil^{\bullet}\dq(\vec{a})$ is a
  Weyl module then the filtration induced on it by
  $\dq(\vec{a})^{\bullet}$ has simple layers and if a layer of
  $\fil^{\bullet}\dq(\vec{a})$ is a projective module then the
  filtration induced on it by $\dq(\vec{a})^{\bullet}$ is one-step.
\end{lemma}

\begin{proof}
  We first consider the case $\vec{a} = (a_1,a_2)$, the case $\vec{a}
  = (a)$ being obvious.
  
  By Lemma \ref{l:ab} each layer of the filtered module
  $\dq(\vec{a})^{\bullet}$ is a direct sum $Q \oplus S$ where $Q$
  is a projective module and $S$ is a sum of non-projective simple
  modules. By Lemma \ref{l:fil2}, since projectives are also injective
  in $\uq$-$\Mod$, each such $Q$ can be split off and we get an
  isomorphism
  \[
    \dq(\vec{a})^{\bullet} \cong P^{\bullet} \oplus M^{\bullet}
  \]
  where $P^{\bullet}$ is a sum of projective modules with one-step
  filtration and the layers of $M^{\bullet}$ have no projective
  summands.

  By Lemma \ref{l:donkin} (2), $M$ is a direct sum of indecomposable
  projective modules and Weyl modules. By the definition of the tensor
  product filtration on $\dq(\vec{a})$, we have $M^3 = 0$ and $M^2$ is
  semisimple. By Lemma \ref{l:socle} and induction on the number of
  simple summands of $M^2$, we see that we may write
  $M = M' \oplus M''$ with both $M'$ and $M''$ sums of Weyl modules
  and indecomposable projective modules and with $M^2$ equal to the
  socle of $M'$. Since $M^2 \subset M'$, the layers of the filtration
  of $M$ are semisimple, and a non-simple projective module has Loewy
  length three (this is an easy special case of \cite[Theorem
  3.17]{andersen-kaneda}), it follows that
  $M''$ is in fact a sum of (non-simple) Weyl modules. Furthermore, it
  follows from Lemma \ref{l:u} that if $N \subset M'$ is a Weyl module
  then $N \subset M^1$.
  
  By the structure of tilting modules \eqref{e:tilt} it follows that
  $M'$ contains a submodule $M'_{\dq}$ which is a direct sum of Weyl
  modules, is contained in $M^1$, and $M'/M'_{\dq}$ is a sum of Weyl
  modules. We define a new filtered module $N^{\bullet}$ by setting
  $N = M$, $N^0 = M^0 = M$, $N^1 = M'_{\dq}$ and $N^2 = \{0\}$.  It
  has the property that each layer is a sum of Weyl modules,
  $M^2 \subset N^1$ and $N^1 \subset M^1$. It follows that we may
  define a filtration $\fil^{\bullet} \,\dq(\vec{a})$ of the desired
  type by refining (so each layer is a single Weyl module or a
  projective module) and reindexing appropriately the filtrations
  $P^{\bullet}$ and $N^{\bullet}$ and then taking the direct sum.

  Now suppose $n \geq 3$ and let $\vec{a'} =
  (a_1,a_2,\dots,a_{n-1})$. We define the filtration
  $\fil^{\bullet} \,'$ on
  $\dq(\vec{a}) = \dq(\vec{a'}) \otimes \dq(a_n)$ to be the tensor
  product of the filtration $\fil^{\bullet}\, \dq(\vec{a'})$ and the
  trivial one-step filtration on $\dq(a)$, i.e., the filtration
  $\fil^{\bullet} \,\dq(a)$. The filtration
  $\fil^{\bullet} \, \dq(\vec{a'})$ will be constructed as a
  refinement and reindexing of $\fil^{\bullet}\, ' \dq(\vec{a})$, so it
  suffices to describe it on each layer of this filtration.

  If a layer of $\fil^{\bullet} \,\dq(\vec{a'})$ is a projective
  module $P$ with the filtration induced on it by
  $\dq(\vec{a'})^{\bullet}$ being a one-step filtration then we define
  the filtration on $P \otimes \dq(a)$ to be the tensor product of the
  filtration on $P$ and $\dq(a)^{\bullet}$. Since the tensor product
  of a projective module with any module is projective, it follows
  that the layers of this filtration are projective modules and the
  filtration on them induced by the filtration
  $\dq(\vec{a})^{\bullet}$ is one-step.

  If a layer of $\fil^{\bullet} \,\dq(\vec{a'})$ is a Weyl module
  $\dq(b)$ with a two step filtration (with simple layers) induced
  by $\dq(\vec{a})^{\bullet}$, then we refine
  $\fil'_{\bullet} (\dq(\vec{a}))$ by a filtration of the form
  $\fil^{\bullet} (\dq(b) \otimes \dq(a_n))$ as constructed above
  in the case $n=2$. The layers of this filtration are again of the
  desired form, so we have constructed the desired filtration
  $\fil^{\bullet} \,(\dq(\vec{a}))$.
\end{proof}

\begin{lemma}\label{l:deg}
  Let $M^{\bullet}$ be a $\uq$-module with a finite separated
  decreasing filtration with semisimple layers. If $M$ has another
  finite separated decreasing filtration $\fil^{\bullet} (M)$ whose
  layers are Weyl modules or projective modules and the filtration
  induced by $M^{\bullet}$ on each layer of $\fil^{\bullet} (M)$ is
  (at most) two-step, then the $\Ext$ spectral sequence for
  $\Ext_{\uq}^*(M,\nq (b))$ corresponding to the filtration
  $M^{\bullet}$ degenerates at $E_2$.
\end{lemma}

\begin{proof}
  We will prove that the $\Ext$ spectral sequence for
  $\Ext_{\uq}^*(\fil^i(M),\nq (b))$ degenerates at $E_2$ for all $i$
  by descending induction on $i$.  Since $\Ext_{\uq}^j(N,\nq (b)) = 0$
  for $j >0$ if $N$ is a Weyl module or a projective module we note
  that $E_{\infty}^{r,s} = 0$ for $p \neq -q$ for all these spectral
  sequences.

  For $i \gg 0$, $\fil^i(M) = \{0\}$ is zero so the statement holds. For
  general $i$, we have an exact sequence
  \begin{equation}\label{e:fil}
    0 \to \fil^{i+1}(M) \to \fil^i(M) \to \fil^i(M)/\fil^{i+1}(M) \to 0 .
  \end{equation}
  By functoriality of the $\Ext$ spectral sequence, these maps give
  rise to maps of the spectral sequences for the filtration induced by
  $M^{\bullet}$ on each term. We claim that the maps on
  the $E_1$ terms form exact sequences.

  To see this we note that the layers of the filtration $M^{\bullet}$
  are semisimple, so this also holds for the layers of all the induced
  filtrations. Therefore, the maps on the layers in \eqref{e:fil} form
  an exact sequence of semisimple modules, so the sequences must all
  split, proving the claim.

  Now by induction we know that spectral sequence for $\fil^{i+1}(M)$
  degenerates at $E_2$ and this also holds for the spectral sequence
  for $\fil^{i}(M)/\fil^{i+1}(M)$ since the filtration on the layers
  is (at most) two-step, so we must have $E_2^{r,s} = 0$ for
  $r \neq -s$ for both these spectral sequences. A simple diagram
  chase, using the exactness at $E_1$, then implies that
  $E_2^{r,s} = 0$ also for the spectral sequence associated to
  $\fil^{i}(M)$, so it must clearly also degenerate at $E_2$.
\end{proof}

\subsubsection{}\label{s:sl2deg}

\begin{proof}[Proof of Theorem \ref{t:sl2deg}]
  The proof is now easy from the foregoing lemmas: by Lemma
  \ref{l:veca2} $\dq(\vec{a})$ has a filtration of the type needed to
  apply Lemma \ref{l:deg}, so the degeneration at $E_2$ follows.
\end{proof}

\begin{remark}
  It would be very interesting to know whether the analogue of Theorem
  \ref{t:sl2deg} holds for other Lie algebras. If this is true then
  one could hope for an algorithm for computing the ranks of the
  weight filtration. But this would require, in particular, computing
  the ranks of the Andersen filtration for the layers of the
  filtration $\dq(\vec{\lambda})^{\bullet}$, but for general $\mf{g}$
  we do not at present know how to do this.
\end{remark}

\section{Pieri rules for $\mf{sl}_m$ for all fundamental weights}\label{s:pieri}

In standard notation, the lattice of integral weights for $\mf{sl}_m$
is generated by $L_1,\dots,L_m$, with the relation $\sum L_i=0$.  The
simple roots are $\alpha_i=L_i-L_{i+1}$ for $1\leq i<m$ and the
fundamental dominant weights are $\varpi_k = \sum_{1\leq i\leq k}L_i$,
$1 \leq k \leq m$. The Cartan subalgebra of $\mf{gl}_m$ is freely
generated by $\epsilon_1,\dots,\epsilon_m$ with dual elements
$L_1,\dots,L_m$.  The Cartan subalgebra of $\mf{sl}_m$ is the subspace
given by the vanishing of $\sum L_i$.

\smallskip

Let $\lambda$ be a dominant integral weight for $\mf{sl}_m$. Then 
$V_{\lambda}\tensor V_{\vpi_k}$ is a multiplicity-free sum of all
$V_{\mu}$ such that $\mu=\lambda+\sum_{i\in I}L_i$ for some
$I\subset\{1,\dots,m\}$ with $|I|=k$, provided that $\mu$ is dominant.
We are interested in the Hodge polynomial of
${KZ}_{\kappa}((\lambda,\vpi_k),\mu^*)$,  where $\kappa = K/d$ is a positive rational number with
$(K,d)=1$. We assume $K>1$, since the case $K=1$ was considered in
\cite[Section 11.4]{BFM}; in that case all the relevant Hodge
polynomials are equal to the constant polynomial $1$.

For $k=1$, an explicit formula for the Hodge polynomial was given in
\cite[Theorem 1.11]{BFM}, and it was shown there that the $k=1$ case determines the
other Pieri Hodge polynomials. In Theorem~\ref{t:rational-pieri} we
give an explicit formula for arbitrary $k$. The corresponding weight
formula \eqref{WeightFormula} will play a key role in the proof of
Theorem~\ref{t:G}.

The formula involves the Hodge polynomials of Aomoto motives. For
$a\geq0$, let $[a;\kappa]$ denote the rank one Aomoto motive as in
\cite[Definition 3.13]{BFM}. We introduce the following notation for its Hodge
polynomial, with a shift.
\begin{defi}
For $H\geq1$, let
\begin{equation*}
  Q_\kappa(H)= \mr{P}([H-1;\kappa]).
\end{equation*}
\end{defi}

For any integer $x$, let $[x]_K\in\{0,\dots,K-1\}$ denote its
remainder modulo $K$, and put $\delta=[d]_K$. A direct computation
using \cite[Lemma 11.5]{BFM} gives that for all $H \geq 1$, 
\begin{equation}\label{Qdef}
Q_\kappa(H)=
\begin{cases}
1,&1\leq [\delta H]_K\leq \delta,\\
t,&\text{otherwise.}
\end{cases}
\end{equation}
We then use \eqref{Qdef} to define $Q_\kappa(H)$ for all integers $H$.

\subsection{The Pieri formula}\label{s:defN}

Write $\lambda=(a_1,\ldots,a_m)=\sum a_iL_i$, with
$a_1\geq\cdots\geq a_m$. For $I\subset\{1,\ldots,m\}$, call $I$
\emph{admissible} if $\lambda+\sum_{i\in I}L_i$ is weakly decreasing,
and hence defines a dominant weight for $\mf{sl}_m$. Define
$\Inv(I)=\{(i,j):i<j,\ i\notin I,\ j\in I\}$.

For $i<j$, let $\alpha_{ij}=L_i-L_j$ and put
\begin{equation}\label{hdef}
h_{ij}(\lambda)=\langle\lambda+\rho,\alpha_{ij}^{\vee}\rangle
=a_i-a_j+j-i.
\end{equation}
For $r\in\Z/K\Z$, put
$N_r^{(K)}(\lambda,I)=
\#\{(i,j)\in\Inv(I):h_{ij}(\lambda)\equiv r\pmod K\}$.

\begin{theorem}\label{t:rational-pieri}
Let $I$ be admissible, $|I|=k$, and
$\mu=\lambda+\sum_{i\in I}L_i$. Then the Hodge polynomial of
${KZ}_{\kappa}((\lambda,\vpi_k),\mu^*)$ is
\begin{equation}\label{PieriFormula}
\prod_{(i,j)\in\Inv(I)}
Q_\kappa\!\left(\langle\lambda+\rho,\alpha_{ij}^{\vee}\rangle\right).
\end{equation}
Moreover, ${KZ}_{\kappa}((\lambda,\vpi_k),\mu^*)$ is one dimensional of weight
\begin{equation}\label{WeightFormula}
|\Inv(I)|+N_0^{(K)}(\lambda,I)-N_1^{(K)}(\lambda,I).
\end{equation}
Note that by a simple calculation,
$|\Inv(I)|=\ExpM((\lambda,\vpi_k),\mu^*)$.
\end{theorem}

\begin{proof}[Proof of the weight assertion from \eqref{PieriFormula}]
  At the complex conjugate embedding, $d$ is replaced by $K-d$ modulo
  $K$.  By a simple computation,
\begin{equation}\label{Qcdef}
Q_\kappa^c(H)=
\begin{cases}
t,&0\leq [\delta H]_K\leq\delta-1,\\
1,&\text{otherwise}.
\end{cases}
\end{equation}
Put $H=h_{ij}(\lambda)$. By \eqref{Qdef} and \eqref{Qcdef}, the
contribution of $(i,j)$ to the weight is $2$ if $[\delta H]_K=0$, is
$0$ if $[\delta H]_K=\delta$, and is $1$ otherwise. Since
$(\delta,K)=1$, these cases are respectively $H\equiv0\pmod K$,
$H\equiv1\pmod K$, and $H\not\equiv0,1\pmod K$. Summing over
$\Inv(I)$ gives \eqref{WeightFormula}.
\end{proof}

\subsection{The universal $Q_\kappa$ identity}
Let $\kappa=K/d,K>1,d,\delta$ be as above.
\begin{theorem}\label{Qidentity} 
Let integers $A_{pq}$, $1\leq p<q\leq N$, satisfy
\begin{equation*}%
A_{pr}=A_{pq}+A_{qr}\qquad (p<q<r).
\end{equation*}
Then,
\begin{equation}\label{e:S}
\sum_{r=1}^N
\prod_{p<r}Q_\kappa(A_{pr}+1)
\prod_{q>r}Q_\kappa(A_{rq})
=
\sum_{r=1}^N
\prod_{p<r}Q_\kappa(A_{pr})
\prod_{q>r}Q_\kappa(A_{rq}+1).
\end{equation}
\end{theorem}

\begin{proof}
Set $x_N=0$
and $x_p=A_{pN}$ for $p<N$. Additivity gives $A_{pq}=x_p-x_q$.
Put $z_r=[\delta x_r]_K\in\Z/K\Z$; the $z_r$ may be regarded as
points on a residue circle, and $[\delta A_{pq}]_K=z_p-z_q$.

Let
\begin{equation}\label{e:LR}
\mr{L}_r=\prod_{p<r}\Qk(A_{pr}+1)\prod_{q>r}\Qk(A_{rq}),
\qquad
\mr{R}_r=\prod_{p<r}\Qk(A_{pr})\prod_{q>r}\Qk(A_{rq}+1).
\end{equation}
Let $a_r$ and $b_r$ be the numbers of factors equal to $1$ in $\mr{L}_r$
and $\mr{R}_r$. Then $\mr{L}_r=t^{N-1-a_r}$ and $\mr{R}_r=t^{N-1-b_r}$.
With $u=t^{-1}$, it is enough to prove
$\sum_r u^{a_r}=\sum_r u^{b_r}$.

For $c\in\Z/K\Z$, let $I_c=\{c-\delta+1,\ldots,c\}$ be the cyclic
interval of $\delta$ residues ending at $c$. For $0\le j\le N$, place
a cut after the first $j$ indices, and define
\begin{equation*}%
H_c(j)=\#\{p\le j:z_p\in I_c\}
       \ + \ \#\{q>j:z_q\in I_{c-1}\}.
\end{equation*}
Thus points to the left of the cut are counted in $I_c$, and points to
the right are counted in the neighboring interval $I_{c-1}$. 

If $z_r=c$, the definition of $\Qk$ shows that the factors equal to
$1$ in $\mr{L}_r$ are precisely those counted by $H_c(r-1)$; hence
$a_r=H_c(r-1)$. Similarly, if $z_r=c-\delta$, the factors equal to
$1$ in $\mr{R}_r$ are precisely those counted by $H_c(r)$; hence
$b_r=H_c(r)$.

As the cut passes $z_r$, this point moves from the right-hand count to
the left-hand count. The intervals $I_c$ and $I_{c-1}$ differ only at
their endpoints, so $H_c$ rises by $1$ when $z_r=c$, falls by $1$
when $z_r=c-\delta$, and otherwise does not change.

 Put
$D_r(c)=u^{H_c(r)}-u^{H_c(r-1)}$. Then  $\sum_{r=1}^N D_r(c)=u^{H_c(N)}-u^{H_c(0)}$. 
Let us look only at terms $D_r(c)$ which are non-zero. Note that
\[
D_r(c)\neq 0 \quad\Longleftrightarrow\quad H_c(r)\neq H_c(r-1)
\quad\Longleftrightarrow\quad
z_r=c\ \text{or}\ z_r=c-\delta.
\]

If $z_r=c$, then   $D_r(c)=u^{a_r+1}-u^{a_r}=(u-1)u^{a_r}$, and if 
$z_r=c-\delta$, then  $D_r(c)=u^{b_r}-u^{b_r+1}=-(u-1)u^{b_r}$. Therefore the same telescoping sum can be written as
\begin{equation}\label{e:telescoping}
(u-1)\left(\sum_{z_r=c}u^{a_r}-\sum_{z_r=c-\delta}u^{b_r}\right)
=u^{H_c(N)}-u^{H_c(0)}.
\end{equation}

Now $H_c(N)$ counts the points in $I_c$, while $H_c(0)$ counts the
points in $I_{c-1}$; in other words, $H_c(0)=H_{c-1}(N)$. Summing
\eqref{e:telescoping} over $c$ makes the right-hand
side telescope to zero and hence
$(u-1)(\sum_r u^{a_r}-\sum_r u^{b_r})=0$ which gives \eqref{e:S}.

\end{proof}

\subsection{Proof of Theorem \ref{t:rational-pieri}}

Define
\begin{equation}\label{Cdef}
C_\kappa(\lambda,I)=
\prod_{(i,j)\in\Inv(I)}Q_\kappa(h_{ij}(\lambda))
\end{equation}
and define $\Z[t]$ linear operators $T_k$ on
$\mf{R}_{\kappa}(\mf{sl}_m)$ by
\begin{equation}\label{Tkdef}
T_k[\lambda]=
\sum_{\substack{|I|=k\\ I\text{ admissible}}}
C_\kappa(\lambda,I)
\left[\lambda+\sum_{i\in I}L_i\right].
\end{equation}

\begin{proposition}\label{switch}
For $1\leq k\leq m-1$, $T_1\circ T_k=T_k\circ T_1$.
\end{proposition}

\begin{proof}
  We will show that the coefficients of any dominant weight $[\nu]$ in
  $T_1\circ T_k[\lambda]$ and $T_k\circ T_1[\lambda]$ agree.

Fix a dominant weight
$\nu=\lambda+\sum_{s\in S}L_s$, where
$S=\{s_1<\cdots<s_N\}$ and $N=k+1$, and first suppose that the
$k+1$ added boxes lie in distinct rows. For $r=1,\ldots,N$, put
$I_r=S\setminus\{s_r\}$.

The factors involving a row outside $S$ are the same in the two orders
and for every $r$; their product is
$C_S'=
\prod_{\substack{i\notin S,\ j\in S\\ i<j}}
Q_\kappa(h_{ij}(\lambda))$.
For $p<q$, put
\begin{equation}\label{Apq}
A_{pq}=h_{s_p,s_q}(\lambda)
=\langle\lambda+\rho,\alpha_{s_p,s_q}^{\vee}\rangle.
\end{equation}
Then $A_{pr}=A_{pq}+A_{qr}$ for $p<q<r$.

After removing the common factor $C_S'$, the contribution of
the history in which $T_k$ is applied first and the $T_1$-box is added
in row $s_r$ is
$\mr{L}_r=\prod_{p<r}Q_\kappa(A_{pr}+1)\prod_{q>r}Q_\kappa(A_{rq})$,
whereas the contribution of the reverse history is
$\mr{R}_r=\prod_{p<r}Q_\kappa(A_{pr})\prod_{q>r}Q_\kappa(A_{rq}+1)$.
Theorem~\ref{Qidentity} therefore gives
\begin{equation}\label{formal}
\sum_r\mr{L}_r=\sum_r\mr{R}_r.
\end{equation}

It remains to remove the histories for which the intermediate weight is
not dominant. A subset $I$ is admissible precisely when there is no
$i$ such that
\begin{equation}\label{bad}
a_i=a_{i+1},\qquad i\notin I,\qquad i+1\in I.
\end{equation}
Since the final set $S$ is admissible, the history on the left indexed
by $r$ fails to be admissible precisely when
\begin{equation}\label{badpair}
s_{r+1}=s_r+1,\qquad a_{s_r}=a_{s_{r+1}}.
\end{equation}
In this case the history on the right indexed by $r+1$ is also
non-admissible. Moreover,
$A_{r,r+1}=\langle\lambda+\rho,\alpha_{s_r}^{\vee}\rangle=1$, so that
$A_{p,r+1}=A_{pr}+1$ and $A_{r+1,q}=A_{rq}-1$. It follows that
$\mr{L}_r=\mr{R}_{r+1}$. Thus the forbidden histories occur in equal pairs on the
two sides of \eqref{formal}; deleting them gives the equality of the
admissible contributions.

It remains to consider a final weight for which one row receives both
the $T_k$-box and the $T_1$-box. Let
$J\subset\{1,\ldots,m\}$, $|J|=k$, let $s\in J$, and suppose that
$\nu=\lambda+\sum_{j\in J}L_j+L_s$ is dominant. We compare the sequence
in which $T_k$ is applied first with the sequence in which $T_1$ is
applied first.

If $(i,j)\in\Inv(J)$ and $j\neq s$, then $i\neq s$ as well, and hence
$h_{ij}(\lambda+L_s)=h_{ij}(\lambda)$. Thus these inversions contribute
the same factors in the two orders.

For $i<s$, put
$H_i=\langle\lambda+\rho,\alpha_{is}^{\vee}\rangle$. In the first
order, the $T_k$-move contributes
$\prod_{i<s,\ i\notin J}Q_\kappa(H_i)$ from inversions ending at $s$,
while the subsequent $T_1$-move contributes
$$\prod_{i<s,\ i\in J}Q_\kappa(H_i)
\prod_{i<s,\ i\notin J}Q_\kappa(H_i-1).$$ Thus the total contribution
of the factors involving $s$ is
$$\prod_{i<s,\ i\in J}Q_\kappa(H_i)
\prod_{i<s,\ i\notin J}Q_\kappa(H_i)Q_\kappa(H_i-1).$$

In the reverse order, the initial $T_1$-move contributes
$\prod_{i<s}Q_\kappa(H_i)$, while the subsequent $T_k$-move contributes
$\prod_{i<s,\ i\notin J}Q_\kappa(H_i-1)$ from inversions ending at
$s$. The product is the same. 

Finally, both intermediate weights are dominant: Since $s\in J$,
row $s$ receives two boxes in $\nu$, while every other row receives
at most one. For $\lambda+L_s$, the only possible failure of dominance is between
rows $s-1$ and $s$. If $s>1$, row $s-1$ has length at most
$a_{s-1}+1$ in $\nu$, whereas row $s$ has length $a_s+2$.
Dominance of $\nu$ therefore gives
$a_{s-1}+1\ge a_s+2,$ so 
$a_{s-1}\ge a_s+1$, as required. If $s=1$, there is nothing
to check. For $\lambda+\sum_{j\in J}L_j=\nu-L_s$, removing a box from row $s$
of $\nu$ can violate dominance only between rows $s$ and $s+1$.
If $s<m$, the new length of row $s$ is $a_s+1$, while row $s+1$
has length at most $a_{s+1}+1$. Since $\lambda$ is dominant,
$
a_s+1\ge a_{s+1}+1
$( if $s=m$, there is nothing to check.).

This proves the
switching identity also when one row receives two boxes.
\end{proof}

The Pieri-coefficients appearing in the definition of the $T_k$ (See \eqref{Tkdef}) have the following properties:
\begin{enumerate}
\item If the last row is not in $I$, then no inversion in $\Inv(I)$ involves the last row, 
and the coefficient is the corresponding coefficient obtained by deleting that
row and passing from $\mf{sl}_m$ to $\mf{sl}_{m-1}$. 
\item If the first row
belongs to $I$, then no inversion in $\Inv(I)$ involves the first row,
and deleting that row gives the corresponding coefficient in rank
$m-1$, now with $k$ replaced by $k-1$. 
\item It has  opposite-end symmetry i.e., the
following duality for $\mf{sl}_m$: reverse the rows and replace
$\lambda=(a_1,\ldots,a_m)$ by
$\lambda^*=(-a_m,\ldots,-a_1)$. Under this operation $\vpi_k$ is
replaced by $\vpi_{m-k}$, and the selected rows are replaced by the
complement of the reflected set $\{m+1-i:i\in I\}$. The corresponding
Pieri coefficient is unchanged.
\end{enumerate}

\begin{lemma}\label{unique}
  Let $U_k$ be $\Z[t]$ linear operators $T_k$ on
  $\mf{R}_{\kappa}(\mf{sl}_m)$ such that, for every dominant
  $\lambda$, the terms in $U_k[\lambda]$ with non-zero coefficients
  are exactly ``Pieri-type'' dominant weights
  $\lambda+\sum_{i\in I}L_i$ with $|I|=k$. Assume that $U_1=T_1$, that
  the coefficients are compatible with the standard type $A$ end-rank
  reductions and the opposite-end symmetry (these properties are
  stated above), and that $U_1\circ U_k=U_k\circ U_1$. Then $U_k$ is
  uniquely determined.
\end{lemma}

\begin{proof}
We argue by induction on the number of rows, then on $|\lambda|$, and
then on the position of the vertical strip. If an end-rank reduction
applies, the coefficient is determined in smaller rank. Otherwise,
using the opposite-end symmetry if necessary, remove a box from
$\lambda$ so that the resulting weight $\lambda'$ is dominant and the
required term occurs in the coefficient comparison obtained from
$(U_k\circ U_1)[\lambda']=(U_1\circ U_k)[\lambda']$.

In this comparison the desired coefficient occurs multiplied by a
one-box coefficient of $U_1=T_1$, hence by either $1$ or $t^e$ for
some $e>0$. Every other coefficient is known by induction, either in
smaller rank, for a smaller weight, or for an earlier position of the
vertical strip.  Thus the required coefficient is uniquely
determined. Repeating the argument determines all coefficients of
$U_k$.
\end{proof}  

Let $\mathsf M_k$ denote the operator of multiplication by $[\varpi_k]$ in
$\mf{R}_{\kappa}(\mf{sl}_m)$.  For every dominant $\lambda$, the
possible terms of $T_k[\lambda]$ with non-zero coefficients are exactly
the dominant weights $\lambda+\sum_{i\in I}L_i$ with $|I|=k$. The
operators $T_k$ have the correct $k=1$ coefficients and the required
end-rank compatibilities, and by Proposition~\ref{switch} they commute
with $T_1$. By associativity, the same properties hold for
$\mathsf M_k$.  Lemma~\ref{unique} therefore gives $T_k=\mathsf
M_k$. Equation \eqref{Cdef} is precisely \eqref{PieriFormula}, proving
Theorem~\ref{t:rational-pieri}. \qed

\begin{remark}
It would be interesting to know whether the Pieri three-point motive
$M_\kappa(\lambda,I)$ can be related motivically to the one-dimensional
Aomoto motives occurring in \eqref{PieriFormula}. Similarly, it would
be interesting to find a motivic explanation of \eqref{e:S}. In both
cases, the equality of the corresponding Hodge polynomials for all
Galois conjugates is known.
\end{remark}

\section{The first moment weight deformation ring}\label{s:fm}

The main goal of this section is the proof of Theorem \ref{t:G}, which
gives an explicit formula for the product $\diamond_{\kappa}$ of the
ring $\mfrfm_{\kappa}(\mf{sl}_m)$ . The proof depends on comparing the
products in the ring $\mf{R}_{\kappa}(\mf{sl}_m)$ (as well as the
``complex conjugate'' ring $\mf{R}_{\kappa^{\dagger}}(\mf{sl}_m)$
defined below) with that of $\mf{R}(\mf{sl}_m) \otimes \Z[t]$ up to
first order.  A key ingredient is the Pieri formula for weights
\eqref{WeightFormula}.

In the case of $\mf{sl}_2$, multiplicity one implies that computing
$\diamond_{\kappa}$ is equivalent to computing the weight of
$KZ_{\kappa}((\lambda,\mu),\nu^*)$ for all
$\lambda,\mu,\nu$. Proposition \ref{p:A1-floor} gives a simple formula
for these numbers.

At the end of the section we give a conjectural analogue of Theorem
\ref{t:G} for arbitrary simple $\mf{g}$.

\subsection{Associativity of the first moment ring}\label{s:assoc}

We begin by proving the associativity of the product
$\diamond_{\kappa}$ of the ring $\mfrfm_{\kappa}(\mf{g})$ for general $\mf{g}$.

The associativity of the product ${\star}_{\kappa}$ of the ring
$\mf{R}_{\kappa}(\mf{g})$ was proved in \cite[Theorem 1.4]{BFM} using the main
motivic factorisation result  \cite[Theorem 6.1, Section 10.2]{BFM} and the fact that the Hodge polynomial is
preserved by nearby cycles. The weight polynomial itself is not
preserved by nearby cycles, but its first moment is. This follows from
Schmid's  theorem on the limiting mixed Hodge structure \cite[Theorem 6.16 and Lemma 6.4] {schmid} applied to the layers of the weight filtration, together with a simple
computation: the weight polynomial is preserved modulo $(w-1)^2$ under
nearby cycles. Indeed, for every $m$, $(w-1)^2$ divides $w^m+w^{-m}-2$
in $\Z[w,w^{-1}]$. The associativity of
$\diamond_{\kappa}$ for arbitrary $\kappa \in \Q^{\times}$ then
follows using the same argument as in \cite[Section 10.2]{BFM}.

Alternatively, if $\kappa\ge2$ is an integer, let
$\kappa^{\dagger}=\frac{\kappa}{\kappa-1}$. Then
$KZ_{\kappa^{\dagger}}(\vec{\lambda},\nu^*)$ is the complex conjugate
of the VMHS $KZ_{\kappa}(\vec{\lambda},\nu^*)$, so we have
\begin{equation}\label{e:Wdef}
FM^{\mr{W}}_\kappa(\lambda,\mu,\nu^*)
=\left(P_\kappa^{\lambda,\mu,\nu^*}\right)'(1)
 +\left(P_{\kappa^{\dagger}}^{\lambda,\mu,\nu^*}\right)'(1).
\end{equation}
This gives another proof that the product $\diamond_{\kappa}$ is
associative. Indeed, we have two associative deformations (for
$\kappa$ and for $\kappa^{\dagger}$) of the representation ring over
$\mathbb{Z}[t]$, hence over $A=\mathbb{Z}[t]/(t-1)^2$ (with
$\epsilon=t-1$). The product $\diamond_{\kappa}$ defined above is
obtained by summing their first-order deformations over $A$; this
makes sense since these first-order deformations form an abelian
group.

\begin{remark}\label{r:kappa}
  From now on we will mostly assume that $\kappa > 1$ is an
  integer. This does not lead to any loss of generality when
  considering numerical properties of the weight filtration for simply
  laced $\mf{g}$, in particular for $\mf{sl}_m$: by the discussion at
  the end of \cite[Section 1.2]{BFM} and beginning of \cite[Section
  1.3]{BFM}, up to Galois conjugation of the coefficients which
  (preserves the weight filtration) the MHS associated to the stalk at
  any point of $\mc{C}_n$ any $KZ_{\kappa}(\vec{\lambda},\nu^*)$ is
  isomorphic to the stalk at the same point of
  $KZ_{\kappa'}(\vec{\lambda},\nu^*)$ with $\kappa'$ an integer. Also,
  as noted earlier, the case $\kappa = 1$ is already dealt with in
  \cite{BFM}.  We will drop this assumption in Section \ref{ss:arb}
  where we consider arbitrary $\mf{g}$.
\end{remark}

\subsection{The first order basis change identity}

Let $\kappa>1$ be a positive integer. As noted in \cite[Lemma
10.9]{BFM}, every $\mf{R}_{\kappa}(\mf{g})$ is a polynomial ring over
$\mathbb Z[t]$ freely generated by all the fundamental classes
$[\varpi_i]$.  Therefore there is a unique fundamental class
preserving $\Z[t]$-algebra isomorphism
\begin{equation*}%
\Phi_\kappa:\mf{R}(\slg_m)\otimes_\mathbb Z\mathbb Z[t]
\stackrel{\sim}{\longrightarrow}\mf{R}_{\kappa}(\slg_m),
\qquad
\Phi_\kappa([\varpi_i])=[\varpi_i].
\end{equation*}
Write $\epsilon=t-1$ and expand
\[
\Phi_\kappa(f)=f+\epsilon\,b_\kappa(f)+O(\epsilon^2).
\]
Put
\[
  B=b_\kappa+b_{\kappa^{\dagger}} .
\]

Note that $B:\mf{R}(\slg_m)\to \mf{R}(\slg_m)$. Using the isomorphism \eqref{chara}, this induces a linear map
on the character ring, $B_{\op{char}}:\mathbb{Z}[P]^W\to \mathbb{Z}[P]^W$. 

\begin{defi}%
If $Z\in\mathbb{Z}[P]^W$, write it in the irreducible-character basis as
$
Z=\sum_{\eta\in P^+} a_\eta\,\chi_\eta$.
Then let $[\chi_\nu]Z:=a_\nu$. Equivalently, for the standard inner product on the character ring for which
the irreducible characters are orthonormal, $[\chi_\nu]Z=(\chi_\nu,Z)$.
\end{defi}

Using the fact that $\Phi_\kappa$ and $\Phi_{\kappa^{\dagger}}$ preserve products, a simple computation gives
\begin{lemma}
\begin{equation}\label{e:W-B}
FM^{\mr{W}}_\kappa((\lambda,\mu),\nu^*)
=[\chi_\nu]\{B_{\operatorname{char}}(\chi_\lambda\chi_\mu)
-B_{\operatorname{char}}(\chi_\lambda)\chi_\mu
-\chi_\lambda B_{\operatorname{char}}(\chi_\mu)\}.
\end{equation}
\end{lemma}

\subsection{Identification of the basis-change operator $B_{\operatorname{char}}$}\label{s:id}

For any integral weight $\xi$, define the
generalised Weyl character $\chi(\xi)$ by 
$\chi(\xi)=0$ if $\xi$ is dot-singular, and otherwise $
\chi(\xi)=(-1)^{\ell(w)}\chi_{w\cdot\xi}$,
where $w \in W$ is the unique element such that $w\cdot\xi$ is
dominant.

Using this definition, set
\begin{equation}\label{e:J}
\Jop(\chi_\lambda)
=\sum_{\alpha\in R^+}
 \sum_{\substack{r\ge1\\r\kappa<\langle\lambda+\rho,\alpha^\vee\rangle}}
 \chi(\lambda-r\kappa\alpha).
\end{equation}
This extends linearly to the character ring. Define also 
\[
\Sop(\chi_\lambda)=s(\lambda)\chi_\lambda,
\qquad
\Fop=\Sop+\Jop.
\]
The operator $\Sop$ also extends linearly to the character ring. Note
that $\Sop(\chi_0)=\Jop(\chi_0)=0$.

Since
$
\mathbb{Z}[P]^W
=\mathbb Z[\chi_{\varpi_1},\ldots,\chi_{\varpi_{m-1}}]$,
there is a unique derivation $\partial_s: \mathbb{Z}[P]^W \to \mathbb{Z}[P]^W$ satisfying
\[
  \partial_s(\chi_{\varpi_i})=s(\varpi_i)\chi_{\varpi_i}.
\]

\begin{proposition}\label{p:B}

\begin{equation}\label{e:Bformula}
{\ B_{\operatorname{char}}=-\bigl(\Sop+\Jop-\partial_s\bigr) = -\Fop+\partial_s.\ }
\end{equation}
\end{proposition}

\begin{proof}
  Let $B_0$ be the right-hand side of \eqref{e:Bformula}. By
  Proposition~\ref{p:minJ} below, both $B_{\operatorname{char}}$
  and $B_0$ vanish on all fundamental generators.  Let
  $x_i=\chi_{\varpi_i}$. By the Pieri formula \eqref{WeightFormula},
  Proposition~\ref{p:minJ}, and equation~\eqref{soppy},
\[
B_{\operatorname{char}}(fx_i)-B_{\operatorname{char}}(f)x_i
=B_0(fx_i)-B_0(f)x_i.
\]
Hence $C=B_{\operatorname{char}}-B_0$ satisfies
\[
C(fx_i)=C(f)x_i
\]
for every $i$, and $C(1)=0$.  Induction on monomials in the polynomial
generators gives $C=0$.
\end{proof}

\subsection{Some computations in the character ring}

 For an integral weight $x\in P$, put
\[
A_x=\sum_{w\in W}(-1)^{\ell(w)}e^{wx},
\qquad \Delta=A_\rho.
\]
For dominant $\lambda$ let
\[
  \chi_\lambda=A_{\lambda+\rho}/\Delta.
\]

If $\omega=\varpi_k$ is minuscule, then
\begin{equation}\label{e:minuscule-alternant}
\chi_\omega=\sum_{\zeta\in W\omega}e^\zeta,
\qquad
A_x\chi_\omega=\sum_{\zeta\in W\omega}A_{x+\zeta}.
\end{equation}
Note that if $x+\zeta$ above is not dominant then $A_{x+\zeta}=0$.

The following is standard; see \cite{Jantzen} and \cite[Theorem 4.1, Corollary 4.2, and Cor.~5.1]{AK}.  

\begin{lemma}\label{lem:singular} 
For $x$ in the closed dominant chamber, define
\[
\mathbf J_\kappa(A_x)
=\sum_{\alpha\in R^+}
 \sum_{\substack{r\ge1\\0<r\kappa<\langle x,\alpha^\vee\rangle}}
 A_{x-r\kappa\alpha}.
\]
If $x$ is dot-singular, then $\mathbf J_\kappa(A_x)=0$.
\end{lemma}

\begin{proof}
Choose a simple reflection $s_i$ fixing $x$.  It permutes
$R^+\setminus\{\alpha_i\}$.  Terms indexed by $\beta$ and $s_i\beta$ have
identical $r$-ranges and 
\[
A_{x-r\kappa s_i\beta}=A_{s_i(x-r\kappa\beta)}
=-A_{x-r\kappa\beta}.
\]
A root fixed by $s_i$ contributes zero, while
the $\alpha_i$-sum is empty because
$\langle x,\alpha_i^\vee\rangle=0$.
\end{proof}

 Recall that in Section \ref{s:defN}, we  defined the quantities
$N_r^{(\kappa)}(\lambda,I)$ for $r\in\Z/\kappa\Z$ (note that in this section $\kappa>1$ is an integer). Here $\lambda$ is dominant and $I\subseteq \{1,\ldots,m\}$ is
such that  $\mu=\lambda+\sum_{a\in I}L_a$ is weakly decreasing
(and hence defines a dominant weight for $\mf{sl}_m$ )which appears in the Pieri product  of $[\lambda]$ and  $[\varpi_k]$ with $k=|I|$.

If $\alpha=L_i-L_j$ is a positive root in $\Inv(I)$ (i.e., $i\not\in I$ and $j\in J$) then it contributes a $1$ to  
$N_r^{(K)}(\lambda,I)$ if 
$\langle\lambda+\rho,\alpha_{ij}^{\vee}\rangle \equiv r\pmod \kappa$ and zero otherwise.

\begin{proposition}\label{p:minJ}
For $\omega=\varpi_k$ and dominant $\lambda$,
\begin{equation}\label{e:minJ}
(\Jop\chi_\lambda)\chi_\omega
-\Jop(\chi_\lambda\chi_\omega)
=\sum_{\substack{|I|=k\\ I\text{ admissible}}}
\bigl(N^{(\kappa)}_0(\lambda,I)-N^{(\kappa)}_1(\lambda,I)\bigr)\chi_{\lambda+\sum_{a\in I}L_a}.
\end{equation}

Moreover $\Jop(\chi_{\varpi_k})=0$.
\end{proposition}

\begin{proof}
Put $x=\lambda+\rho$.  For  $\zeta\in W\omega$ and $\alpha=L_i-L_j\in R^+$ write 
\[
h=\langle x,\alpha^\vee\rangle,
\qquad e=\langle\zeta,\alpha^\vee\rangle\in\{-1,0,1\}.
\]
Note that  such $\zeta$ is of the form $\sum_{a\in I}L_a$ with $|I|=k$. This gives a one-one correspondence between the set of $\zeta$ and  $I$ with $|I|=k$. 

Then $e=-1$ if $\alpha\in \Inv(I)$, it is $1$ if $i\in I$ and $j\not\in I$ and zero otherwise.
If $i\in I$ and $j\not\in I$ then $\alpha\in \Inv(I\cup\{j\}-\{i\})$.

Using \eqref{e:minuscule-alternant} and Lemma~\ref{lem:singular},
\begin{align*}
&\Delta\bigl[(\Jop\chi_\lambda)\chi_\omega
-\Jop(\chi_\lambda\chi_\omega)\bigr]\\
&\qquad=
\sum_{\alpha>0}\sum_{\zeta\in W\omega}
\left(\sum_{0<r\kappa<h}-\sum_{0<r\kappa<h+e}\right)
A_{x+\zeta-r\kappa\alpha}.
\end{align*}
We will collect the various $A_{x+\zeta-r\kappa\alpha}$ that appear with non-zero coefficient.
We want these to be of the ``Pieri-form'' $A_{x+\sum_{a\in I}L_a}$ with $|I|=k$ and with the 
desired coefficient.  

\smallskip

If $e=0$, the ranges agree.  If $e=-1$, the first range has one extra term
exactly when $h-1=r\kappa$; that term equals
\[
A_{x+\zeta-(h-1)\alpha}=A_{s_\alpha(x+\zeta)}=-A_{x+\zeta}.
\]
The corresponding term is thus of ``Pieri-form'' and  receives a coefficient of $-1$ exactly when
$h\equiv1\pmod\kappa$. Note that in this case $\alpha\in \Inv(I)$ and contributes $1$ to $N^{(\kappa)}_1(\lambda,I)$ where $\zeta=\sum_{a\in I} L_a$.

 If $e=1$, the second range has one extra term
exactly when $h=r\kappa$.  With
$\zeta'=s_\alpha\zeta=\zeta-\alpha$ one gets
\[
-A_{x+\zeta-h\alpha}=+A_{x+\zeta'},
\]
Let $\zeta=\sum_{i\in I}L_i$, then  $\zeta'$ is  $\sum_{a\in I'} L_a$ with $I'=I\cup\{j\}-\{i\}$, 
so the corresponding term is also  of ``Pieri-form'' receives a coefficient $+1$ exactly when
$h\equiv0\pmod\kappa$.  In this case 
$\alpha\in \Inv(I')$ and contributes $1$ to $N^{(\kappa)}_0(\lambda,I')$.

Summing over roots gives $N_0-N_1$, proving
\eqref{e:minJ}.  If $\lambda=0$ there is  only one term in the RHS of \eqref{e:minJ}, and
this gives $\Jop(\chi_{\varpi_k})=0$.
\end{proof}

\subsection{An adjointness statement}
Define
\begin{equation}\label{e:Aop}
\Aop(\chi_\nu)
=\sum_{\alpha\in R^+}
 \sum_{\substack{t\in\mathbb Z\\
 t<0\ \mathrm{or}\ t\kappa>\langle\nu+\rho,\alpha^\vee\rangle}}
 \chi(\nu-t\kappa\alpha).
\end{equation}
Equip the character ring with the orthonormal pairing
$(\chi_\lambda,\chi_\mu)=\delta_{\lambda\mu}$.
See \cite[Prop.~4.9]{AK} for a statement equivalent  to the following.
\begin{proposition}\label{p:adjoint}
\begin{equation*}%
(\chi_\nu,\Jop\chi_\eta)
=(\Aop\chi_\nu,\chi_\eta).
\end{equation*}
\end{proposition}

\begin{proof}
  Non-zero terms $\chi(\eta-r\kappa\alpha)$ of $\Jop\chi_\eta$
  contributing to the left hand side are such that
  $\eta+\rho-r\kappa\alpha$ and $\nu+\rho$ are in the same orbit of
  the finite Weyl group.  Write the sum in $\Aop\chi_\nu$ as
  parametrised by positive roots $\beta$. Non-zero terms
  $\chi(\nu-t\kappa\beta)$ of $\Aop\chi_\nu$ contributing to the right
  hand side are such that $\nu+\rho -t\kappa\beta$ and $\eta+\rho$ are
  in the same orbit of the finite Weyl group.  The corresponding terms
  $\eta-r\kappa\alpha$ and $\nu-t\kappa\beta$ are dot-regular.

Consider now a non-zero term $\chi(\eta-r\kappa\alpha)$ of $\Jop\chi_\eta$ contributing to the left hand side:
Here $r\geq1$ and $r\kappa<\langle\eta+\rho,\alpha^\vee\rangle$. Now
choose $w\in W$ with
\[
w(\eta+\rho-r\kappa\alpha)=\nu+\rho
\]
If $w\alpha>0$, put $\beta=w\alpha$ and $t=-r$.  
Then
$w^{-1}(\nu+\rho-t\kappa\beta)=\eta+\rho$.

Note that $t<0$ and hence we get an allowed term $\chi(\nu-t\kappa\beta)$ to  of $\Aop\chi_\nu$ contributing to the right hand side.
The finite Weyl signs term are  the same on the two contributions. Also note that the assumption $r\kappa<\langle\eta+\rho,\alpha^\vee\rangle$ is automatic here because $\langle\eta+\rho,\alpha^\vee\rangle= \langle \nu+\rho-t\kappa\beta,\beta^{\vee}\rangle$ and $t<0$.

  If $w\alpha<0$, put $\beta=-w\alpha>0$ and $t=r$.  Then, we have again
  $w^{-1}(\nu+\rho-t\kappa\beta)=\eta+\rho$.
We need to show  $t\kappa>\langle\nu+\rho,\beta^\vee\rangle$. But 
$$\langle\nu+\rho,\beta^\vee\rangle=\langle w(\eta+\rho-r\kappa\alpha),\beta^{\vee}\rangle=-\langle \eta+\rho,\alpha^{\vee}\rangle+2r\kappa> -r\kappa+2r\kappa=r\kappa=t\kappa$$

 The construction is reversible in both cases. Note that if $t<0$, and $w^{-1}(\nu+\rho-t\kappa\beta)=\eta+\rho$, then necessarily $w^{-1}\beta$ is a positive root, since we can evaluate the previous equation on $w^{-1}\beta^{\vee}$.
Thus the the correspondence is sign-preserving and bijective.
\end{proof}

\subsection{Completion of the proof of Theorem \ref{t:G}}\label{s:repeat}
Since $\partial_s$ is a derivation, \eqref{e:W-B} and Proposition~\ref{p:B} give
\begin{equation}\label{e:Fdefect-general}
FM^{\mr{W}}_\kappa((\lambda,\mu),\nu^*)
=[\chi_\nu]\Bigl
\{\Fop(\chi_\lambda)\chi_\mu
+\chi_\lambda\Fop(\chi_\mu)
-\Fop(\chi_\lambda\chi_\mu)\Bigr\}.
\end{equation}

The $\Sop$ part of \eqref{e:Fdefect-general} contributes
\begin{equation}\label{soppy}
\bigl(s(\lambda)+s(\mu)-s(\nu)\bigr)c_{\lambda\mu}^{\nu}
={\ExpM}((\lambda,\mu),\nu^*)c_{\lambda\mu}^{\nu}.
\end{equation}
The term $[\chi_\nu]
\Jop(\chi_\lambda)\chi_\mu$ equals 
\[
\sum_{\alpha>0}
\sum_{\substack{r\ge1\\r\kappa<\langle\lambda+\rho,\alpha^\vee\rangle}}
 {c}_{\lambda-r\kappa\alpha,\mu}^{\nu},
\]
and $[\chi_\nu] \Jop(\chi_\mu)\chi_\lambda$ is the analogous sum with
$\mu$ interchanged with $\lambda$.  For the term with
$\chi_{\lambda}\chi_{\mu}$, expand
\[
\chi_\lambda\chi_\mu
=\sum_\eta c_{\lambda\mu}^{\eta}\chi_\eta.
\]
By Proposition~\ref{p:adjoint},
\begin{align*}
[\chi_\nu]\Jop(\chi_\lambda\chi_\mu)
&=\sum_\eta c_{\lambda\mu}^{\eta}(\chi_\nu,\Jop\chi_\eta)\\
&=\sum_\eta c_{\lambda\mu}^{\eta}(\Aop\chi_\nu,\chi_\eta)\\
&=\sum_{\alpha>0}
 \sum_{\substack{t\in\mathbb Z\\
 t<0\ \mathrm{or}\ t\kappa>\langle\nu+\rho,\alpha^\vee\rangle}}
  {c}_{\lambda,\mu}^{\nu-t\kappa\alpha}.
\end{align*}
Substitution into \eqref{e:Fdefect-general} is exactly \eqref{e:G}.
This proves Theorem~\ref{t:G}.

\subsection{The formula for $\mathfrak{sl}_2$}
\label{sec:A1}

We index the dominant weights of $\mathfrak{sl}_2$ by integers $a\ge0$, writing $[a]$ for
the irreducible representation of highest weight $a\varpi$.  The ordinary
Clebsch--Gordan rule is
\begin{equation}\label{e:A1-CG}
  [a]\cdot [b]=\sum_{i=0}^{\min(a,b)}[a+b-2i].
\end{equation}
Thus
\begin{equation}\label{e:A1-CGcoef}
c_{ab}^{c}=1
\quad\Longleftrightarrow\quad
|a-b|\le c\le a+b,\qquad a+b+c\equiv0\pmod2,
\end{equation}
and otherwise $c_{ab}^{c}=0$.

\subsubsection{The first-moment deformation ring} 
Since every nonzero Clebsch--Gordan multiplicity is one, 
\begin{equation}\label{e:A1-dual-product}
  [a]\diamond_\kappa[b]
  =\sum_{i=0}^{\min(a,b)} \bigl(1+\epsilon
  FM^{\mr{W}}_\kappa(a,b,a+b-2i)\bigr)[a+b-2i].  
\end{equation}

\begin{proposition}\label{p:A1-floor}
Suppose $|a-b|\leq c\leq a+b$. Then

\begin{equation}\label{e:A1-G-abc}
{
\begin{aligned}
FM^{\mr{W}}_\kappa(a,b,c)
={}&\frac{a+b-c}{2}
+\left\lfloor\frac{a+b+c+2}{2\kappa}\right\rfloor
+\left\lfloor\frac{b+c-a}{2\kappa}\right\rfloor
+\left\lfloor\frac{a+c-b}{2\kappa}\right\rfloor\\
&+\left\lfloor\frac{a+b-c}{2\kappa}\right\rfloor
-\left\lfloor\frac a\kappa\right\rfloor
-\left\lfloor\frac b\kappa\right\rfloor
-\left\lfloor\frac{c+1}{\kappa}\right\rfloor.
\end{aligned}}
\end{equation}
\end{proposition}
\begin{proof}
Set $x=(a+b-c)/2$, $y=(a+c-b)/2$, and $z=(b+c-a)/2$.
Then ${\ExpM}(a,b,c)=x$. We record the signed counts appearing in formula~\eqref{e:G}.  For the
first sum on the right hand side of ~\eqref{e:G}, consider the term indexed by $r\ge1$. 

If $a-2r\kappa\ge0$, it contributes $+1$ to the
coefficient of $[c]$ exactly when
\[
r\kappa\le x,\qquad r\kappa\le y.
\]
If $a-2r\kappa\le-2$, finite-Weyl regularisation changes the sign, and a contribution $-1$
occurs exactly when
\[
\max(x,y)<r\kappa\le x+y.
\]
(The dot-singular value $-1$ contributes zero.)  Hence the whole first sum
is
\begin{equation}\label{e:A1-Sa}
S_a=
\left\lfloor\frac{x}{\kappa}\right\rfloor
+\left\lfloor\frac{y}{\kappa}\right\rfloor
-\left\lfloor\frac{x+y}{\kappa}\right\rfloor.
\end{equation}
By symmetry, the second  sum in ~\eqref{e:G} is
\begin{equation}\label{e:A1-Sb}
S_b=
\left\lfloor\frac{x}{\kappa}\right\rfloor
+\left\lfloor\frac{z}{\kappa}\right\rfloor
-\left\lfloor\frac{x+z}{\kappa}\right\rfloor.
\end{equation}
 
For the third sum in \eqref{e:G}, the negative integers $t=-r<0$ give
ordinary  weights $c+2r\kappa$; these occur precisely for
$r\kappa\le x$, and hence contribute
$\lfloor x/\kappa\rfloor$.  For positive $t$ satisfying
$t\kappa>c+1=y+z+1$, the term $c-2t\kappa$ is on the negative side of the
finite wall and contributes with sign $-1$ precisely while
$t\kappa\le x+y+z+1$.  Thus the entire quantity subtracted in the
last line of formula~\eqref{e:G} is
\begin{equation}\label{e:A1-Sout}
T_{\rm c}
=\left\lfloor\frac{x}{\kappa}\right\rfloor
+\left\lfloor\frac{y+z+1}{\kappa}\right\rfloor
-\left\lfloor\frac{x+y+z+1}{\kappa}\right\rfloor.
\end{equation}
Since ${\ExpM}=x$, formula~\eqref{e:G} reads
$FM^{\mr{W}}_\kappa=x+S_a+S_b-T_{c}$, which gives \eqref{e:A1-G-abc}.
\end{proof}

\subsection{The weight of  $KZ_{\kappa}((\lambda,\lambda^*),0)$}

Let $\mf{g}=\mf{sl}_m$. For a dominant weight $x$ put
\[
  n_\kappa(x)=\#\{\alpha\in R^+:\ \langle
  x,\alpha^\vee\rangle\in\kappa\mathbb Z\}.
\]

\begin{proposition}\label{two_point}
The rank one $KZ_{\kappa}((\lambda,\lambda^*),0)$ has weight 
$\ExpM((\lambda,\lambda^*),0)
 +n_\kappa(\lambda+\rho)-n_\kappa(\rho).$
\end{proposition}
\begin{proof}
 
Formula \eqref{e:Fdefect-general} gives
\[
 FM^{\mr{W}}_\kappa((\lambda,\lambda^*),0)
 =\ExpM((\lambda,\lambda^*),0)
 -[\ch_0]\Jop(\ch_\lambda\ch_{\lambda^*}),
\tag{1}
\]
because the quantities $[\ch_0](\Jop\ch_\lambda)\ch_{\lambda^*}$ and $[\ch_0](\Jop\ch_\lambda^*)\ch_{\lambda}$ vanish (see Section \ref{s:id}):
 
By symmetry it suffices to show that the first quantity vanishes. Now,
$[\ch_0](\Jop\ch_\lambda)\ch_{\lambda^*}=[\ch_\lambda]\Jop\ch_\lambda$,
and a term $\ch(\lambda-r\kappa\alpha)$ could contribute $\ch_\lambda$
only if $r\kappa=\langle\lambda+\rho,\alpha^\vee\rangle$, whereas the
sum defining $\Jop\ch_\lambda$ has the strict inequality
$r\kappa<\langle\lambda+\rho,\alpha^\vee\rangle$. This shows the
desired vanishing. 

By Proposition \ref{p:adjoint} we have
\[
[\chi_0]\Jop(\chi_\lambda\chi_{\lambda^*})
=
(\Aop\chi_0,\chi_\lambda\chi_{\lambda^*}),
\]
where
\[
\Aop(\chi_\nu)
=
\sum_{\alpha>0}
\sum_{\substack{r\in\mathbb Z\\
r<0\ \mathrm{or}\ r\kappa>\langle\nu+\rho,\alpha^\vee\rangle}}
\chi(\nu-r\kappa\alpha).
\]
Set $ h_\alpha=\langle\rho,\alpha^\vee\rangle $. Then
\[
  \Aop(\chi_0) = \sum_{\alpha>0} \left(
    \sum_{r\ge1}\chi(r\kappa\alpha) +
    \sum_{r\kappa>h_\alpha}\chi(-r\kappa\alpha) \right).
\]
Finite-Weyl antisymmetry gives
$
\chi(-r\kappa\alpha)
=
-\chi((r\kappa-h_\alpha)\alpha),
$
because $
s_\alpha\cdot(-r\kappa\alpha)
=
(r\kappa-h_\alpha)\alpha
$. 
Therefore
\begin{align*}
\Aop(\chi_0)
={}&
\sum_{r\ge1}\sum_{\alpha>0}
\left(
\chi(r\kappa\alpha)
-\chi((r\kappa-h_\alpha)\alpha)
\right)\\
&+
\sum_{\alpha>0}
\sum_{r\kappa\le h_\alpha}
\chi((r\kappa-h_\alpha)\alpha).
\end{align*}

For every positive integer \(N\),
\[
\sum_{\alpha>0}
\left(
\chi(N\alpha)-\chi((N-h_\alpha)\alpha)
\right)
=
\sum_{\beta\in R}e^{N\beta}.
\]
Indeed, multiplying by \(A_\rho\), the left side becomes
$
\sum_{\alpha>0}
\left(
A_{\rho+N\alpha}
-
A_{\rho+(N-h_\alpha)\alpha}
\right).
$
Since
$
s_\alpha(\rho-N\alpha)
=
\rho+(N-h_\alpha)\alpha,
$,
we have
$
A_{\rho+(N-h_\alpha)\alpha}
=
-A_{\rho-N\alpha},
$
and hence
\[
\sum_{\alpha>0}
\left(
A_{\rho+N\alpha}+A_{\rho-N\alpha}
\right)
=
\sum_{\beta\in R}A_{\rho+N\beta}
=
A_\rho\sum_{\beta\in R}e^{N\beta}.
\]

Now consider the second line above.  The terms with \(r\kappa<h_\alpha\)
sum to 
$0=-\Jop(\chi_0)=
-\sum_{\alpha>0}
\sum_{0<r\kappa<h_\alpha}
\chi(-r\kappa\alpha)$.

The remaining boundary terms are those with \(r\kappa=h_\alpha\), and each
contributes \(\chi(0)=1\).  Thus
\begin{equation}\label{formulaAOP}
\Aop(\chi_0)
=
n_\kappa(\rho)
+
\sum_{r\ge1}\sum_{\beta\in R}e^{r\kappa\beta}.
\end{equation}

Put \(x=\lambda+\rho\), and set
$
F=\Aop(\chi_0).
$
First,
$
(F,\chi_\lambda\chi_{\lambda^*})
=
(F\chi_\lambda,\chi_\lambda),
$
because multiplication by \(\chi_{\lambda^*}\) is adjoint to multiplication
by \(\chi_\lambda\) for the standard orthonormal character pairing.  Thus we get 
$
(F,\chi_\lambda\chi_{\lambda^*})
=
[\chi_\lambda](F\chi_\lambda).
$
Now
$
A_\rho\chi_\lambda=A_x.
$
If
$
F\chi_\lambda=\sum_{\mu\in P^+}c_\mu\chi_\mu,
$
then multiplication by \(A_\rho\) gives
$
FA_x=\sum_{\mu\in P^+}c_\mu A_{\mu+\rho}.
$
Hence the coefficient of \(\chi_\lambda\) in \(F\chi_\lambda\) is exactly
the coefficient of \(A_x\) in \(FA_x\).  Therefore
$
(\Aop\chi_0,\chi_\lambda\chi_{\lambda^*})
=
[A_x]\bigl(\Aop(\chi_0)A_x\bigr).
$

Using \eqref{formulaAOP}, the constant term contributes
\(n_\kappa(\rho)\).  A nonconstant term contributes only if, for some
\(w\in W\), \(\beta\in R\), and \(r\ge1\), $ wx+r\kappa\beta=x.  $
Equivalently, $ wx=x-r\kappa\beta.  $ Since the Weyl group preserves
the invariant inner product, $ |x-r\kappa\beta|^2=|x|^2 $. Expanding
and simplifying yields
$ wx = x-\langle x,\beta^\vee\rangle\beta = s_\beta x.  $

Because \(x=\lambda+\rho\) is dot-regular, its Weyl stabiliser is trivial, so
\(w=s_\beta\).  Since \(w=s_\beta\), the coefficient of \(e^{wx}\) in
$
A_x=\sum_{u\in W}(-1)^{\ell(u)}e^{ux}
$
is
  $
(-1)^{\ell(s_\beta)}=-1.
$
Hence each affine wall through \(x=\lambda+\rho\) contributes \(-1\) to the
coefficient of \(e^x\) in
$
\mathcal A_\kappa(\chi_0)A_x.
$
Conversely, every positive root \(\beta\) satisfying
$
\langle\lambda+\rho,\beta^\vee\rangle\in\kappa\mathbb Z
$
gives exactly one such contribution.  Hence
\[
[\chi_0]\Jop(\chi_\lambda\chi_{\lambda^*})
=
n_\kappa(\rho)-n_\kappa(\lambda+\rho).
\]
Substitution into the first formula proves the claim.
\end{proof}
We expect that the analogue of Proposition \ref{two_point} holds for
all $\mf{g}$.

\begin{remark}\label{r:ll0}
  Conjecture \ref{c:simple} implies that the weight of
  $KZ_{\kappa}((\lambda,\lambda^*),0)$ is determined by the
  (generalised) Andersen filtration of
  $\Hom_{\uq}(\lq(\lambda) \otimes \lq(\lambda)^*, \nq(0))$ (which is a
  one-dimensional vector space). It would be interesting to give a
  representation theoretic computation of this filtration (which is
  essentially a single integer). Is it equal to the
  degree of a suitable (parabolic) Kazhdan--Lusztig polynomial?
\end{remark}

\subsection{An extension to arbitrary root systems}\label{ss:arb}

Let $\g$ be arbitrary.  Assume $\kappa>0$ and write
\[
  \kappa=\frac{K}{d},
  \qquad K,d\in\Z_{>0},\qquad \gcd(K,d)=1.
\] 

Recall (see Section \ref{s:conj}) the quantum parameter
$q=\exp(-{\pi i}/{m_{\g}\kappa})$ and assume
$l=\ord(q^2)>1$. Since $l = m_{\mf{g}}K/(m_{\mf{g}},d)$, this implies
$K>1$ in the simply laced cases, $K>1$ or $d$ is odd in types $B$, $C$
and $F_4$, and $K>1$ or $3 \nmid d$ in type $G_2$.

Define
\[
 \mathcal J^{\g}_\kappa(\chi_\lambda)
 =\sum_{\substack{\alpha\in R^+\\\ell_\alpha>1}}
   \ \sum_{\substack{r\ge1\\
       r l_\alpha<\langle\lambda+\rho,\alpha^\vee\rangle}}
   \chi(\lambda-r\ell_\alpha\alpha),
\]
where we recall that the $l_{\alpha}$ were defined in Section
\ref{s:mn}.

Set
\[
 \mathcal S(\chi_\lambda)=s(\lambda)\chi_\lambda,
 \qquad
 \mathcal F^{\g}_\kappa=\mathcal S+\mathcal J^{\g}_\kappa.
\]
For dominant $\lambda,\mu,\nu$, define
\begin{equation}\label{e:W-general-type-def}
 \FM^{\mathrm{alg}}_\kappa((\lambda,\mu),\nu^*)
 = [\chi_\nu]\bigl(
 \mathcal F^{\g}_\kappa(\chi_\lambda)\chi_\mu
 +\chi_\lambda\mathcal F^{\g}_\kappa(\chi_\mu) 
 -\mathcal F^{\g}_\kappa(\chi_\lambda\chi_\mu)
 \bigr). 
\end{equation}  

Expanding the coefficient gives the root-length version of
Theorem~\ref{t:G}.
\begin{align}
\FM^{\mathrm{alg}}_\kappa((\lambda,\mu),\nu^*)
={}&{\ExpM}((\lambda,\mu),\nu^*)c_{\lambda\mu}^{\nu}\notag\\
&+\sum_{\alpha>0,l_{\alpha}>1} 
 \sum_{\substack{r\ge1\\rl_\alpha<\langle\lambda+\rho,\alpha^\vee\rangle}}
  {c}_{\lambda-r l_\alpha\alpha,\mu}^{\nu}\notag\\
&+\sum_{\alpha>0,l_{\alpha}>1}
 \sum_{\substack{r\ge1, r l_\alpha<\langle\mu+\rho,\alpha^\vee\rangle}}
  {c}_{\lambda,\mu-r l_\alpha\alpha}^{\nu}\notag\\
&-\sum_{\alpha>0, l_{\alpha}>1}
 \sum_{\substack{t\in\mathbb Z\\
 t<0\ \mathrm{or}\ t l_\alpha>\langle\nu+\rho,\alpha^\vee\rangle}}
  {c}_{\lambda,\mu}^{\nu-t l_\alpha\alpha}.
\label{e:G-general-type}
\end{align}

It is easy to see that the deformation of $\mf{R}(\mathfrak{g})$ given
by \eqref{e:W-general-type-def} is associative because, as in the case
of $\slg_m$, we are reparametrizing the representation ring by
$T=1+\epsilon \mathcal F^{\mathfrak g}_\kappa$:
\begin{proposition}\label{p:general-type-assoc}
The product
\[
 [\lambda]\widetilde\diamond_\kappa[\mu]
 =\sum_\nu\left(c_{\lambda\mu}^{\nu}
 +\epsilon\FM^{\mathrm{alg}}_\kappa
            ((\lambda,\mu),\nu^*)\right)[\nu]
\]
is associative and commutative.
\end{proposition}
We expect that  $\FM^{\mathrm{alg}}_\kappa$ is the true weight moment $\FM^{\mr{W}}_\kappa$ for general $\mathfrak{g}$ (as in Definition~\ref{d:fmring}), so that ${\diamond}_\kappa=\tilde{\diamond}_\kappa$, as is the case for $\slg_m$.

\section{The weight-moment implication of Conjecture \ref{c:tilting}}
\label{s:wmi}

Suppose $\mathfrak{g}=\slg_m$ and that $\vec{\lambda}$ is such that
$\dq(\lambda_i)$ is a simple (equivalently, tilting) module for each
$i$. Then Conjecture \ref{c:tilting} asserts that the weight
filtration on $KZ_{\kappa}(\vec{\lambda},\nu^*)$ coincides, up to a shift by
${\ExpM}(\vec{\lambda},\nu^*)$, with the dual of the Andersen
filtration on $\operatorname{Hom}_{\uq}(\dq(\vec{\lambda}),\nq(\nu))$.

This implies that
\begin{equation}\label{e:multipoly}
\mr{Wt}\bigl(KZ_{\kappa}(\vec{\lambda},\nu^*)\bigr)
=w^{{\ExpM}(\vec{\lambda},\nu^*)}f_\nu(\dq(\vec{\lambda}))\in\mathbb{Z}[w].
\end{equation}
This equation holds at $w=1$, since
\[
f_\nu(\dq(\vec{\lambda}))(1)
=\dim\operatorname{Hom}_{\uq}(\dq(\vec{\lambda}),\nq(\nu))
=c_{\vec{\lambda}}^{\nu}.
\]
We show here that the equation also holds after taking derivatives at $w=1$. Recall that
\[
\mr{Wt}\bigl(KZ_{\kappa}(\vec{\lambda},\nu^*)\bigr)'(1)
=FM^{\mr{W}}_\kappa\bigl(KZ_{\kappa}(\vec{\lambda},\nu^*)\bigr).
\]

\begin{theorem}
\label{t:many}
Let $\mathfrak g=\mathfrak{sl}_m$, let $\kappa\ge2$, and let $n>1$.
Suppose that every $\dq(\lambda_i)$ is simple. Then
\begin{equation}\label{e:conclusion}
FM^{\mr{W}}_\kappa\bigl(KZ_{\kappa}(\vec{\lambda},\nu^*)\bigr)
={\ExpM}(\vec{\lambda},\nu^*)c_{\vec{\lambda}}^{\nu}
+f_\nu(\dq(\vec{\lambda}))'(1).
\end{equation}
\end{theorem} 

The proof is obtained by computing the two sides separately and showing that they are equal.

\subsection{The Andersen filtration moment}

Let $Q=\dq(\vec{\lambda})$. Since $Q$ is tilting, the
Andersen--Kulkarni tilting sum formula \cite[Theorem 6.4]{AK} applies
directly\footnote{In \cite{AK} it is assumed that $l$ is odd, but
  using the results of \cite[Section 7]{LTV} (instead of those from
  \cite{APW} as in \cite{AK}),the results of \cite{AK} concerning
  quantum groups extend to all $l$.} to
$Q$. Together with
\[
f_\nu(Q)'(1)=\sum_{j>0}\dim F_\nu(Q)^j,
\]
it gives (see Section~\ref{s:id})
\begin{equation}\label{e:Andersen-many}
f_\nu(Q)'(1)
=-\sum_{\alpha>0}
\sum_{\substack{t\in\mathbb Z\\
t<0\ \mathrm{or}\ t\kappa>\langle\nu+\rho,\alpha^\vee\rangle}}
(\operatorname{ch}Q,\chi(\nu-t\kappa\alpha)).
\end{equation}
Note that $\operatorname{ch}Q=\prod_{i=1}^n\chi_{\lambda_i}$.

\subsection{The KZ first moment}

We now compute $FM^{\mr{W}}_\kappa(KZ_{\kappa}(\vec{\lambda},\nu^*))$, generalizing the $n=2$ case from
\eqref{e:W-B}. Recall the operators and notation of Section~\ref{s:id}.

For $n>1$ and $f_1,\dots,f_n\in\mathbb{Z}[P]^W$, define
\begin{equation}\label{e:higher-Leibniz-defect1}
\DD_n(f_1,\ldots,f_n)
=\sum_{i=1}^n\Fop(f_i)\prod_{j\ne i}f_j
-\Fop\!\left(\prod_{i=1}^n f_i\right).
\end{equation}

\begin{lemma}\label{l:fmwn}
Let $\mu_i$ be arbitrary dominant weights. Then
\begin{equation}\label{e:multif}
FM^{\mr{W}}_\kappa\bigl(KZ_{\kappa}(\vec{\mu},\nu^*)\bigr)
=[\chi_\nu]\DD_n(\chi_{\mu_1},\ldots,\chi_{\mu_n}).
\end{equation}
\end{lemma}

\begin{proof}
  We claim that the character of the coefficient of $\epsilon$ in the
  $n$-fold product of the $[\mu_i]$ with respect to the product
  $\diamond_\kappa$ is $\DD_n(\chi_{\mu_1},\ldots,\chi_{\mu_n})$.  This
  follows from the case $n=2$, given by \eqref{e:W-B} and
  \eqref{e:Bformula}, by an easy induction.
\end{proof}

Recall that $\Fop=\Sop+\Jop$ and
$\Sop(\chi_{\lambda_i})=s(\lambda_i)\chi_{\lambda_i}$. For every $i$, the simplicity of
$\dq(\lambda_i)$ implies that its positive Jantzen filtration is zero. The quantum Jantzen sum
formula of Andersen--Kulkarni \cite[Theorem 6.3]{AK} therefore gives
$\Jop(\chi_{\lambda_i})=0$ for all $i$. Thus \eqref{e:higher-Leibniz-defect1} becomes
\begin{align}
\DD_n(\chi_{\lambda_1},\ldots,\chi_{\lambda_n})
={}&\left(\sum_{i=1}^n s(\lambda_i)\right)
\prod_{i=1}^n\chi_{\lambda_i}
-\Sop\!\left(\prod_{i=1}^n\chi_{\lambda_i}\right)\notag\\
&-\Jop\!\left(\prod_{i=1}^n\chi_{\lambda_i}\right).
\label{e:multi-defect-simple}
\end{align}

Taking the coefficient of $\chi_\nu$ gives
\begin{equation}\label{e:MSB}
FM^{\mr{W}}_\kappa\bigl(KZ_{\kappa}(\vec{\lambda},\nu^*)\bigr)
={\ExpM}(\vec{\lambda},\nu^*)c_{\vec{\lambda}}^{\nu}
-[\chi_\nu]\Jop\!\left(\prod_{i=1}^n\chi_{\lambda_i}\right).
\end{equation}

\subsection{Conclusion of the proof of Theorem~\ref{t:many}}

Comparing the formulas, it remains to show that
\[
\sum_{\alpha>0}
\sum_{\substack{t\in\mathbb Z\\
t<0\ \mathrm{or}\ t\kappa>\langle\nu+\rho,\alpha^\vee\rangle}}
\left(\prod_{i=1}^n\chi_{\lambda_i},\chi(\nu-t\kappa\alpha)\right)
=[\chi_\nu]\Jop\!\left(\prod_{i=1}^n\chi_{\lambda_i}\right).
\]

This is an immediate generalization of the argument for the term in $\chi_{\lambda}\chi_{\mu}$ in Section~\ref{s:repeat}.
Apply Proposition~\ref{p:adjoint} to
$Z=\prod_{i=1}^n\chi_{\lambda_i}$. We obtain
\begin{align}
[\chi_\nu]\Jop(Z)
={}&(\chi_\nu,\Jop Z)=(\Aop\chi_\nu,Z)\notag\\
={}&\sum_{\alpha>0}
\sum_{\substack{t\in\mathbb Z\\
t<0\ \mathrm{or}\ t\kappa>\langle\nu+\rho,\alpha^\vee\rangle}}
{c}_{\vec{\lambda}}^{\,\nu-t\kappa\alpha}.
\label{e:multi-output-sum}
\end{align}
Since
\[
\left(\prod_{i=1}^n\chi_{\lambda_i},\chi(\nu-t\kappa\alpha)\right)
={c}_{\vec{\lambda}}^{\,\nu-t\kappa\alpha},
\]
the desired equality follows.

\begin{remark}
  Assuming that $\FM^{\mathrm{alg}}_\kappa$ is indeed equal to
  $\FM^{\mr{W}}_\kappa$ for general $\mathfrak{g}$, arguments
  analogous to the ones above would imply that the weight moment
  implication of Conjecture \ref{c:tilting} also holds for $\mf{g}$.
\end{remark}

\subsection{Example \ref{e: GenQGN0} at the level of
  moments}\label{s:jantzen}

In this section we prove that the prediction for
$FM_{\kappa}^W(\vec{\lambda},\nu^*)$ implied by Conjecture
\ref{c:simple} agrees with its actual value when one of the
$\lambda_i$ is in $N_0$ or $\nu \in N_0$. As explained in Example
\ref{e: GenQGN0}, these cases are easier to deal with than the general
case since there are no higher $\Ext$ terms in the spectral sequence.

\begin{proposition}\label{p:N_0}
  Let $\mf{g} = \mf{sl}_m$ and consider
  $KZ_{\kappa}(\vec{\lambda},\nu^*)$. If we choose the filtrations on
  the Weyl modules to be the Jantzen filtration, then the prediction
  of the first moment implied by Conjecture \ref{c:simple} agrees with
  $FM_{\kappa}^W(\vec{\lambda},\nu^*)$ if
  \begin{enumerate}
  \item $\lambda_i \in N_0$ for some $i$, or
  \item $\nu \in N_0$.
  \end{enumerate}
\end{proposition}
The advantage of the Jantzen filtration over the radical filtration is
the Jantzen sum formula. Any Weyl module $\dq(\lambda)$ for
$\lambda \in P^+$ has a finite decreasing Jantzen filtration
$\dq(\lambda)^{\bullet}$ with $\dq(\lambda)^{0}=\dq(\lambda)$ and
$\dq(\lambda)^{1}$ the radical of $\dq(\lambda)$. Then the Jantzen
sum formula in this context \cite[Theorem 6.3 (a)]{AK} says
that\footnote{As noted earlier, the results of \cite[Section 7]{LTV}
  can be used to show that the results of \cite{AK} extend to all
  $l$.}
\begin{equation}\label{AKFormula}
  \sum_{c>0}\operatorname{ch}(\dq(\lambda)^{c})= - \Jop(\chi_{\lambda}) .
\end{equation}
We note that the semisimplicity of the layers of the Jantzen
filtration is expected, but not known (as far as we are aware), but this
will not matter for the proof.
\begin{proof}
  We first prove (1). Without loss of generality we can assume
  $\lambda_1 \in N_0$ and then by Lemma
  \ref{e:multif}, we have
\begin{equation}\label{e:FDG}
FM^{\mr{W}}_\kappa\bigl(\vec{\lambda},\nu^*\bigr)
={\ExpM}(\vec{\lambda},\nu^*)c_{\vec{\lambda}}^{\nu}
+[\chi_\nu]
\Bigl \{\sum_{i=2}^n\Jop(\chi_{\lambda_i})\prod_{j\neq
  i,j\in[n]}\chi_{\lambda_j} \Bigr \}
-[\chi_\nu]\Jop\!\left(\prod_{i=1}^n\chi_{\lambda_i}\right) .
\end{equation} 
Here we have used $\Jop(\chi_{\lambda_1})=0$ since $\dq(\lambda_1)$ is
simple.

Now we compute the first moment coming from Conjecture \ref{c:simple}
(using the Jantzen filtration). We first note that since
$\lambda_1 \in N_0$, each
$Q_c= \dq(\vec{\lambda})^c/\dq(\vec{\lambda})^{c+1}$ is projective, so
there are no higher $\Ext$ terms in the spectral sequence and it
degenerates at $E_1$.  The space
$\Hom_{\uq}(\dq(\vec{\lambda}),\nq (\nu))$ receives a
filtration whose graded quotients are the $E^{s,-s}_{\infty}$. By
degeneration of the spectral sequence at $E_1$ and $\Ext$ vanishings, letting
$c=-s$,

$$E^{s,-s}_{\infty}=\Hom_{\uq}(\dq(\vec{\lambda})^c/\dq(\vec{\lambda})^{c+1},\nq (\nu)).$$ 
 
We have the increasing dual Andersen filtration
$\Hom_{\uq}(\dq(\vec{\lambda})^c/\dq(\vec{\lambda})^{c+1},\nq
(\nu))_{\bullet}$ on $E^{s,-s}$. In the definition of $W_r'$
in Conjecture \ref{c:simple} this filtration is shifted by $D-c$. Since all the $Q_c$ are projective we have
$\dq(\vec{\lambda}) \cong \oplus_c Q_c$, so without the $-c$ shifts (but
keeping the $\ExpM$ shift) in the definition of the filtration on
$E^{s,-s}$, we would get the same answer for the first moment as the
first moment of the Andersen filtration on $\Hom(Q,\nq (\nu))$, i.e.,
the sum of the first and the third terms in the RHS of
\eqref{e:FDG} (as in \eqref{e:MSB}).

By a simple computation, using that if $P$ is a tilting module then
$[\chi_{\nu}]\op{ch}(P)$ is the rank of $\Hom_{\dq}(P,\nq(\nu))$, it
follows that in order to compensate for the $-c$ terms that we have
ignored above, we need to subtract
$[\chi_{\nu}]\sum_{c\geq 1}\operatorname{ch} \dq(\vec{\lambda})^c$ from
the sum. By \eqref{AKFormula} and Lemma \ref{l:sum} below we have
$$[\chi_{\nu}]\sum_{c\geq 1}\operatorname{ch} \dq(\vec{\lambda})^c= - [\chi_\nu]\Bigl
\{\sum_{i=2}^n\Jop(\chi_{\lambda_i})\prod_{j\neq
  i,j\in[n]}\chi_{\lambda_j} \Bigr\} $$ which shows the desired
agreement for (1) at the level of first moments.

We now prove (2), so we need to show that Conjecture
\ref{c:simple} implies that
\begin{equation}\label{e:nu}
FM^{\mr{W}}_\kappa\bigl(\vec{\lambda},\nu^*\bigr)
={\ExpM}(\vec{\lambda},\nu^*)c_{\vec{\lambda}}^{\nu}
+[\chi_\nu]
\Bigl \{\sum_{i=1}^n\Jop(\chi_{\lambda_i})\prod_{j\neq
  i,j\in[n]}\chi_{\lambda_j} \Bigr \}
-[\chi_\nu]\Jop\!\left(\prod_{i=1}^n\chi_{\lambda_i}\right) .
\end{equation} 
The argument is similar to that of (1)
except we must deal with the fact that the $Q_c$ need not be
projective.

The main observation is that if $\nu \in N_0$ then
$[\chi_{\nu}]\Jop(\chi) = 0$ for any character $\chi$, equivalently,
$[\chi_{\nu}]\Jop(\chi_{\mu}) = 0$ for each dominant weight $\mu$.  This
follows from the Jantzen sum formula \eqref{AKFormula} since the block
of $\dq(\nu)$ consists only of modules which are sums of
$\dq(\nu)$. Thus, the third term in the RHS of \eqref{e:nu} is $0$
and the first term gives the value of
$FM^{\mr{W}}_\kappa\bigl(\vec{\lambda},\nu^*\bigr)$ predicted by
Conjecture \ref{c:simple} if we ignore the shifts by $-c$ as in (1).

To complete the proof we need to show that adding the second term
compensates for these shifts. Using Lemma \ref{l:sum} as in (1), this
follows from the observation that for any $\uq$-module $M$ and
$\nu \in N_0$ we have
$[\chi_{\nu}]\op{ch(M)} = \dim(\Hom_{\uq}(M,\nabla(\nu))$, a
consequence, as in the previous paragraph, of the structure of the
block of $\dq(\nu)$.

\end{proof}

\begin{lemma}\label{l:sum}
  Let $T$ be an algebraic torus and $M_1,\dots,M_n$ be
  finite-dimensional $T$-modules with finite separated decreasing
  filtrations $M_j^\bullet$, with $M_j^0=M_j$. Give
\[
M=M_1\tensor\cdots\tensor M_n
\]
the induced tensor product filtration. Then
\begin{equation}\label{e:tfc}
\sum_{c\geq 1}\operatorname{ch}(M^c)
=
\sum_{i=1}^n
\left(
\sum_{s\geq 1}\operatorname{ch}(M_i^s)
\right)
\prod_{j\neq i}\operatorname{ch}(M_j).
\end{equation}
\end{lemma}

\begin{proof}
  Since the filtration is finite,
  $\operatorname{ch}(M^c) = \sum_{r\geq c}
  \operatorname{ch}(M^r/M^{r+1})$.  Consequently,
  $ \sum_{c\geq 1}\operatorname{ch}(M^c) = \sum_{r\geq 1}
  r\,\operatorname{ch}(M^r/M^{r+1})$.

For the tensor product filtration we have
\[
M^r/M^{r+1}
\simeq
\bigoplus_{a_1+\cdots+a_n=r}
(M_1^{a_1}/M_1^{a_1+1})\tensor\cdots\tensor
(M_n^{a_n}/M_n^{a_n+1}).
\]
It follows that
\begin{align*}
\sum_{r\geq 1}
r\,\operatorname{ch}(M^r/M^{r+1})
&=
\sum_{a_1,\dots,a_n\geq0}
(a_1+\cdots+a_n)
\prod_{j=1}^n
\operatorname{ch}(M_j^{a_j}/M_j^{a_j+1})\\
&=
\sum_{i=1}^n
\left(
\sum_{a_i\geq1}
a_i\operatorname{ch}(M_i^{a_i}/M_i^{a_i+1})
\right)
\prod_{j\neq i}
\left(
\sum_{a_j\geq0}
\operatorname{ch}(M_j^{a_j}/M_j^{a_j+1})
\right).
\end{align*}
For each $j$, we have
$\sum_{a_j\geq0} \operatorname{ch}(M_j^{a_j}/M_j^{a_j+1}) =
\operatorname{ch}(M_j), $ while
$ \sum_{a_i\geq1} a_i\operatorname{ch}(M_i^{a_i}/M_i^{a_i+1}) =
\sum_{s\geq1}\operatorname{ch}(M_i^s)$.  Substituting these identities
proves \eqref{e:tfc}.
\end{proof}

\subsection{The case of $\mf{sl}_2$}
In this section we will show that for $\mf{g}=\mf{sl}_2$ and all
$\vec{\lambda}$ and $\nu$, the prediction of the first moment of the
weight filtration derived from Conjecture \ref{c:simple} agrees with
$FM_{\kappa}^W(\vec{\lambda},\nu^*)$.

The argument is similar to that of Proposition \ref{p:N_0}, but is
complicated by the fact that in general there can be higher $\Ext$
terms in the spectral sequence, so it need not degenerate at $E_1$. In
order to deal with this we will need the following:
\begin{lemma}\label{l:js}
  Let $\mf{g} = \mf{sl}_2$.  For any dominant weights $\mu$ and
  $\nu$ we have
  \begin{enumerate}
  \item $[\chi_{\nu}]\op{ch}(L(\mu)) = (-1)^i$ if
    $\Ext^i(L(\mu),\nabla(\nu)) \neq 0$ for some $i \geq 0$ (which is
    then unique) and is $0$ otherwise,
  \item $[\chi_{\nu}]\op{ch}(T(\mu)) = 1$ if $\nu = \mu$ or $\nu =
    d(\mu)$ and is $0$ otherwise.
  \item $-[\chi_{\nu}]\Jop(\op{ch}(L(\mu))) = (-1)^i\,i$ if
    $\Ext^i(L(\mu),\nabla(\nu)) \neq 0$ for some $i > 0$ (which is
    then unique) and is $0$ otherwise,
  \item $-[\chi_{\nu}]\Jop(\op{ch}(T(\mu))) = 1$ if $\nu = d(\mu)$ and
    $\mu \neq d(\mu)$ and is $0$ otherwise.
 \end{enumerate}
\end{lemma}   
\begin{proof}
  For $\mu$ dominant, we have
  $\chi_{\mu} = \op{ch}(L(\mu)) + \op{ch}(L(d(\mu)))$ if
  $\mu \neq d(\mu)$ and $\chi_{\mu} = \op{ch}(D(\mu))$
  otherwise. This implies that
  \begin{equation}\label{e:charsimple}
    \op{ch}(L(\mu)) = \chi_{\mu} + \sum_{\substack{i \geq 1, \\ d^i(\mu)\, \neq \,
      d^{i+1}(\mu)}} (-1)^i \chi_{d^i(\mu)} .
  \end{equation}
  Clearly (1) follows immediately from this, the orthogonality of
  characters of Weyl modules and Lemma \ref{l:exts}.

  Part (2) follows immediately from \eqref{e:tilt}.
  
  Since $-\Jop(\chi_{\mu})$ is equal to
  $\op{ch}(L(d(\mu)))$ if $\mu \neq d(\mu)$ and is $0$
  otherwise, it follows that
  \[
    -\Jop(\op{ch}(L(\mu))) = \sum_{\substack{i \geq 1, \\ d^i(\mu) \neq
        d^{i+1}(\mu)}} (-1)^{i+1} \op{ch}(L(d^i(\mu)) =
    \sum_{\substack{i \geq 1, \\ d^i(\mu) \neq d^{i+1}(\mu)}}
    (-1)^{i+1} i\chi_{d^i(\mu)} .
  \]
  If $\nu$ is any dominant weight, it therefore follows that
  \begin{equation*}
    -[\chi_{\nu}]\Jop(\op{ch}(L(\mu))) = \begin{cases}
      (-1)^{i+1} & \mbox{if $\nu = d^i(\mu)$ for some $i>0$ and
        $d^{i-1}(\mu) \neq d^i(\mu)$}, \\
      0 & \mbox{otherwise.}
    \end{cases}
  \end{equation*}
  Part (3) follows from this by applying Lemma \ref{l:exts}.

  If $\mu \neq d(\mu)$ then
  $\op{ch}(T(\mu)) = \chi_{\mu} + \chi_{d(\mu)}$ so
  $-\Jop(\op{ch}(T(\mu))) = -\Jop(\chi_{\mu} + \chi_{d(\mu)})
  = -\chi_{d(\mu)}$. Part (4) follows immediately from this.
\end{proof}

\begin{theorem}\label{t:sl2fm}
  For $\mf{g} = \mf{sl}_2$ and all $\vec{\lambda}$ and $\nu$, the
  prediction for the first moment of the weight filtration implied by
  Conjecture \ref{c:simple} agrees with
  $FM_{\kappa}^W(\vec{\lambda},\nu^*)$.
\end{theorem}
We note that for $\mf{sl}_2$ the radical filtration is equal to the
Jantzen filtration.

\begin{proof}
  By Lemma \ref{e:multif}, we have
  \begin{equation}\label{e:dsl2}
    FM^{\mr{W}}_\kappa\bigl(\vec{\lambda},\nu^*\bigr)
    ={\ExpM}(\vec{\lambda},\nu^*)c_{\vec{\lambda}}^{\nu}
    +[\chi_\nu]
    \Bigl \{\sum_{i=1}^n\Jop(\chi_{\lambda_i})\prod_{j\neq
      i,j\in[n]}\chi_{\lambda_j} \Bigr \}
    -[\chi_\nu]\Jop\!\left(\prod_{i=1}^n\chi_{\lambda_i}\right) .
\end{equation} 

Recall that the $\Ext$ spectral sequence of Conjecture \ref{c:simple}
need not degenerate at $E_1$, but by Theorem \ref{t:sl2deg} it
degenerates at $E_2$.  This implies that for each $c \geq 0$ we have
\begin{multline}\label{e:degform}
  \dim(\Hom^0_{\dq}(\dq(\vec{\lambda})^c/\dq(\vec{\lambda})^{c+1},\nq
  (\nu))) \\
  = \sum_{i \geq 0} (-1)^i \dim
  (\Ext^i_{\uq}(\dq(\vec{\lambda})^{c-i}/\dq(\vec{\lambda})^{c+1-i},\nq
  (\nu))).
\end{multline}

Also recall from Lemma \ref{l:ab} that each
$\Delta(\vec{\lambda})^c/\Delta(\vec{\lambda})^{c+1}$ is a direct sum
of simple and indecomposable tilting modules. Since
$\prod_{i=1}^n\chi_{\lambda_i} = \sum_{c \geq
  0}\op{ch}(\dq(\vec{\lambda})^c/\dq(\vec{\lambda})^{c+1})$, Lemma
\ref{l:js} and the formulas for the Jantzen filtration for simple and
tilting modules for $\mf{sl}_2$ imply that
$-[\chi_\nu]\Jop\!\left(\prod_{i=1}^n\chi_{\lambda_i}\right )$ is equal
to the first moment of the dual Andersen filtration (without any
shifts) on
\[
  \bigoplus_{c \geq 0}
  \Hom_{\dq}(\dq(\vec{\lambda})^c/\dq(\vec{\lambda})^{c+1},\nq
  (\nu)))\] plus the term
\[
  \sum_{c \geq 0} \sum_{i\geq 0} (-1)^i\, i \cdot \dim
  (\Ext^i_{\uq}(\dq(\vec{\lambda})^c/\dq(\vec{\lambda})^{c+1},\nq
  (\nu))).
\]
We have explained in Section \ref{s:compute} that for any $c \geq 0$
the dual Andersen filtration on
$\Hom_{\uq}^0(\dq(\vec{\lambda})^{c}/\dq(\vec{\lambda})^{c+1}, \nq
(\nu))$ differs from that on
$\Hom_{\uq}(\dq(\vec{\lambda})^{c}/\dq(\vec{\lambda})^{c+1}, \nq
(\nu))$ only in degree $0$, so the first moment for
$\Hom_{\dq}(\dq(\vec{\lambda})^c/\dq(\vec{\lambda})^{c+1},\nq (\nu)))$
without any shifts equals that for
$\Hom^0_{\dq}(\dq(\vec{\lambda})^c/\dq(\vec{\lambda})^{c+1},\nq
(\nu)))$. It then follows from the above and \eqref{e:degform} that
the first moment prediction from Conjecture \ref{c:simple} after
incorporating the $D$ shift is
\begin{multline}\label{e:m1}
  {\ExpM}(\vec{\lambda},\nu^*)c_{\vec{\lambda}}^{\nu}
  -[\chi_\nu]\Jop\!\left(\prod_{i=1}^n\chi_{\lambda_i}\right)
  \\
  +
      \sum_{c \geq 0} \sum_{i>0} (-1)^i\, i \cdot \dim
    (\Ext^i_{\uq}(\dq(\vec{\lambda})^c/\dq(\vec{\lambda})^{c+1},\nq
    (\nu))) .
 \end{multline}
 In order to incorporate the $-c$ shifts as well, we must add the term
 \[
   -\sum_{c \geq 0}c\cdot
   \dim(\Hom^0_{\dq}(\dq(\vec{\lambda})^c/\dq(\vec{\lambda})^{c+1},\nq
   (\nu)))
\]
to \eqref{e:m1}, so by \eqref{e:degform} the predicted first moment is
equal to
\begin{multline}\label{e:m2}
  {\ExpM}(\vec{\lambda},\nu^*)c_{\vec{\lambda}}^{\nu}
  -[\chi_\nu]\Jop\!\left(\prod_{i=1}^n\chi_{\lambda_i}\right)
  \\
  +
      \sum_{c \geq 0} \sum_{i>0} (-1)^i\, i \cdot \dim
    (\Ext^i_{\uq}(\dq(\vec{\lambda})^c/\dq(\vec{\lambda})^{c+1},\nq
    (\nu))) \\
    + \sum_{c \geq 0}\sum_{i > 0} (-1)^i \dim
  (\Ext^i_{\uq}(\dq(\vec{\lambda})^{c-i}/\dq(\vec{\lambda})^{c+1-i},\nq
  (\nu))) \\
  -  \sum_{c \geq 0} \sum_{i \geq 0} (-1)^ic\cdot \dim
  (\Ext^i_{\uq}(\dq(\vec{\lambda})^{c-i}/\dq(\vec{\lambda})^{c+1-i},\nq
  (\nu))).
 \end{multline}

By Lemma \ref{l:js} the sum
$[\chi_{\nu}]\sum_{c\geq 1}\operatorname{ch} \dq(\vec{\lambda})^c$,
which by Lemma \ref{l:sum} and \eqref{AKFormula} equals the middle
term in the RHS of \eqref{e:dsl2}, is equal to
\begin{equation*}
  \sum_{c \geq 0} \sum_{i \geq 0} (-1)^i\, c \cdot \dim
  (\Ext^i_{\uq}(\dq(\vec{\lambda})^c/\dq(\vec{\lambda})^{c+1},\nq
  (\nu))).
\end{equation*}

Since
\begin{multline*}
\sum_{c \geq 0} \sum_{i \geq 0} (-1)^ic\cdot \dim
  (\Ext^i_{\uq}(\dq(\vec{\lambda})^{c-i}/\dq(\vec{\lambda})^{c+1-i},\nq
  (\nu))) \\
  = \sum_{c \geq 0} \sum_{i \geq 0} (-1)^i\, c \cdot \dim
  (\Ext^i_{\uq}(\dq(\vec{\lambda})^c/\dq(\vec{\lambda})^{c+1},\nq
  (\nu))) \\
  + \sum_{c \geq 0} \sum_{i>0} (-1)^i\, i \cdot \dim
  (\Ext^i_{\uq}(\dq(\vec{\lambda})^c/\dq(\vec{\lambda})^{c+1},\nq
  (\nu))) ,
\end{multline*}
it follows that the RHS of \eqref{e:dsl2} is equal to \eqref{e:m2}, as
was to be proved.
\end{proof}

\begin{corollary}\label{c:sl22}
  Conjecture \ref{c:simple} holds for $\mf{sl}_2$ when $n=2$.
\end{corollary}
\begin{proof}
  When $n=2$, $FM_{\kappa}^W(\vec{\lambda},\nu^*)$ equals the weight
  of $KZ_{\kappa}(\vec{\lambda},\nu^*)$ since the latter has rank one whenever
  it is nonzero.
\end{proof}

\section{Proofs of some  more consequences of Conjecture \ref{c:tilting}}\label{s:exp}

In this section we prove some more consequences of Conjecture
\ref{c:tilting} in the case of $\mf{sl}_m$. In Theorem \ref{purity},
we prove a basic lower bound for weights of KZ motives
$KZ_{\kappa}(\vec{\lambda},\nu^*)$ when all $\lambda_i$ are of level
$\ell$ , using Hodge theory, the construction of the KZ motives (which
we recall in Section \ref{DProof}), and results of \cite{BFM}. Using
the first-moment verification, Theorem \ref{t:many}, of Conjecture
\ref{c:tilting} we prove purity of the KZ motives in the case all
$\lambda_i$ are of level $\ell$ and $\nu\in N_0$ as predicted by
Conjecture \ref{c:tilting}. In these results we have to assume that
the level is large enough in comparison with $m$.  The main result of
Section \ref{reflexions} is Theorem \ref{E22} which implies Conjecture
\ref{c:tilting} at the level of ranks for $\mf{sl}_2$.

The methods used in this section are completely different from the
earlier sections. We use the motivic factorisation theorem
\cite[Theorem 6.1]{BFM} for nearby cycles of the motivic VMHS
$KZ_{\kappa}(\vec{\lambda},\nu^*)$ on the configuration spaces
$\mathcal{C}_n$ as a certain number $m$ of the marked points coalesce
at a finite point (which we can think of as the remaining points
coalescing at infinity) and the relation between local monodromy and
nearby cycles.  The formulation of \cite[Theorem 6.1]{BFM} is in terms
of the blow up $W$ of $\Bbb{A}^n$ on the locus where $m$ of the
coordinates coincide and a function $f$ on $W$ such that the the zero
locus $Z$ of $f$ is an open subset of the exceptional divisor. The KZ
VMHS are defined on $W\setminus Z$, and this theorem expresses the
nearby cycles $\Psi_f$ of the KZ VMHS over an open subset of $Z$ in
terms of factorisation data: Here we give a version on stalks (note
that the box product in \cite[Theorem 6.1]{BFM} turns into a tensor
product since we are taking stalks.)

General points of $Z$ parametrize reducible curves with two components
both isomorphic to $\Bbb{P}^1$.  These curves have $n$ marked points
and correspond to a $\Bbb{P}^1$ with the $m$ marked points and the
point $\infty$ on this curve is glued to a finite point on the other
$\Bbb{P}^1$ which has the remaining marked points. The stalks of the
nearby cycles sheaf are filtered, with associated graded a tensor
product of the $KZ$ MHS of these two marked curves (and the data at
the gluing point runs through all possible dominant weights). The
filtration is a direct sum if suitable differences of Casimir values
(divided by $2\kappa$) are not integers.

\subsection{Some weight computations for $\mf{sl}_m$}\label{SomeW}

Let $\nu$ be a dominant integral weight for $\mf{sl}_m$, and let $\kappa\in \mathbb{Z}_{\geq 0}$.
We examine the weights of the VMHS
$KZ_{\kappa}(\varpi_1^n,\nu^*)$. Here $\varpi_1^n$ is the $n$ tuple of weights all equal to $\varpi_1$. By abuse of notation, we  let it also denote $\varpi_1$ repeated $n$ times when it occurs inside a larger tuple as in the statement of  Lemma \ref{PieriMonodromy} below.

\begin{lemma}\label{l:D}
  Suppose $\kappa>(\rho+\varpi_1,\rho+\varpi_1)/2$. Then the weights
  of $KZ_{\kappa}(\varpi_1^n,\nu^*)$ are bounded below by
  ${\ExpM}(\varpi_1^n,\nu^*)$.
 \end{lemma}
 We will prove this later in Section \ref{DProof}. The key point is
 that under the hypothesis of the lemma,
 $KZ_{\kappa}(\varpi_1^n,\nu^*)$ injects into the cohomology of a
 smooth variety of dimension ${\ExpM}(\varpi_1^n,\nu^*)$, hence has
 weights bounded below by ${\ExpM}(\varpi_1^n,\nu^*)$ by
 \cite{hodge2}.

 \begin{remark}
For $\mf{sl}_m$, $\frac{(\rho+\varpi_1,\rho+\varpi_1)}{2} =\frac{(m^2-1)(m^2+12)}{24m}$. This has values $1$, $7/3$ and $35/8$ for $m=2,3,4$ respectively. 

\end{remark}
\begin{lemma} \label{PieriMonodromy}

  Suppose $\kappa=\ell+m$ with $\ell,a\in\mathbb{Z}_{>0}$, and assume
  that $KZ_{\kappa}((\varpi_1^a,\lambda_1,\dots,\lambda_n),\nu^*)$ has
  the property that it is mixed with weights bounded below by
  $\ExpM((\varpi_1^a,\lambda_1,\dots,\lambda_n),\nu^*)$.
  
 Then, if $\lambda_1$ is of level $\ell$, the property remains true for
  $KZ_{\kappa}((\varpi_1^{a-1},\mu,\lambda_2,\dots,\lambda_n),\nu^*)$ where $\mu$ is any term that appears in the Pieri product
  $\lambda_1\star_{\kappa} \varpi_1$.

\end{lemma}
\begin{proof}
Recall that the terms $\mu$ in the Pieri product correspond
to
\[
\lambda_1+\varpi_1-\mu=\gamma=\alpha_1+\dots+\alpha_M
\]
for some $M<m$. These terms give rise to a filtration of the nearby
cycles sheaf of
$KZ_{\kappa}((\varpi_1^a,\lambda_1,\dots,\lambda_n),\nu^*)$ as the
finite points carrying $\lambda_1$ and $\varpi_1$ coalesce. The graded
terms of this filtration are of the form
$KZ_{\kappa}((\varpi_1^{a-1},\mu,\lambda_2,\dots,\lambda_n),\nu^*)\tensor
KZ_{\kappa}((\varpi_1,\lambda_1),\mu^*)$.

It follows from \cite[Theorem 6.1]{BFM}  and Schmid's  theorem on the limiting mixed Hodge structure  \cite[Theorem 6.16 and Lemma 6.4] {schmid}, that if the monodromy residues corresponding to this filtration do not have eigenvalues that differ by a non-zero integer, then each graded term is mixed with weights $\geq$ the expected weight of
$KZ_{\kappa}((\varpi_1^a,\lambda_1,\dots,\lambda_n),\nu^*)$.

For a $\mu$ occurring in the Pieri product, the Casimir value
corresponding to the term
$KZ_{\kappa}((\varpi_1^{a-1},\mu,\lambda_2,\dots,\lambda_n),\nu^*)\tensor
KZ_{\kappa}((\varpi_1,\lambda_1),\mu^*)$ is
$(\mu,\mu+2\rho)/(2\kappa)$ and, up to a uniform shift, this is equal
to $1/\kappa$ times
\[
(\lambda_1+\varpi_1,\gamma)-\sum_{1\leq i<j\leq M}(\alpha_i,\alpha_j)
\]
with $\gamma$ as above.

Note that $\gamma$ is a positive root, $(\varpi_1,\gamma)\leq 1$,
$(\lambda_1,\gamma)\leq \ell$, and
$\sum_{1\leq i<j\leq M}(\alpha_i,\alpha_j)=-(M-1)$. Therefore the
Casimir values are, up to a uniform shift, bounded above by
$(\ell+1+M-1)/\kappa<1$.

Also note that the Pieri factors
$KZ_{\kappa}((\lambda_1,\varpi_1),\mu^*)$ are conformal blocks if
$\mu$ is of level $\ell$, hence are pure of weight
$\ExpM((\lambda_1,\varpi_1),\mu^*)$. But there may be terms that do
not have this property:
$(\mu,\theta)= (\lambda_1,\theta)+1-(\theta,\gamma)$ but
$(\theta,\gamma)=1$ or $2$ unless $\gamma=0$. Therefore there is
exactly one non-conformal-block term if $\lambda_1$ is of exact level
$\ell$, but the corresponding space
$KZ_{\kappa}((\lambda_1,\varpi_1),\mu^*)$ with
$\mu=\lambda_1+\varpi_1$ is $\mathbb{Q}(0)$. Therefore all Pieri
factors $KZ_{\kappa}((\lambda_1,\varpi_1),\mu^*)$ have the expected
weight, and hence the terms
$KZ_{\kappa}((\varpi_1^{a-1},\mu,\lambda_2,\dots,\lambda_n),\nu^*)$
have weights $\geq$ their expected weights. This completes the proof.
\end{proof}
\begin{theorem}\label{purity}
  Suppose all $\lambda_i$ are of level $\ell$ and
  $\ell+m>(\rho+\varpi_1,\rho+\varpi_1)/2$. Then
\begin{enumerate}
\item[(a)] $KZ_{\kappa}(\vec{\lambda},\nu^*)$ is mixed of weights
  bounded below by $\ExpM(\vec{\lambda},\nu^*)$.
\item[(b)] If $\nu\in N_0$ then
  $KZ_{\kappa}(\vec{\lambda},\nu^*)$ is pure of
  weight $\ExpM(\vec{\lambda},\nu^*)$. 
\end{enumerate}
\end{theorem}
\begin{proof}
  Part (a) of the theorem follows from Lemma \ref{PieriMonodromy}. For
  part (b), assuming (a), it suffices to prove that the first moment
  of the weight filtration on $KZ_{\kappa}(\vec{\lambda},\nu^*)$
  equals its rank times $\ExpM(\vec{\lambda},\nu^*)$. This follows
  from the first-moment verification, Theorem \ref{t:many}, of the
  Conjecture \ref{c:tilting} (see also Example \ref{ex:proj}).
\end{proof}

\subsection{The case of $\mf{sl}_2$}\label{reflexions} 

Using the definition of the maps $u$ and $d$  from Section \ref{s:sl2}, we have the following.
 \begin{theorem}\label{E22}
 Let $\kappa\geq 3$,  $a_1,\dots,a_n\in N_{0}\cup\{0,\dots,\kappa-2\}$, $x\geq 0$, and ${\ExpM}=(\sum a_i-x)/2$ is the expected weight of $KZ_{\kappa}(\vec{a},x)$.
 \begin{enumerate}
 \item[(a)] If $x\in N_0$, then $KZ_{\kappa}(\vec{a},x)$ is pure of weight $\ExpM$.
 \item[(b)] The weights of
 $KZ_{\kappa}(\vec{a},x)$ are in $\{{\ExpM},{\ExpM}+1\}$. Denote the ranks of its $\ExpM$ and $\ExpM+1$ weight graded quotients by $w(\vec{a},x)$ and $w^+(\vec{a},x)$, respectively.
\item[(c)]  For $x\not\in N_0$
 \begin{equation}\label{e:recursion1}
w^+(\vec{a},x)=w(\vec{a},u(x))
\end{equation}
\item[(d)]  If $a_1\in N_0$ and $x\in [0,\kappa-2]$ then $KZ_{\kappa}(\vec{a},x)$ is pure of weight $\ExpM+1$.

\end{enumerate}
\end{theorem}

\begin{corollary}\label{corollary_sl2}
Conjecture \ref{c:tilting} holds at the level of ranks for $\mf{sl}_2$.
\end{corollary}
\begin{proof} 
  This is a consequence of the following: (a) The conjecture holds at
  the level of moments by Theorem \ref{t:many}, and (b) the possible
  weights of the KZ motives are in a two element set by Theorem
  \ref{E22}.
\end{proof}

For arbitrary $\mf{g}$, and representations
$\lambda_1,\dots,\lambda_n, \nu$ at level $\ell$, Conjecture
\ref{c:tilting} and \cite{Finkelberg} imply that the rank of the
weight $\ExpM(\vec{\lambda},\nu)$ subspace of $KZ_{\kappa}(\vec{\lambda},\nu)$
equals the rank of conformal blocks for the same data at level $\ell$.
Since the space of conformal blocks sits inside the weight
$\ExpM(\vec{\lambda},\nu)$ subspace of $KZ_{\kappa}(\vec{\lambda},\nu)$ (see
Remark \ref{r:fink}) it would therefore be equal to the latter.

Since Conjecture \ref{c:tilting} holds for $\mf{sl}_2$ at the level of
ranks by Corollary \ref{corollary_sl2}, the Kazhdan--Lusztig
equivalence and \cite{Finkelberg}) imply that we get unconditionally
(using the notation $CB_{\kappa}$ for the Hodge realization of
$\mathcal{CB}_{\kappa}$ of \cite{BFM} for conformal blocks) the
following:
\begin{corollary}\label{c:fink}
  In the notation of Theorem \ref{E22}, suppose $\kappa=\ell+2$ with
  $\ell>0$ and all $a_i$ and $x$ are in $[0,\ell]$. Then Conjecture
  \ref{c:tilting} holds, and
  $\mr{W}_{\ExpM}(KZ_{\kappa}(\vec{a},x))=CB_{\kappa}(\vec{a},x)$.
\end{corollary}

 \subsubsection{Proof of Theorem \ref{E22} when $a_i\in [0,\ell+1]$ for all $i$}\label{specialcase}

 By Lemma \ref{PieriMonodromy}, the weights of
 $KZ_{\kappa}(\vec{a},x)$ are bounded below by $\ExpM$. We also know
 from Theorem \ref{purity} that it is pure of weight $\ExpM$ if
 $x\in N_0$.
 
 Let $c\in[1,\ell+1]$ and $b\in N_0$. We examine the monodromy of
 $KZ_{\kappa}((a_1,\dots,a_n,c),b)$ as the point carrying the weight
 $c$ heads to infinity. Let $\ExpM'$ be the expected weight of the
 pure motive $KZ_{\kappa}((a_1,\dots,a_n,c),b)$.  There is potentially
 nonsemisimple monodromy in this degeneration. The graded quotients
 are $KZ_{\kappa}(\vec{a},y)\tensor KZ_{\kappa}((y,c),b)$. For
 $y=b-c+2j$ with $j=0,\dots,c$, the weight of $KZ_{\kappa}((y,c),b)$
 is $j-1$ if $2j>c$ and $j$ otherwise (so that for $y>b$ the weight is
 $j-1$, and for $y\leq b$ it is $j$) and the expected weight of
 $KZ_{\kappa}(\vec{a},y)$ is $\ExpM'-j$.
 
 We will prove the statement in (a) by induction on $\ExpM$ (the case
 $\ExpM=0$ is trivial). There is nothing to prove for $x\in N_0$. Let
 $b$ be the smallest element of $N_0$ that is greater than $x$.  Let
 $c=b-x$, so that the term $y=b-c+2j$ with $j=0$ is $y=x$. The
 induction hypothesis applies for all $j>0$, and the formulas follow
 from induction and  Schmid's theorem on the limiting mixed Hodge structure \cite[Theorem 6.16 and Lemma 6.4]{schmid} that the rank of the
 $\ExpM'-m$ graded quotient of the weight filtration of the nearby
 cycles of a pure Hodge structure of weight $\ExpM'$ equals the rank
 of its $\ExpM'+m$ graded quotient for any $m>0$.  We also use the
 fact that the graded quotient corresponding to $j=0$ is actually a
 quotient of the nearby cycles (since it has the smallest Casimir
 value $c(b-c+2j)$), so that the weights that appear are necessarily
 bounded below by $\ExpM'$.
 
 \subsubsection{The general case of Theorem \ref{E22}}

  We begin with the following special case. Suppose $a_2,\dots,a_n$ are of level $\ell$ and $a_1$ is in $N_0$. 

 The motive $KZ_{\kappa}(\vec{a}, 0)$ degenerates into $KZ_{\kappa}((a_2,\dots,a_n),a_1)\tensor KZ_{\kappa}((a_1,a_1),0)$. The monodromy is a scalar in this degeneration and hence weights are preserved.
 Suppose ${\ExpM}$ is the expected weight of $KZ_{\kappa}(\vec{a},0)$. Then by the case of the theorem already proved,  $KZ_{\kappa}((a_2,\dots,a_n),a_1)$ is pure of the expected weight ${\ExpM}-a_1$. Also by an easy application of \eqref{e:A1-G-abc}, $KZ_{\kappa}((a_1,a_1), 0)$ is pure of weight $a_1+1$. Therefore $KZ_{\kappa}(\vec{a},0)$ is pure of weight ${\ExpM}+1$. Using a similar switch, we also get that $KZ_{\kappa}(\vec{a},x)$ is pure of the expected weight plus one whenever $x$ is of level $\ell$. 

We prove the following equivalent form of item (c) in the theorem
\begin{enumerate}
\item[(c')] If $x>\kappa-1$, and $x\not\in N_0$ then
  $w(\vec{a},x)=w^+(\vec{a},x')$ where $x'=d(x)=u^{-1}(x)<x$
 \end{enumerate}
 
 We induct on the number $k$ of $a_i$ that are in $N_0$. If $k=1$, suppose $a_1\in N_0$. We know the statement for $b$ of level $\ell$ by our earlier discussion (and item (c') does not apply here). We then induct on $b$.
The induction step is similar to the case of $k-1$ and, using a swap as above, gives the base step in the next layer of induction.

For the induction step, we first consider the case $x\in N_0$. Consider the degeneration of $KZ_{\kappa}((a_1,\dots,a_n, 1), x-1)$ as the point carrying the weight $1$ goes to infinity, or as the remaining variables coalesce at a finite point. This is a semisimple local monodromy situation (since the difference of the Casimirs $c(x)-c(x-2)=2x$ is not a multiple of $2\kappa$)  for a variation of a mixed motive with weights $\{{\ExpM}+1,{\ExpM}+2\}$.
One of the two terms is $KZ_{\kappa}(\vec{a},x)$ tensored with $KZ_{\kappa}(x,1,x-1)$ (of weight $2$).  The other term is $KZ_{\kappa}(\vec{a},x-2)$ which is mixed of weight $\{{\ExpM}+1,{\ExpM}+2\}$.

Thus $KZ_{\kappa}(\vec{a},x)$ has weights in $\{{\ExpM}-1,{\ExpM}\}$. Since its first moment equals its rank times $\ExpM$, we conclude that $KZ_{\kappa}(\vec{a}, x)$
is pure of weight ${\ExpM}$.

We next consider the case $x\not\in N_0$. Let $b$ be the smallest element of $N_0$ that is less than $x$. Let $c=x-b$. 
We consider a degeneration of $KZ_{\kappa}((a_1,\dots,a_n,c),b)$ as the point carrying the weight $c$ heads to infinity as in Section \ref{specialcase}. The term $j=c$ is now the term of interest.  We also use the fact that the graded quotient corresponding to $j=c$ is actually a subobject of the nearby cycles (since it has the largest Casimir value $c(b-c+2j)$),
 so that the weights that appear are necessarily $\leq \ExpM'$. The same argument now gives (a), (b), and (c) because other elements in the degeneration have a smaller weight at infinity. Part (d) follows from the others by a switch with the element $a_1$ of $N_0$ as before.

This completes the proof of Theorem \ref{E22}.

\subsection{Proof of Lemma \ref{l:D}}\label{DProof}

We first recall the construction of the VMHS $KZ_{\kappa}(\vec{\lambda},\nu^*)$
for arbitrary $\mf{g}$.
\subsubsection{The Schechtman--Varchenko arrangements and master
  functions}\label{s:sv}
Fix a split simple Lie algebra $\mf{g}$, let
$\vv{\lambda} \in (P^+)^n$, $\nu \in P^+$, and assume
 that $\sum\lambda_i-\nu$ is a nonnegative integral sum of simple roots (possibly
repeated).

We then write
\begin{equation*}%
  \sum\lambda_i-\nu=\sum_{\alpha\in\Delta}n_{\alpha} \alpha  =
  \sum_{a=1}^{\ExpM}\beta(a) .
\end{equation*}
Here $\ExpM=\sum_{\alpha\in \Delta}n_{\alpha}$, and we have made a choice
of $\beta:[\ExpM]=\{1,\dots,\ExpM\} \to \Delta$ so that the second equality
holds.

Fix $\vec{z}=(z_1,\dots,z_n)$ in $\mc{C}_n$ and let
$W=\A^{\ExpM}$. The coordinates of $W$ will be denoted by
$t_1,\dots,t_{\ExpM}$. We regard the variable $t_b$ to be coloured by
the simple root $\beta(b)$. Let $\kappa \in \Bbb{C}^{\times}$.

Consider the weighted hyperplane arrangement
$\ms{A} = (W,\mathscr{C},a)$ given by the following collection $\mathscr{C}$ of hyperplanes and their associated weights $a$:
\begin{enumerate}
\item For $i\in [1,n]$ and $b\in [1,{\ExpM}]$, the hyperplane
  $t_b-z_i=0$, with weight $ {(\lambda_i,\beta(b))}/{\kappa}$.
  \smallskip
\item For $b,c\in [1,{\ExpM}]$ with $b<c$,  the hyperplane $t_b-t_c=0$, with weight $-{(\beta(b),\beta(c))}/{\kappa}$.
\end{enumerate} 

Let $U$ be the complement of this hyperplane arrangement in $W$.
Consider the multivalued master function
\begin{equation} \label{e:master} %
  \ms{R} = \prod_{1 \leq i < j \leq n} (z_i - z_j)^{-\frac{(\lambda_i,
      \lambda_j)}{\kappa}} \prod_{b=1}^{\ExpM} \prod_{j=1}^n (t_b -
  z_j)^{\frac{(\lambda_j, \beta(b))}{\kappa}} \prod_{1 \leq b < c \leq
    {\ExpM}} (t_b - t_c)^{-\frac{(\beta(b), \beta(c))}{\kappa}}
\end{equation}
The weighted hyperplane arrangement $(W,\mathscr{C},a)$ determines the form
$\eta=d\ms{R}/\ms{R}$ in $\Omega^1(U)$.  The differentials $d(z_i-z_j)$ are zero
because the $z_j$ are fixed, but we will consider this multivalued
function later in contexts where the $z$'s are allowed to vary. Thus
\begin{equation}\label{deta}
  \eta=\frac{1}{\kappa}\biggl(\sum_{b=1}^{\ExpM}\sum_{j=1}^n (\lambda_j,\beta(b))\frac{d(t_b-z_j)}{t_b-z_j} - \sum_{1 \leq b < c \leq {\ExpM}}(\beta(b),\beta(c))\frac{d(t_b-t_c)}{t_b-t_c}\biggr) .
\end{equation}

Let $\Sigma = \Sigma_{\beta}$ be the subgroup of the symmetric group
$S_{\ExpM}$ on $[{\ExpM}]=\{1,\dots,{\ExpM}\}$ consisting of permutations $\sigma$ that
preserve colour, i.e. $\beta(\sigma(b))=\beta(b)$ for all $b\in[{\ExpM}]$;
$\Sigma_{\beta}$ acts on $U$, the complex
$(A^{\starr}(U),\eta\wedge)$, as well as the local system $\ml(\eta)$
(over the action on $U$).

\subsubsection{}\label{s:top}

Let $P$ be any smooth projective $\Sigma_{\beta}$-equivariant
compactification of $U$, with $P \setminus U=\cup_{\alpha} E_{\alpha}$
a simple normal crossings divisor.  Let
$V=P \setminus \cup'_{\alpha}E_{\alpha}$, where the union is over those $\alpha$ for which
$a_{\alpha}=\op{Res}_{E_{\alpha}}\eta$ is not a strictly positive integer. Let
$j:U\to V$ be the inclusion.

Let
$V'=P \setminus \cup'_{\alpha}E_{\alpha}$, where the union is over those $\alpha$ for which
$a_{\alpha}=\op{Res}_{E_{\alpha}}\eta$ is not a nonnegative integer,
and let $q:U\to V'$ be the inclusion.  Let $\chi$ be the sign
character on $S_{\ExpM}$ restricted to $\Sigma_{\beta}$.

By \cite{SV,L2,BBM} (also see \cite{BF-survey}), we have the following.
\begin{proposition} \label{p:kz} The space  $\mb{A}(\vv{\lambda},\nu^*)^*$ is identified, compatibly with the KZ
  connection, with the image of the map
  \begin{equation}\label{Sho}
    H^{\ExpM}(V',q_{!}\mathcal{L}(\eta))^{\chi}\to H^{{\ExpM}}(V,j_{!}\mathcal{L}(\eta))^{\chi}.
  \end{equation}
  \end{proposition}
  The image of \eqref{Sho} is $KZ_{\kappa}(\vec{\lambda},\nu^*)$
  (considered as a VMHS) \cite{L2,BBM,BF-survey}.

\begin{remark}\label{signkappa}
  Let $\kappa$ be rational.  As explained in Remark 4.24 in
  \cite{BFM}, $KZ_{-\kappa}(\vec{\lambda},\nu^*)$ is isomorphic to the
  MHS $KZ_{\kappa}(\vec{\lambda},\nu^*)^*(-\ExpM)$ (a Tate twist).
  This is seen as follows: The two MHS
  $KZ_{\pm\kappa}(\vec{\lambda},\nu^*)$ can be computed using the same
  spaces. The local systems over $U$ are dual, and residues in the two
  situations are negatives of each other. Now, Proposition \ref{p:kz}
  gives us the isomorphism.

We also note that the sign in the KZ equation \eqref{e:KZ} is the one
in \cite[p.~140]{Ueno} (the equation given there is on the
dual). This differs in sign from the form of the KZ equation given in
\cite[Equation 7.2.8.1]{SV} and the corresponding master function in
\cite[Equation 7.1.2]{SV} is the inverse of the master function
$\ms{R}$ above.
\end{remark}

\subsubsection{Residue computations}\label{residue}

We first recall the expressions for the residues of the differential $\eta$ at various strata given in \cite{L2} (the computation at infinity is an easy generalization).
\begin{enumerate}

\item The residue as $t_1,\dots,t_s$ get together at a finite point, with $s\geq 1$. Let $\gamma=\sum_{i=1}^s\beta(i)$. The residue is
\begin{equation}\label{eq1}
-\frac{1}{2\kappa}\bigl((\rho-\gamma,\rho-\gamma)-(\rho,\rho)\bigr)
\end{equation}
\item $t_1,\dots,t_s$ get together at the finite point $z_1$. Let $\gamma=\sum_{i=1}^s\beta(i)$.

The residue is
\begin{equation}\label{eq2}
-\frac{1}{2\kappa}\bigl((\rho+\lambda-\gamma,\rho+\lambda-\gamma)-(\rho+\lambda,\rho+\lambda)\bigr)
\end{equation}

\item $t_1,\dots,t_s$ get together at $\infty$. Let $\gamma=\sum_{i=1}^s\beta(i)$
The residue is
\begin{equation}\label{eq3}
-\frac{1}{2\kappa}\bigl((\rho+\nu+\gamma,\rho+\nu+\gamma)-(\rho+\nu,\rho+\nu)\bigr)
\end{equation}
This equals
$$-\frac{1}{2\kappa}\bigl((2\rho,\gamma) + (\gamma,\gamma)+ (2\nu,\gamma)\bigr)
= -\frac{1}{2\kappa}\bigl((\gamma+2(\nu+\rho),\gamma)\bigr)<0$$

Note that for $\lambda-\gamma=\nu$ (when dominant), the expression \eqref{eq2} is the negative of \eqref{eq3}.

\end{enumerate}

\subsubsection{Conclusion of the proof of Lemma \ref{l:D}}

 Let $\nu$ be a dominant integral weight for $\mf{sl}_m$, and let $\kappa\in \mathbb{Z}_{\geq 0}$.
We examine the weights of the VMHS $KZ_{\kappa}(\varpi_1^n,\nu^*)$. Let ${\ExpM}={\ExpM}(\varpi_1^n,\nu^*)$, and let $U$, $V$, and $V'$ be as in Section \ref{s:top}.

We will show that none of the residues at the dense strata is a
strictly positive integer, so that $V=U$. Hence the image of the map
\eqref{Sho} is a subspace of $H^{{\ExpM}}(U,\mathcal{L}(\eta))^{\chi}$
and is mixed of weights bounded below by ${\ExpM}$. This will prove
the lemma.  Using the residue formulas in Section \ref{residue} (in
the formulas below, $\gamma=\sum_{i=1}^s\beta(i)$),
\begin{enumerate}
\item The residue as $t_1=\dots=t_s$ is
  $1/(2\kappa)$ times $(\rho,\rho)-(\rho-\gamma,\rho-\gamma)$.  Thus, as long as $(\rho,\rho)/2\kappa<1$, we cannot get positive integer values. Note also that $(\rho,\rho)<(\rho+\varpi_1,\rho+\varpi_1)$.
\item The residue as $t_1=\dots=t_s=z_1$
  is $1/(2\kappa)$ times
  $(\rho+\varpi_1,\rho+\varpi_1)-(\rho+\varpi_1-\gamma,\rho
  +\varpi_1-\gamma)$. Thus, as long as
  $(\rho+\varpi_1,\rho+\varpi_1)/2\kappa<1$, we cannot get positive integer residues.
\item The residues at infinity are always negative.
\end{enumerate}

\begin{remark}
  Note that the local system $\mathcal{L}(\eta)$ on $U$ extends to a
  motivic local system $\mathcal{L}'$ on $V$ . Now
  $H^{{\ExpM}}(V,\mathcal{L}')^{\chi}$ has weights bounded below by
  $\ExpM$.  Are there more general conditions which ensure that the
  map
  $KZ_{\kappa}(\vec{\lambda},\nu^*)\to
  H^{{\ExpM}}(V,\mathcal{L}')^{\chi}$ is injective?
\end{remark}

\appendix

\section{}\label{a:A}

This appendix is devoted to the proof of Lemma \ref{l:u}.

\smallskip

Fix integers $a,b \geq 0$ and $\kappa \geq 2$. Let
\[
  E(a,b) = \{(a'+b')\kappa, (a'+b' -2)\kappa,(a'+b' -4)\kappa,\dots, |a' -
  b'|\kappa\} 
\]
and let
\[
  E'(a,b) = \{a'' + b'', a'' + b'' -2, a'' + b'' -4,\dots, |a'' -
  b''|\} .
\]
We now split into two cases.
\begin{itemize}
\item If $a'' + b'' < \kappa$, let
\[
  E_1(a,b) = E'(a,b)
\]
\item If $a'' + b'' \geq \kappa$, let
  \[
    E_1(a,b) = \{2\kappa -a'' -b'' -4, 2\kappa -a'' - b'' -6, 2\kappa
    -a'' - b'' -8,\dots, |a'' -b''| \} .
  \]
\end{itemize}
Let
 \[
  S(a,b) = \{d + e: d \in E(a,b)\mbox{ and } e \in E_1(a,b)) \} .
\]
It follows from the computations in Section \ref{s:lab}, that $S(a,b)$
is the set of highest weights of the simple summands of $\lq (a) \otimes
\lq (b)$. Thus, Lemma \ref{l:u} is equivalent to the following
\begin{lemma}\label{l:u2}
  If $c \in S(a,b)$ and $a,b,c \notin N_0 $ then
  $u(c) \notin S(u(a),u(b))$.
\end{lemma}

\begin{proof}
Write, as usual, $ a=a'\kappa+a''$, $b=b'\kappa+b''$,
$0\le a'',b''\le \kappa-1$.  The assumption on $a,b$ implies that
$ 0\le a'',b''\le \kappa-2$.  Moreover,
$ u(a)=(a'+1)\kappa+(\kappa-2-a'')$,
$ u(b)=(b'+1)\kappa+(\kappa-2-b'')$, so the $\kappa$-quotients of
$u(a)$ and $u(b)$ are $a'+1$ and $b'+1$, respectively.

We also note that $e\in E_1(a,b)$ satisfies $0\le e\le \kappa-1$.
This is clear if $a'' + b'' < \kappa$. If $a'' + b'' \geq \kappa$,
then the largest element of $E_1(a,b)$ is $l = 2\kappa-a''-b''-4$.
Because $a'',b''\le \kappa-2$, we have $a''+b''\le 2\kappa-4$, and
therefore $l\ge 0$ and $a''+b''\ge\kappa$ gives
$l \le \kappa-4<\kappa-1$.

The definition of the set $E(a,b)$ implies:
\begin{itemize}
\item If $d\in E(a,b)$ and $d=m\kappa$, then
  $m\equiv a'+b' \mod 2$.
\item If $\widetilde d\in E(u(a),u(b))$ and
  $\widetilde d=n\kappa$, then
  $ n\equiv a'+b'\mod 2$.
\end{itemize}

After these preliminary observations, we prove the desired
disjointness. Let $c\in S(a,b)$. By definition, there exist
$d\in E(a,b)$ and $e\in E_1(a,b)$ such that $c=d+e$. Write
$d=m\kappa$, so $c = m\kappa + e$. The above observations imply that
$c' = m$, $c'' = e$ and $ m\equiv a'+b'\mod 2$.

We shall prove that $ u(c)\notin S(u(a),u(b))$. Since $c \notin N_0$,
$u(c)=(m+1)\kappa+(\kappa-2-e)$.  Since $0\le \kappa-2-e\le \kappa-2$,
this is the canonical $\kappa$-expansion, i.e., $u(c)' = m+1$ and
$u(c)'' = (\kappa -2-e)$. Note that $ m+1\not\equiv a'+b'\mod 2$.

On the other hand, consider any element of $S(u(a),u(b))$. It has the
form $ \widetilde d+\widetilde e=n\kappa+\widetilde e$, with
$\widetilde d\in E(u(a),u(b))$, $\widetilde e\in E_1(u(a),u(b))$.  Now
$u(a), u(b) \notin N_0$, so $ 0\le \widetilde e\le \kappa-1$.  Thus
$n\kappa+\widetilde e$ is its canonical $\kappa$-expansion, so
$ n\equiv a'+b'\mod 2$. Thus, $m+1 \neq n$, hence
$u(c) \notin S(u(a),u(b))$.
\end{proof}

\bibliographystyle{siam}
\bibliography{biblio.bib}

\end{document}